\documentclass{article}
\usepackage[utf8]{inputenc}
\usepackage{amssymb}
\usepackage[T1]{fontenc}
\usepackage{amsmath}
\usepackage{amsthm}
\usepackage{xcolor}
\usepackage{graphicx}
\usepackage{tikz}
\usetikzlibrary{arrows.meta}
\usepackage{subcaption}
\usepackage{enumerate}
\usepackage{hyperref} 
\usepackage{epic,eepic}

\newtheorem{lemma}{Lemma}
\newtheorem{fact}{Fact}
\newtheorem{definition}[lemma]{Definition}
\newtheorem{remark}[lemma]{Remark}
\newtheorem*{esc_cond}{Assumption (E)}
\newtheorem{theorem}[lemma]{Theorem}
\newtheorem{proposition}[lemma]{Proposition}
\newtheorem{corollary}[lemma]{Corollary}
\def\R{\mathbb{R}}

\def\sign{\mathop{\text{sign}}}

\newcommand{\Xup}{X_{\uparrow}}
\newcommand{\Xdown}{X_{\downarrow}}

\title{A Mathematical Characterization of Neural Activation Induced by Temporal Interference Stimulation}
\author{Esteban Paduro\footnotemark[1], \, Antoine Chaillet\footnotemark[2],   \, Mario Sigalotti\footnotemark[3]}
\date{\today}
\begin{document}
\maketitle

\footnotetext[1]{Instituto de Ingenier\'ia Matem\'atica y Computacional, Facultad de Matem\'aticas, Pontificia Universidad Cat\'olica de Chile, 
Santiago, Chile. Corresponding author. Email: {\tt esteban.paduro@uc.cl}
}
\footnotetext[2]{Universit\'e Paris Saclay, CNRS, CentraleSup\'elec, Laboratoire des signaux et syst\`emes, 91190 Gif-sur-Yvette, France}
\footnotetext[3]{
Sorbonne Universit\'e, Universit\'e Paris Cit\'e, CNRS, Inria, Laboratoire Jacques-Louis Lions, Paris, France
}

\begin{abstract}
Temporal Interference Stimulation (TIS) is a non-invasive neuromodulation technique in which two high-frequency sinusoidal currents with slightly different frequencies generate a low-frequency envelope that can activate deep neural structures. This study investigates the conditions under which TIS elicits action potentials in a single neuron modeled by the FitzHugh--Nagumo system. This research integrates phase-plane analysis and geometric singular perturbation to develop a mathematical framework for analyzing TIS. By combining a mathematical analysis of differential equations with computer simulations, the study elucidates how the amplitudes and beat frequency jointly determine whether the neuron remains quiescent, exhibits only transient responses, or undergoes persistent (tonic) firing.
\end{abstract}

\vspace{0.2 cm}


\vspace{0.2 cm} {\bf AMS subject classifications:} 37N25, 92-10.

\section{Introduction}

Temporal Interference Stimulation (TIS) is a non-invasive brain stimulation technique in which two or more electrodes deliver sinusoidal currents at kilohertz frequencies $\omega_1$ and $\omega_2$ that differ only slightly, so that their superposition produces a low-frequency envelope, oscillating at the beat frequency $\eta=\omega_2-\omega_1\ll\omega_1$. The modulation depth of this low-frequency resultant depends on the relative amplitudes and orientations of the two fields and can be made maximal in a target region \cite{Ward2009, Grossman2017, Mirzakhalili2020}. Neurons do not respond directly to each kilohertz carrier, but can be recruited by the slowly modulated envelope \cite{Grossman2017}. This opens the way to a selective and non-invasive activation of deep brain regions \cite{xu_precision_2025}, with promising therapeutic applications for neurological disorders such as Parkinson's disease \cite{yang_transcranial_2025}, epilepsy \cite{missey2025non}, and memory deficits \cite{violante2023non}.
 
The effectiveness of TIS depends critically on the choice of the stimulation parameters: the location of the sources and the electrical parameters (amplitude, frequency, and shape of the delivered signals). As reviewed in \cite{xu_precision_2025}, this question has been addressed through three complementary approaches: in vivo experimentation \cite{Grossman2017, Missey2021, opancarSameBiophysicalMechanism2025}, physics-based numerical simulations \cite{Grossman2017, Mirzakhalili2020, wangResponsesModelCortical2023}, and theoretical analysis \cite{Karimi2019, Mirzakhalili2020, karimiNeuromodulationEffectTemporal2024, cerpa2023, plovieNonlinearitiesTimescalesNeural2025, cerpaApproximationStabilityResults2025}. These works have examined the influence of the positioning of the sources \cite{Grossman2017, Karimi2019, xu_precision_2025, karimiNeuromodulationEffectTemporal2024}, of the axons orientation \cite{Missey2021, wangResponsesModelCortical2023}, of conduction block away from the target region \cite{Mirzakhalili2020, wangResponsesModelCortical2023}, of the source frequencies \cite{Mirzakhalili2020, Cao2020b, opancarSameBiophysicalMechanism2025}, and of the neuron model's nonlinearities \cite{plovieNonlinearitiesTimescalesNeural2025}. In particular, it has been observed that the beat frequency able to elicit a neuronal response is strongly linked to the time constants of the ionic channels \cite{Mirzakhalili2020, plovieNonlinearitiesTimescalesNeural2025}. Most of these works rely on simulations of numerical models. To the best of our knowledge, no rigorous characterization is available of the joint conditions on amplitudes and beat frequency under which a neuron model exhibits sustained activity under TIS.

In this paper, we address this question at the level of a single neuron. We consider a fixed position in space and a neuron with a given orientation, so that the distance to the sources, the axon orientation, and the attenuation of the medium are all lumped into two amplitude parameters denoted as $A$ and $B$. The neuron is described by the FitzHugh--Nagumo (FHN) model \cite{FitzHugh1961,nagumoActivePulseTransmission1962}, 
which offers a good compromise between mathematical tractability and biological relevance while retaining the essential features of excitability: a fast membrane potential $v$, a slow recovery variable $w$ evolving at a rate $\varepsilon\ll 1$, and a cubic nonlinearity. Following \cite{cerpa2023}, we average the system over the fast carrier oscillations. A key feature of the resulting model is that the stimulation no longer acts as an additive input, but rather modulates a linear coefficient of the membrane equation. In other words, TIS acts on the neuron as a slow periodic deformation of the cubic $v$-nullcline, with a constant offset determined by $A^2+B^2$ and an oscillating part of amplitude $AB$ at the beat frequency. The problem is therefore that of an excitable slow--fast system subject to a slow parametric modulation, in which three time scales coexist: that of $v$, that of $w$ (of order $1/\varepsilon$), and that of the envelope (of order $1/\eta$).
 
Since most neuroscience applications require durable neuronal activity rather than an isolated response, our objective is to determine conditions on the parameters $A, B$ and on the beat frequency $\eta$ under which the neuron exhibits tonic spiking, meaning repeated action potentials. We show as preliminary observations  (Lemma~\ref{lemma_general_properties}) that this problem is not trivial. Under suitable stability assumptions on the corresponding frozen system, averaging arguments show that solutions remain close to the equilibrium of the averaged system and emit no action potential if the envelope varies too fast. On the other hand, if the envelope varies too slowly, solutions remain close to the equilibria of the  frozen system (quasi-static regime), and again no action potential is emitted. Tonic spiking must therefore be sought in intermediate regimes.
 
Our main contributions can be summarized as follows. We first provide a universal obstruction to tonic spiking, by identifying explicit geometric conditions (expressed in terms of the position of the equilibria relative to the local minima of the $v$-nullclines) under which, for $\varepsilon$ small enough, no solution exhibits tonic spiking, regardless of the applied input envelope (Theorem~\ref{thm_geometric_condition}). This condition can be written in closed form and implies in particular that, in the excitable regime considered here, sufficiently small amplitudes cannot trigger sustained activity. We also show that, when this necessary condition is fulfilled, a square-wave envelope $\sign(\cos(\eta t))$ induces tonic spiking of every solution, provided that $\varepsilon$ is small and $\eta$ is small enough with respect to $\varepsilon$ (Theorem~\ref{thm_activation_piecewise_constant}). The discontinuity of the envelope is what allows the system to escape the quasi-static regime mentioned above.
 
For the smooth envelope $\cos(\eta t)$ actually produced by TIS, we consider the regime in which the beat frequency scales with the slow time scale of the neuron ($\eta$ proportional to $\varepsilon$) and we analyze the singular limit $\varepsilon\to 0$. The resulting differential-algebraic system has regular solutions that are concatenations of integral curves of two explicit planar vector fields, $\Xup$ and $\Xdown$, associated with the increasing and decreasing phases of the envelope. We derive a sufficient condition for the absence of tonic spiking that requires the computation of a single integral curve (Theorem~\ref{theorem_results_singular_system}), and which holds automatically when the proportional factor between $\eta$ and $\varepsilon$ is small (Corollary~\ref{cor:noescaping}). We also identify the local mechanism that generates action potentials: when the minimum of the $v$-nullcline rises faster than the slow variable can follow, the trajectory leaves the stable branch of the nullcline and an action potential is emitted (Lemma~\ref{lem:escaping}). Based on this mechanism, we formulate a geometric condition (Assumption~(E)) under which a heuristic argument leads us to conjecture that, for $\varepsilon$ small, tonic spiking occurs for almost every initial condition. We stress that this last statement is not proven: the obstacle lies in the lack of uniformity of the estimates near exceptional trajectories, whose existence is unavoidable since periodic envelopes give rise to periodic solutions.
 
We finally assess these predictions through systematic simulations over the parameters $(A,B,\kappa,\varepsilon)$, both for the averaged model and for the original two-frequency FHN model, which reproduce the main activation patterns predicted by the analysis, except in the vicinity of the transition boundaries.

From a practical viewpoint, these results suggest that the stimulation amplitudes must be large enough to violate an explicit geometric condition, and that, for the smooth envelopes delivered by TIS, the beat frequency should be tuned to the time constant of the slow recovery variable. This is consistent with the dependence between beat frequency and ionic time constants reported in \cite{Mirzakhalili2020, plovieNonlinearitiesTimescalesNeural2025}, for which our analysis provides a theoretical ground.

The paper is organized as follows. Section~\ref{sec-2} introduces the model, the notion of tonic spiking, and the stability properties of the frozen systems. Section~\ref{sec-3} establishes necessary conditions for tonic spiking, first for general envelopes and then for the TIS envelope in the singular limit. Section~\ref{sec:four} provides sufficient conditions for tonic spiking under piecewise constant envelopes. Section~\ref{section_escaping_condition} analyzes the smooth TIS envelope in the regime $\eta\propto \varepsilon$ and presents the heuristic argument for tonic spiking. Numerical experiments are reported in Section~\ref{sec:numerical_experiments}, and Section~\ref{sec:discussion} discusses perspectives.

\section{Setup and preliminary results}\label{sec-2}

We consider a FitzHugh–Nagumo (FHN) model of the targeted neuron.  
This well-known model is made of two differential equations:
\begin{equation}\label{full_FHN}
\left\{
\begin{aligned}
\dot{V} &= V- \frac{V^3}{3} -W +I(t),\\
\dot{W} &= \varepsilon(V-\gamma W +\beta), 
\end{aligned}
\right.
\end{equation}
where $\varepsilon$, $\beta$, and $\gamma$ are positive parameters, $V$ is the neuron's membrane potential, and $W$ is a recovery variable (accounting for both sodium-channel inactivation and potassium-channel activation). The input $I$ is the weighted sum of the sine waves arising from the two interfering sources:
\begin{equation}\label{interferential_input}
I(t) = A \omega_1 \cos(\omega_1 t) + B \omega_2 \cos(\omega_2 t),
\end{equation}
where $A$ and $B$ denote positive constants. 
Following \cite{cerpa2023}, applying a temporal averaging over the fast carrier frequency yields a more tractable system in which the effects of the beat frequency are explicit:
\begin{equation}\label{FHN_intro}
\left\{
\begin{aligned}
    \dot{v} &= v\left(1-\frac{A^2}{2}-\frac{B^2}{2}- AB \cos(\eta t)\right) - \frac{v^3}{3}- w,\\
    \dot{w} &= \varepsilon (v-\gamma w + \beta).
\end{aligned}
\right.
\end{equation}

\noindent The model \eqref{FHN_intro} falls in the more general class of nonautonomous systems
\begin{equation}\label{general_switching_system}
\left\{
\begin{aligned}
    \dot{v} &= v\left(1-\frac{A^2}{2}-\frac{B^2}{2}- ABf(t) \right) - \frac{v^3}{3}- w,\\
    \dot{w} &= \varepsilon (v-\gamma w + \beta),
\end{aligned}
\right.
\end{equation}
where $f\in L^\infty([0,+\infty),[-1,1])$ and $A,B,\varepsilon,\gamma,\beta>0$. 
This system coincides with \eqref{FHN_intro} when $f(t)=\cos(\eta t)$, but it is useful to consider more general signals for two main reasons. On the one hand, we can approximate a trigonometric function by piecewise constant functions for which the mathematical analysis is simpler. On the other hand, some geometric reasoning is independent of the specific signal $f$ and can be better understood in a more general framework.  

At rest, the cell membrane of a neuron is polarized, resulting in a negative membrane voltage. An action potential is characterized by a rapid depolarization followed by a repolarization, during which the membrane potential $V$ reaches positive values. Since $v$ can be seen as the membrane potential $V$ from which direct effects of high-frequency inputs are removed \cite[Figure 3]{cerpa2023}, requesting positive values for $v$ essentially corresponds to having a positive value of the low frequencies in $V$. Accordingly, we adopt the following convention.

\begin{definition}[Action potential, tonic spiking]\label{def-spike}
Given $t_1\geq t_0\geq 0$, a solution of \eqref{general_switching_system}  is said to undergo an \emph{action potential} in the time interval $[t_0, t_1]$ if there exists $t\in [t_0,t_1]$ such that $v(t) \geq 0$. 
A solution of \eqref{general_switching_system} is said to undergo
\emph{tonic spiking} if it exhibits persistent action potentials, that is, for every $t_0\geq 0$ there exists $t\geq t_0$ such that $[t_0,t]$ contains an action potential.
\end{definition}

Most neuroscience applications require not only eliciting an action potential in the targeted neuron but also inducing durable neuronal activity. Accordingly, this study seeks to determine conditions on the parameters $A, B,\varepsilon, \gamma,\beta$, and on the signal $f$ such that the solutions of \eqref{general_switching_system} exhibit tonic spiking.

\subsection{An attractive invariant set}

The following observation shows that, no matter the precise choice of the input signal $f$, it is always possible to design a bounded area from which solutions cannot escape and that attracts all possible solutions. In turn, this result shows that, despite the cubic nonlinearity in the vector field of \eqref{general_switching_system}, all solutions are well defined (and bounded) at all positive times.

\begin{fact}[Forward invariant and exponentially attractive set]\label{fact1}
For any $f \in L^\infty([0,+\infty),$ $[-1,1])$ and any $L>0$ large enough, the set
 $[-L,L]\times [-S,S]$ with $S:=\frac{L+\beta}{\gamma}+1$ is forward invariant and globally exponentially attractive for  system \eqref{general_switching_system}.
\end{fact}

\begin{proof}
 At every point of the segment $[-L,L]\times \{S\}$, any solution of \eqref{general_switching_system} satisfies $\dot w=v-\gamma S+\beta\le 0$. Similarly, we have that $\dot w\geq 0$ over the segment $[-L,L]\times \{-S\}$. These facts show that solutions cannot escape the rectangle $[-L, L]\times [-S, S]$ from its top and bottom edges. Now consider the points in $\{L\}\times [-S,S]$. For such points, we have $\dot v\leq L(1-A^2/2-B^2/2-f(t)AB)-L^3/3+S$ $= L(1-A^2/2-B^2/2-f(t)AB)-L^3/3+\frac{L+\beta}{\gamma}+1$. By picking $L$ large enough, we see that $\dot v\leq 0$. Similarly, $\dot v\geq 0$ at each point of  $\{-L\}\times [-S, S]$ and solutions cannot escape the rectangle by its lateral edges either. Thus, the set $[-L,L]\times [-S,S]$ is positively invariant.

 To show it is globally exponentially attractive, let us first notice that there exist $\rho,\alpha>0$ such that the directional derivative of the Lyapunov function $Q(v,w):=v^2+w^2$ along the trajectories of \eqref{general_switching_system} satisfies
 \begin{align*}
     |(v,w)|>\rho\quad \Rightarrow \quad \frac{d}{dt}Q(v,w)\leq -\alpha Q(v,w),
 \end{align*}
where $|\cdot|$ denotes the standard Euclidean norm in $\R^2$.
This easily follows from the fact that, given $c_1,c_2,c_3\in \R$, there exists $\alpha>0$ such that $v^4+w^2+c_1 vw +c_2 v^2+c_3 w>\alpha (v^2+w^2)$ for $|(v,w)|$ large enough.  Hence, every trajectory of  \eqref{general_switching_system} reaches the ball of radius $\rho$ centered at the origin exponentially fast. Choosing $L$ large enough so that $[-L, L]\times [-S, S]$ contains such a ball gives us the result.
\end{proof}

\subsection{The frozen system}\label{key_objects}

Let us first examine the phase portrait of \eqref{general_switching_system} when the input term $f$ is constant. To that aim, fix $c\in [-1,1]$ and consider the ``frozen'' version of \eqref{general_switching_system} in which $f(\cdot)\equiv c$, namely:
 \begin{align}\label{eq-1}
     \begin{pmatrix}
         \dot v \\ \dot w
     \end{pmatrix}=G_c(v,w),
 \end{align}
where $G_c(v,w):=(g_1(v,w,c),\varepsilon g_2(v,w))^\top$ with
\begin{align}
    g_1(v,w,c) &:= v\left(1-\frac{A^2}{2}-\frac{B^2}{2}- c AB\right) - \frac{v^3}{3}- w, \label{defi_F1} \\  
    g_2(v,w) &:= v-\gamma w + \beta. \label{defi_F2}
\end{align}
The $v$-nullcline of \eqref{eq-1} is
given by all the state values for which $g_1$ vanishes. Defining 
\begin{align}\label{eq-rc}
    r(c):=1-\frac{A^2}{2}-\frac{B^2}{2}- c AB,
\end{align}
this nullcline reads:
\begin{equation*}
  C_c:=  \{(v,w)\in\mathbb R^2\mid g_1(v,w,c)=0\} = \left\{(v,w)\in\mathbb R^2\mid w = v\left(r(c)-\frac{v^2}{3}\right)\right\}.
\end{equation*} 

In the $(v,w)$ plane, this nullcline is thus a cubic curve that crosses the origin and whose number of intersections with the horizontal axis depends on the sign of $r(c)$. 
Because of the monotonicity of $c\mapsto g_1(v,w,c)$ (which is increasing for $v<0$ and decreasing for $v>0$), each cubic $C_c$, $c\in [-1,1]$, is contained in the region of the state space which is below $C_1$ and above $C_{-1}$ for $v<0$ and below $C_{-1}$ and above $C_{1}$ for $v>0$ (see Figure~\ref{fig_cubics_monotonicity}).

\begin{fact}\label{fact_three_intersections}
If $A+B<\sqrt{2}$ then, given any $c\in[-1,1]$, the $v$-nullcline $C_c$ has three intersections with the horizontal axis.
\end{fact}
\begin{proof}
We see from its definition that $C_c$ has a 
unique intersection at the origin if $r(c)\leq 0$, and three intersections otherwise. Our assumption implies $r(1)>0$, and since $r$ is monotone decreasing, this guarantees that, for every $c\in [-1,1]$, the $v$-nullcline $C_c$ has three distinct intersections with the horizontal axis: $\left(\sqrt{3 r(c)},0\right)$, $\left(-\sqrt{3 r(c)},0\right)$, and $(0,0)$.
\end{proof}

Under the assumption $A+B<\sqrt{2}$ appearing in 
Fact~\ref{fact_three_intersections},
for every $c\in [-1,1]$ the point 
\begin{align}\label{eq-vm-wm}
(v_m(c),w_m(c)):=\left(-\sqrt{r(c)},-\frac23 r(c)^{3/2}\right)\in C_c
\end{align}
is such that the cubic nullcline $C_c$ reaches a local minimum $w_m(c)$ at $v=v_m(c)$. Note that $(-v_m(c),-w_m(c))$ turns out to be the local maximum of $C_c$.

Observe that the map $c\mapsto (v_m(c),w_m(c))$ parametrizes a curve $J_m$ which connects $(v_m(-1), w_m(-1))$ and $(v_m(1), w_m(1))$ and that its two components are monotonically increasing with $c$, as depicted by Figure~\ref{fig_cubics_monotonicity}.  

Concerning the $w$-nullcline of \eqref{eq-1}, it is given by
\begin{equation}\label{eq-Lambda}
\Lambda:=\left\{(v,w)\in\R^2\mid g_2(v,w)=0\right\}=\left\{(v,w)\in\mathbb R^2 \mid w=\frac{v+\beta}{\gamma}\right\}.
\end{equation}
The $w$-nullcline is thus a line with positive slope that intersects the horizontal axis at $(-\beta,0)$. 

By definition, the equilibria $(v_e(c),w_e(c))$ of \eqref{eq-1} lie at the intersection of the two nullclines $C_c$ and $\Lambda$. The following result, whose proof is provided in Appendix \ref{sec-properties}, discusses the stability properties of these equilibria.

\begin{lemma}[Stability properties of the frozen system]\label{properties_equilibrium}
Let $c \in [-1,1]$ and assume that the quantity defined in \eqref{eq-rc} satisfies $r(c)>0$.
Then system \eqref{eq-1} has a unique equilibrium $(v_e(c),w_e(c))$ in $(-\infty,0]\times\R$. Moreover, this equilibrium satisfies 
$v_e(c)< 0$ and we have that
\begin{enumerate}[(i)]
\item $(v_e(c),w_e(c))$ is the unique equilibrium of \eqref{eq-1} 
if and only if $(r(c)-\frac{1}{\gamma})^{3}< \frac{9}{4} \frac{\beta^2}{\gamma^2}$.
\label{lemma_item_uniqueness}

\item $(v_e(c),w_e(c))$ is locally exponentially stable for \eqref{eq-1} if and only if  $r(c)-v_e(c)^2 < \min\{ \varepsilon \gamma, \frac{1}{\gamma}\}$.\label{lemma_item_iv}

\item If $(v_e(c), w_e(c))$ is the unique equilibrium of \eqref{eq-1} and if $r(c) - v_e(c)^2<0$ then $(v_e(c),w_e(c))$ is globally exponentially stable for all $\varepsilon>0$ small enough. \label{lemma_item_global}
\end{enumerate}
\end{lemma}

Other criteria for the global asymptotic stability of $(v_e(c),w_e(c))$ can be found in \cite{hayashi_global_1999,kaumann_uniqueness_1983,sugie_nonexistence_1991,kostovaFITZHUGHNAGUMOREVISITED2004}. We stress that our results, and in particular Item~\ref{lemma_item_global}, have the advantage of being easy to state and check and of covering the asymptotic regime where $\varepsilon$ is small. In contrast, other references provide more complete characterizations and cover all parameter ranges, but are correspondingly harder to verify in practice.

Note that $v_e(c)$ can be 
computed explicitly as the unique solution of a cubic equation. As a consequence, all the conditions involving $v_e(c)$ obtained later in the paper can be written explicitly in terms of the parameters of the system. 
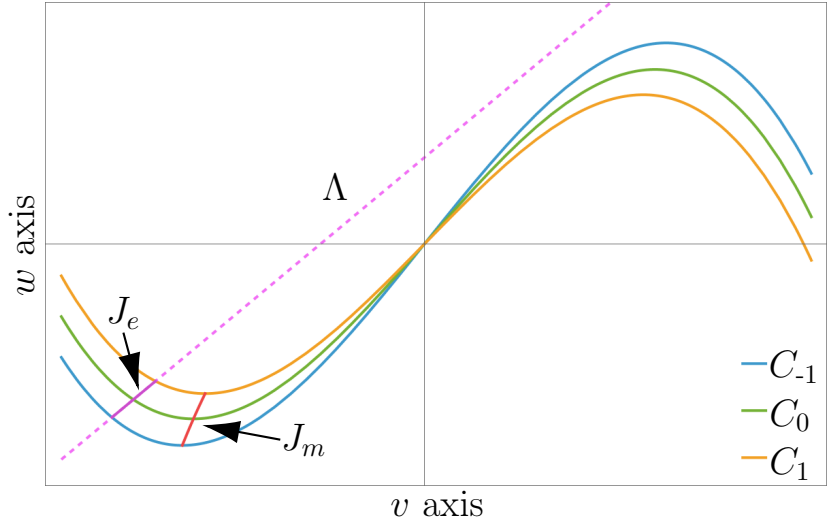
\begin{figure}
  \centering
  \begin{tikzpicture} 
    \node[anchor=south west, inner sep=0cm] (image) at (0,0)
      {\includegraphics[width=0.6\linewidth]{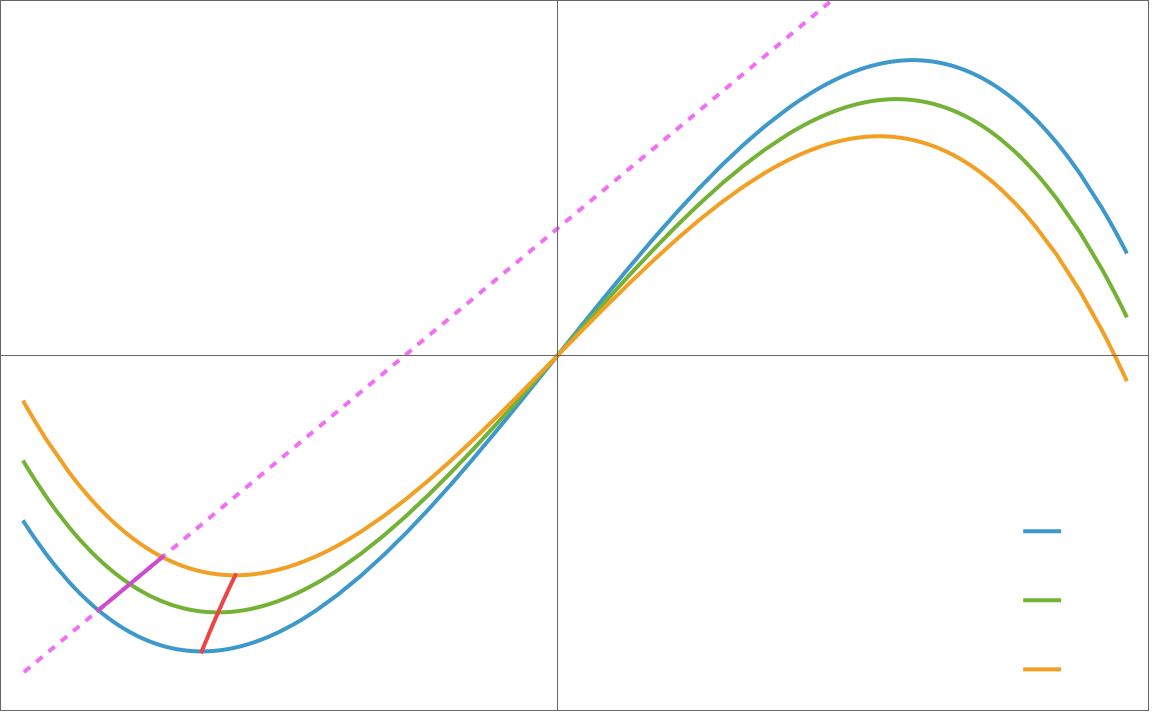}};
    \begin{scope}[x={(image.south east)}, y={(image.north west)}]
      \draw[-{Latex[length=5mm,width=3mm]}, thick]
        (0.3,0.1) -- (0.2,0.13);
      \node at (0.33,0.1)  {\Large $J_m$};
      \draw[-{Latex[length=5mm,width=3mm]}, thick]
        (0.1,0.32) -- (0.12,0.2);
      \node[black] at (0.1,0.36)  {\Large $J_e$};
      \node at (0.37,0.62)  {\Large $\Lambda$};
      \node at (0.955,0.25)  {\Large $C_{\text{-} 1}$};
      \node at (0.95,0.15)  {\Large $C_{0}$};
      \node at (0.95,0.05)  {\Large $C_{1}$};
      \node at (0.5,-0.03) {\Large $v$ axis};
      \node[rotate=90] at (-0.03, 0.5) {\Large $w$ axis};
    \end{scope}
  \end{tikzpicture}

  \caption{$v$-nullclines $C_c$ of the frozen system \eqref{eq-1} for three different values of $c$. The minimum of the cubics $C_c$ evolves monotonically with $c$ and belongs to the arc $J_m$. The segment $J_e$ contains the equilibria of \eqref{eq-1} for $c\in[-1,1]$. }
  \label{fig_cubics_monotonicity}
\end{figure}

The set of all possible equilibria for $c\in[-1,1]$ defines a segment 
\begin{align}\label{eq-Je}
    J_e:=\{(v_e(c),w_e(c))\mid c\in[-1,1]\},
\end{align}
which is contained in the line $\Lambda$ and  connects $(v_e(-1),w_e(-1))$ and $(v_e(1),w_e(1))$: see Figure \ref{fig_cubics_monotonicity}. Both components of this parameterization are monotonically increasing in $c$.

The following statement introduces parameter regions for the pair $(A, B)$ that guarantee specific stability properties of the equilibria, regardless of the considered value of $c$.

\begin{proposition}[Stability conditions regardless of $c$]\label{proposition_stability_frozen}
    Given any $\beta,\gamma,\varepsilon>0$, define the following parameter sets:
\begin{itemize}
    \item $\mathcal{E}_{\rm unique}$: the set of all pairs $(A,B)\in (0,+\infty)^2$ for which, given any $c\in [-1,1]$, system \eqref{eq-1} admits a unique equilibrium  $(v_e(c),w_e(c))$ in $\R^2$; 
    \item  $\mathcal E_{\rm LES}$: the set of all pairs $(A,B)\in \mathcal{E}_{\rm unique}$ for which, given any $c\in [-1,1]$, $(v_e(c),w_e(c))$ is locally exponentially stable for \eqref{eq-1}; 
    \item $ \mathcal E_{\rm GES}$: the set of all pairs $(A,B)\in \mathcal{E}_{\rm unique}$ for which, given any $c\in [-1,1]$, $(v_e(c),w_e(c))$ is globally exponentially stable for \eqref{eq-1};
    \item $\mathcal E_{0}$: the set of all pairs $(A,B)\in \mathcal{E}_{\rm unique}$ for which $A+B<\sqrt{2}$ and, given any $c\in [-1,1]$, $v_e(c) < v_m(c)$.
\end{itemize}
Then $\mathcal{E}_{\rm unique}$ and $\mathcal{E}_0$ do not depend on $\varepsilon$ and we have the following:
\begin{align}
    \mathcal{E}_{\rm unique}&= \left\{(A,B)\in (0,+\infty)^2\mid (A-B)^2>2\left(1-\frac{1}{\gamma}-\left(\frac{9\beta^2}{4\gamma^2}\right)^{1/3}\right) \right\} \label{property_frozen1},\\
    \mathcal E_{\rm LES}&\supset \left\{(A,B)\in \mathcal{E}_{\rm unique}\mid r(-1)<v_e(1)^2+\min\{\varepsilon\gamma,1/\gamma\}\right\}\label{property_frozen2}.  
\end{align}
Moreover, for every $\beta,\gamma>0$, there exists $\varepsilon_0>0$ such that, for every 
$\varepsilon\in (0,\varepsilon_0)$,
$\mathcal{E}_0\subset \mathcal{E}_{\rm GES}$.
\end{proposition}

\begin{proof}
    From Lemma \ref{properties_equilibrium}-\ref{lemma_item_uniqueness}, we know that $(v_e(c), w_e(c))$ is the unique equilibrium in $\R^2$ if and only if $\frac{9}{4}\frac{\beta^2}{\gamma^2}> (r(c)-1/\gamma)^3$. The monotonicity of $r(c)$ implies that the condition holds for all $c\in[-1,1]$ if it holds for $c=-1$, which yields \eqref{property_frozen1}. Inclusion \eqref{property_frozen2} is a  direct consequence of Lemma \ref{properties_equilibrium}-\ref{lemma_item_iv} and the monotonicity of $v_e(c)$ and $r(c)$.
    The last part of the statement follows immediately from the corresponding properties in Lemma \ref{properties_equilibrium}, since $v_m(c)^2 = r(c)$.
\end{proof}

\section{Necessary conditions for action potentials and tonic spiking}\label{sec-3}

In this section, we identify conditions under which no tonic spiking can occur, either under a general input (Section \ref{sec-3.1}) or under the smooth input arising from temporal interference stimulation (Section \ref{sec-3.2}). The corresponding proofs are reported in Section \ref{sec-proof-lemmas}.

\subsection{For a general input}\label{sec-3.1}

Based on the parameter sets introduced in Proposition \ref{proposition_stability_frozen}, the following result provides some insights into the dynamics of \eqref{general_switching_system} with a non-constant input $f$. It identifies parameter classes and initial states in which tonic spiking does not occur and highlights the potential for entrainment by a periodic input, as well as for confinement to specific sub-regions.

\begin{lemma}\label{lemma_general_properties}
Let $\varepsilon, \beta, \gamma,A, B  
>0$ and $f \in L^\infty([0,+\infty),[-1,1])$. The following properties hold: 
\begin{enumerate}[(i)]
    \item \label{lemma_preliminary_item1}
       Assume that $(A,B) \in \mathcal{E}_0$. If $f$ is non-constant and $T$-periodic for some $T>0$, then system \eqref{general_switching_system} has a non-constant $T$-periodic solution.
        In particular, if $T$ is small enough, then there exist initial conditions $(v_0,w_0)$ for which the corresponding solution of \eqref{general_switching_system} never emits action potentials. 
    \item \label{lemma_preliminary_item2}
        Assume that $(A,B)\in \mathcal{E}_{\rm LES}$ and consider any $\delta >0$. For every $\eta >0$ small enough and every $(v_0,w_0)$ with $|(v_0,w_0)-(v_e(f(0)),w_e(f(0)))|<\eta$,
        if   $f$ is $C^1$
         and $\|f'\|_\infty< \eta$, then  the solution of \eqref{general_switching_system} never emits action potentials and is quasi-static, in the sense that, for all $t\geq 0$, 
        \begin{equation}\label{eq:quasi-static}
          |(v(t),w(t)) - (v_e(f(t)),w_e(f(t)))| < \delta.
        \end{equation}
       If, moreover, $(A,B)\in \mathcal{E}_{\rm GES}$ then for every  $\eta >0$ small enough and every $(v_0,w_0)\in\mathbb R^2$, if $f$ is $C^1$ and       $\|f'\|_\infty< \eta$, then inequality \eqref{eq:quasi-static} holds for all $t$ large enough. 
     
    \item  \label{lemma_preliminary_item4} 
        Let $\varphi$ be a 1-periodic function and $\bar \varphi:=\int_0^1 \varphi(s)ds$. 
        Assume that system \eqref{eq-1} with $c=\bar\varphi$ has a unique equilibrium $(v_e(\bar\varphi), w_e(\bar\varphi))$ and that the latter is locally exponentially stable.
        Then,  given $\delta >0$, there exists $\delta_0 >0$ such that for every $\eta$ large enough and every $(v_0,w_0)$ with $|(v_0,w_0) - \left(v_e(\bar \varphi),w_e(\bar \varphi)\right)| <\delta_0$, the solution $(v,w)$ of \eqref{general_switching_system} with $f(t) = \varphi(\eta t)$ never emits action potentials and is quasi-static, in the sense that, for all $t\geq 0$, 
    \begin{equation}\label{eq:quasi-static-av}
    |(v(t),w(t)) - (v_e(\bar \varphi),w_e(\bar \varphi))| < \delta.
    \end{equation}
        If, moreover, $(v_e(\bar\varphi),w_e(\bar\varphi))$ is globally exponentially stable for \eqref{eq-1}
        then, for every $\eta >0$ large enough and every $(v_0,w_0)\in\mathbb R^2$,  inequality 
        \eqref{eq:quasi-static-av} holds for all $t$ large enough. 
    \end{enumerate}
\end{lemma}
For clarity of presentation, we postpone the proof of this result to Appendix~\ref{appendix_proof_lemma_general_properties}. It relies on converse Lyapunov theorems and averaging techniques. 

Item (ii) provides a rigorous counterpart, for the averaged model, to the ``envelope accommodation'' phenomenon described by Cao et al. \cite{Cao2020b} for the non-averaged FitzHugh–Nagumo model and for Hodgkin–Huxley neurons: when the envelope varies slowly, the slow recovery currents keep pace with the fast depolarizing currents and no action potential is emitted. In the same work, firing is attributed to a fast-rising envelope that outpaces the slow outward currents; this is the mechanism that we make precise in Lemma~\ref{lem:escaping}.

Based on the facts presented in the lemma above, we can refine our expectations regarding the system's dynamics. To guarantee tonic spiking with a periodic input, the excitation $f$ should neither vary too fast (item \ref{lemma_preliminary_item4}) nor too smoothly and slowly (item \ref{lemma_preliminary_item2}). 

Our positive results are actually obtained, under further geometric conditions, either for $f(t)=\sign(\cos(\eta t))$ with $\eta$ and $\varepsilon$ small or for $f(t)=\cos(\eta t)$ in the asymptotic regime where $\eta$ is picked proportional to $\varepsilon$ and $\varepsilon$ is small.
In the first case, we move beyond the situation studied in item \ref{lemma_preliminary_item2} by allowing discontinuities in $f$, while in the second case, we let $\eta$ be small, but not with respect to $\varepsilon$, so, again, item~\ref{lemma_preliminary_item2} does not apply. 

Even when action potentials can be generated, there is no guarantee that the considered input will trigger tonic spiking. The following result presents a necessary condition for the existence of repeated action potentials. 

\begin{theorem}[Necessary condition for tonic spiking]\label{thm_geometric_condition}
Let 
$ \beta, \gamma >0$ and $(A,B)\in \mathcal{E}_0$.
If the geometric condition 
\begin{equation}\label{eq:no-geom_condition}
\quad w_e(-1)>w_m(1)
\end{equation}
is satisfied, then there exists $\varepsilon_0>0$ such that, for every  $\varepsilon\in (0,\varepsilon_0)$ and every $f\in$ $L^\infty([0,+\infty),$ $[-1,1])$, 
no solution $(v,w)$ of system \eqref{general_switching_system} 
undergoes tonic spiking.
\end{theorem}

Under the geometric condition \eqref{eq:no-geom_condition}, this result states that no exogenous input can trigger tonic spiking if the recovery variable $w$ evolves too slowly with respect to the membrane potential $v$.

In view of the definition of $v_m$ and $w_m$ in \eqref{eq-vm-wm}, condition \eqref{eq:no-geom_condition} can be equivalently stated as the following constraint on the equilibria of the frozen system \eqref{eq-1} for $c=-1$:
\begin{align*}
    w_e(-1)=\frac{v_e(-1)+\beta}{\gamma}>-\frac{2}{3}\left(1-\frac{1}{2}(A+B)^2\right)^{3/2}.
\end{align*}

Condition \eqref{eq:no-geom_condition} together with the requirement that $(A, B)$ belongs to $\mathcal{E}_0$ translates into assuming that the leftmost equilibrium $(v_e(-1),w_e(-1))$ is above and on the left of the point $(v_m(1),w_m(1))$, which is the highest minimum of the $v$-nullclines $C_c$.

The proof of Theorem \ref{thm_geometric_condition} relies on the following two lemmas. The first one establishes that, for $\varepsilon>0$ small enough, every trajectory of \eqref{general_switching_system} repeatedly approaches the segment of equilibria $J_e$ introduced in \eqref{eq-Je}. Its proof is reported in Section \ref{proof-lem_region_many_visits}.

\begin{lemma}[A region with infinitely many visits]\label{lem_region_many_visits}
Let $f\in L^\infty([0,+\infty), [-1,1])$, $ \beta, \gamma >0$, and $(A,B)\in \mathcal{E}_0$.
Let $\Omega$ be a neighborhood of $J_e=\{(v_e(c),w_e(c))\mid c\in[-1,1]\}$. 
Then there exists $\varepsilon_0>0$ such that for every $\varepsilon\in (0,\varepsilon_0)$ every trajectory $(v,w)$ of \eqref{general_switching_system} admits a sequence of positive times $\{t_k\}_{k\geq 0}$ with $t_k \to +\infty$ so that $\left(v(t_k), w(t_k)\right) \in \Omega$ for every $k\ge 0$.   
\end{lemma}

The second result, proved in Section \ref{proof-lemma_geometric_condition}, shows that the geometric condition \eqref{eq:no-geom_condition} ensures that all solutions starting close to $J_e$ remain with a negative value at all times.

\begin{lemma}\label{lemma_geometric_condition}
Let $\beta ,\gamma>0$ and $(A,B)\in \mathcal{E}_{0}$. 
Suppose that condition \eqref{eq:no-geom_condition} holds. Then there exist $\varepsilon_0>0$ 
and a neighborhood $\Omega$ of $J_e$ 
such that for every $\varepsilon\in 
(0,\varepsilon_0)$, every initial condition $(v_0,w_0)\in \Omega$, and every $f\in L^\infty([0,+\infty),[-1,1])$, the corresponding solution $(v,w)$ of \eqref{general_switching_system} satisfies $v(t) <v_m(1)$ for all $t\in [0,+\infty)$. 
\end{lemma}

\subsection{Under temporal interference stimulation}\label{sec-3.2}

As already mentioned, one way to counteract the obstruction raised by Lemma \ref{lemma_general_properties} is to pick the frequency of the temporal interference stimulation $\eta$ proportional to the parameter $\varepsilon$ and to pick $\varepsilon$ small enough: this will be discussed in Section \ref{section_escaping_condition}. In this setting, the dynamics of the recovery variable $w$ are much slower than those of the membrane potential $v$. The following result provides a necessary condition for tonic spiking by relying on the study of the singular limit corresponding to an infinite timescale difference between $v$ and $w$. It relies on the curve $J_m$ introduced earlier:
\begin{align}
    J_m:=\left\{ (v_m(c),w_m(c))\,:\, c\in[-1,1]\right\},
\end{align}
where $(v_m(c),w_m(c))$ denotes the local minimum of the $v$-nullcline $C_c$ (see \eqref{eq-vm-wm} for their expression).

\begin{theorem}[No tonic spiking for $\eta=\kappa \varepsilon$ small, based on the singular limit]\label{theorem_results_singular_system}
Let $\beta, \gamma, \kappa >0$ with $(A,B)\in \mathcal E_{0}$. Assume that the differential-algebraic equation
\begin{equation}\label{differential_algebraic_system}
\left\{
\begin{aligned}
     0&= \left(1-\frac{A^2}{2}-\frac{B^2}{2}-AB \cos(\kappa s)\right)v -\frac{v^3}{3}-w,\\
    \frac{d}{ds}w &= v-\gamma w + \beta,\\
    v(0)&=v_e(-1), \quad w(0) = w_e(-1).
\end{aligned}
\right.
\end{equation}
admits a solution 
$(v,w) \in C^1([0,\pi/\kappa],\mathbb R^2)$ 
with $(v(s),w(s))\not\in J_m$ for every $s\in [0,\pi/\kappa]$. Then, for $\varepsilon$ sufficiently small and $\eta = \kappa \varepsilon$, no solution of \eqref{FHN_intro} undergoes tonic spiking. 
\end{theorem}

The proof of this result is postponed to Section \ref{sec-nontonic} as it relies on some notations that will be introduced later on.

Let us discuss the applicability of Theorem~\ref{theorem_results_singular_system}. The necessary condition for tonic spiking requires computing a specific solution to system \eqref{differential_algebraic_system} to check that it satisfies $(v(t),w(t))\notin J_m$. This solution can be computed numerically very effectively since, as we will see in Section~\ref{section_escaping_condition}, it is the integral curve of a well-defined vector field. We will also see that for $\kappa>0$ small enough, the assumptions of the theorem can be verified without computing any trajectory (cf.~Corollary~\ref{cor:noescaping} in Section~\ref{sec-nontonic}). 

In Section~\ref{sec:heuristics}, we will also discuss a positive counterpart of Theorem~\ref{theorem_results_singular_system}: we will provide a sufficient condition under which an action potential is ensured for the solutions of \eqref{FHN_intro}
for $\varepsilon$ small, $\eta = \kappa \varepsilon$, and most initial conditions. 
The condition requires that a solution of the singular system \eqref{differential_algebraic_system} with a suitable initial condition reaches $J_m$ before reaching $C_1$. Further heuristic considerations actually allow us to conjecture that, for $\varepsilon$ small enough, $\eta = \kappa \varepsilon$, and almost every initial condition, the corresponding solution of \eqref{FHN_intro} undergoes tonic spiking. Detailed explanations are given in Section~\ref{sec:heuristics}, and relevant numerical experiments are provided in Section~\ref{sec:numerical_experiments}.

\subsection{Proofs}\label{sec-proof-lemmas}
\subsubsection{Proof of Lemma \ref{lem_region_many_visits}}\label{proof-lem_region_many_visits}

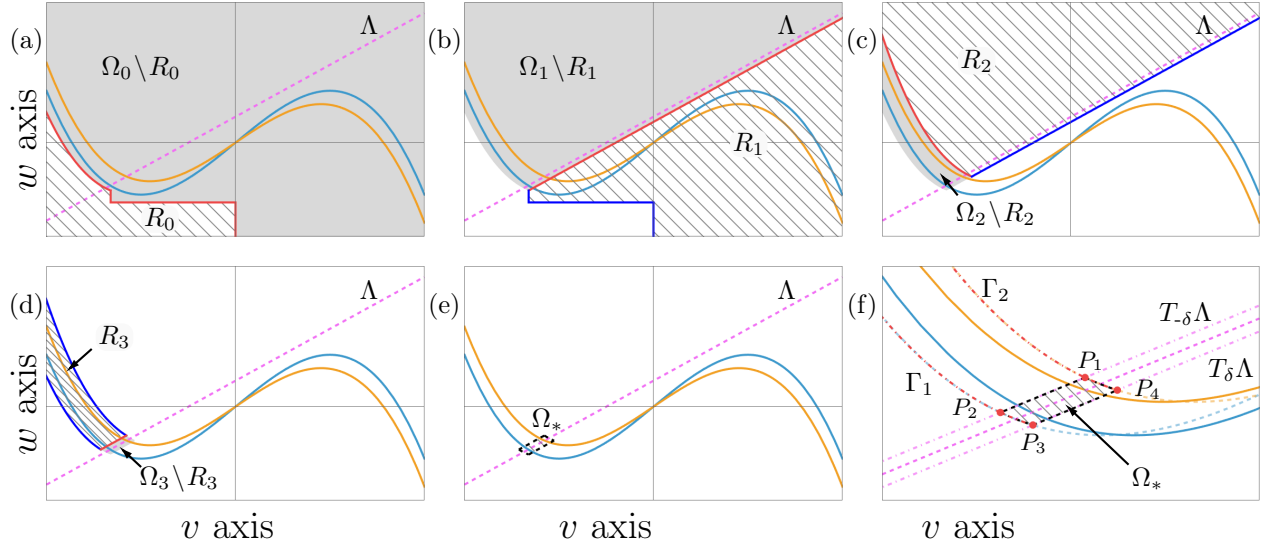
\begin{figure}
  \centering
  \begin{tikzpicture}
    \node[anchor=south west, inner sep=0cm] (image1) at (0,3.5cm)
      {\includegraphics[width=0.29\linewidth]{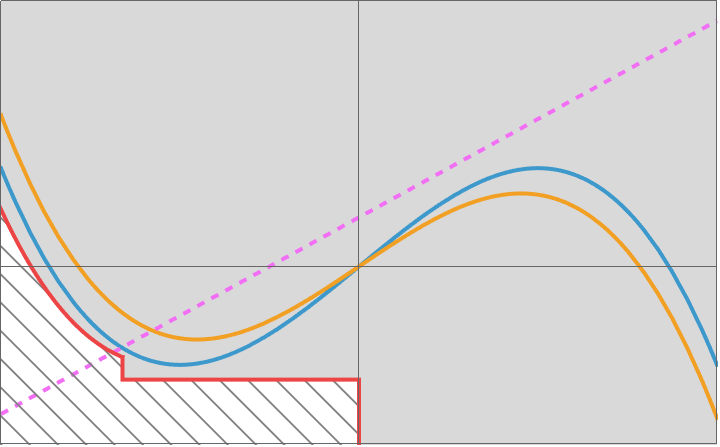}};
    \begin{scope}[ 
      shift={(image1.south west)},
      x={(image1.south east)-(image1.south west)},
      y={(image1.north west)-(image1.south west)}
      ]
      \node[thick, rounded corners, fill=gray!5, inner sep=1pt] at (0.3,0.08) {$R_0$};
      \node at (0.25,0.72) {$\Omega_0\!\setminus\!R_0$};
      \node at (0.85,0.9) {$\Lambda$};
    \end{scope}

    \node[anchor=south west, inner sep=0cm] (image2) at (0.32\linewidth,3.5cm)
      {\includegraphics[width=0.29\linewidth]{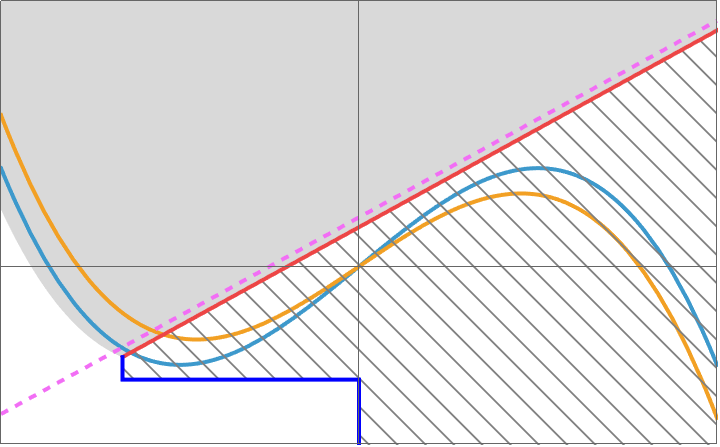}};
    \begin{scope}[ 
      shift={(image2.south west)},
      x={(image2.south east)-(image2.south west)},
      y={(image2.north west)-(image2.south west)}
      ]
      \node[thick, rounded corners, fill=gray!5, inner sep=1pt] at (0.75,0.4) {$R_1$};
      \node at (0.25,0.72) {$\Omega_1\!\setminus\!R_1$};
      \node at (0.85,0.9) {$\Lambda$};
    \end{scope}

    \node[anchor=south west, inner sep=0cm] (image3) at (0.64\linewidth,3.5cm)
      {\includegraphics[width=0.29\linewidth]{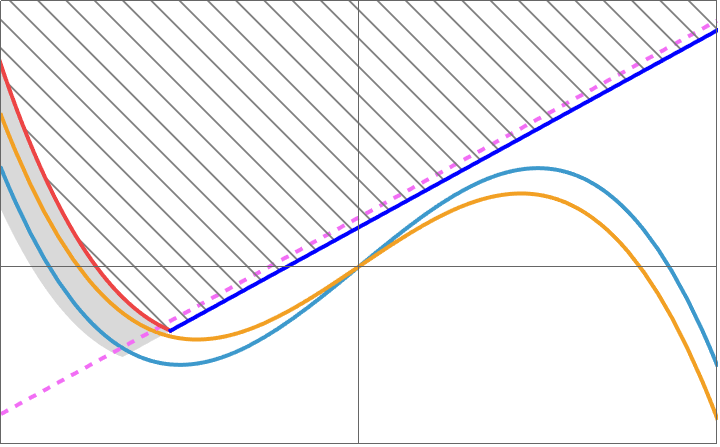}};
    \begin{scope}[ 
      shift={(image3.south west)},
      x={(image3.south east)-(image3.south west)},
      y={(image3.north west)-(image3.south west)}
      ]
      \node[thick, rounded corners, fill=gray!5, inner sep=1pt] at (0.25,0.75) {$R_2$};
      \node at (0.3,0.1) {$\Omega_2\!\setminus\!R_2$};
      \node at (0.85,0.9) {$\Lambda$};
      \draw[-{Latex[length=2mm,width=1mm]}, thick]
        (0.2,0.18) -- (0.15,0.28);
    \end{scope}
  
    \node[anchor=south west, inner sep=0cm] (image4) at (0,0)
      {\includegraphics[width=0.29\linewidth]{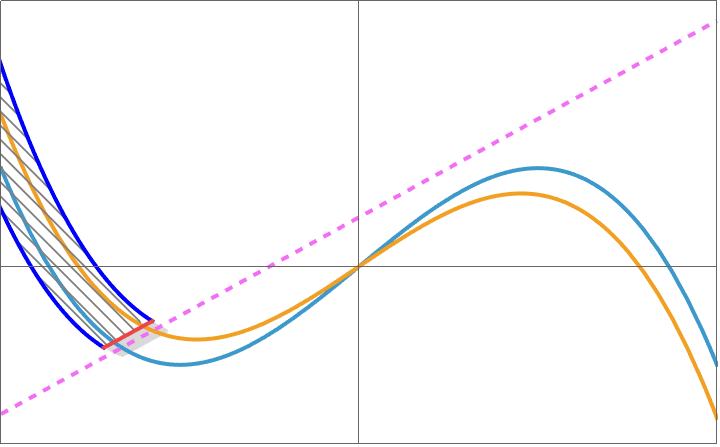}};
    \begin{scope}[ 
      shift={(image4.south west)},
      x={(image4.south east)-(image4.south west)},
      y={(image4.north west)-(image4.south west)}
      ]
      \node[thick, rounded corners, fill=gray!5, inner sep=1pt] at (0.18,0.7) {$R_3$};
      \node at (0.35,0.1) {$\Omega_3\!\setminus\!R_3$};
      \node at (0.85,0.9) {$\Lambda$};
      \draw[-{Latex[length=2mm,width=1mm]}, thick]
        (0.25,0.15) -- (0.19,0.24);
      \draw[-{Latex[length=2mm,width=1mm]}, thick]
        (0.12,0.65) -- (0.05,0.55);
    \end{scope}

    \node[anchor=south west, inner sep=0cm] (image5) at (0.32\linewidth,0)
      {\includegraphics[width=0.29\linewidth]{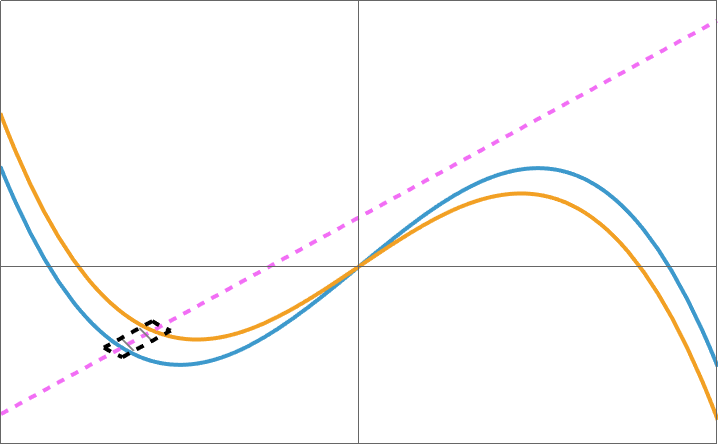}};
    \begin{scope}[ 
      shift={(image5.south west)},
      x={(image5.south east)-(image5.south west)},
      y={(image5.north west)-(image5.south west)}
      ]
      \node at (0.22,0.35) {$\Omega_*$};
      \node at (0.85,0.9) {$\Lambda$};
    \end{scope}
    \node[anchor=south west, inner sep=0cm] (image6) at (0.64\linewidth,0)
      {\includegraphics[width=0.29\linewidth]{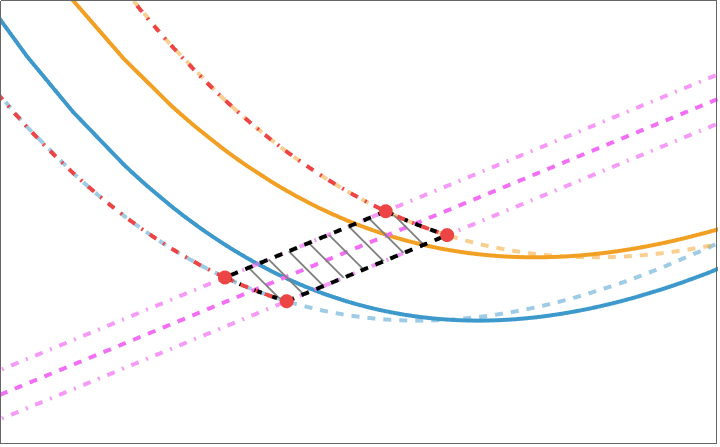}};
    \begin{scope}[ 
      shift={(image6.south west)},
      x={(image6.south east)-(image6.south west)},
      y={(image6.north west)-(image6.south west)}
      ]
      \node at (0.92,0.56) {$T_\delta \Lambda$};
      \node at (0.8,0.8) {$T_{\text{-}\delta}\Lambda$};
      \node at (0.1,0.5) {$\Gamma_1$};
      \node at (0.3,0.9) {$\Gamma_2$};
      \node at (0.55,0.6) {\small $P_1$};
      \node at (0.22,0.4) {\small $P_2$};
      \node at (0.4,0.25) {\small $P_3$};
      \node at (0.7,0.48) {\small $P_4$};
      \draw[-{Latex[length=2mm,width=1mm]}, thick]
        (0.65,0.16) -- (0.48,0.42);
      \node at (0.7,0.1) {$\Omega_*$};
    \end{scope}
  
    \node at (2.4cm,-0.3cm) {\Large $v$ axis};
    \node at (7.3cm,-0.3cm) {\Large $v$ axis};
    \node at (12cm,-0.3cm) {\Large $v$ axis};
    \node[rotate=90] at (-0.3cm, 1.3cm) {\Large $w$ axis};
    \node[rotate=90] at (-0.3cm, 4.8cm) {\Large $w$ axis};
    \node at (-0.015\linewidth, 6.1cm) {(a)};
    \node at (0.305\linewidth, 6.1cm) {(b)};
    \node at (0.625\linewidth, 6.1cm) {(c)};
    \node at (-0.015\linewidth, 2.6cm) {(d)};
    \node at (0.305\linewidth, 2.6cm) {(e)};
    \node at (0.625\linewidth, 2.6cm) {(f)};
  \end{tikzpicture}

  \caption{(a)-(e) Represent the regions described in the proof of Lemma~\ref{lem_region_many_visits} in the $vw$ plane. The red (respectively, blue) arcs indicate the curves through which the trajectories can (respectively, cannot) leave the corresponding region $R_i$. (f) shows the neighborhood $\Omega_*$ used in the construction
  }
  \label{fig:region_many_visits}
\end{figure}

We say that a region $\Omega_*\subset \R^2$ is a region of {\it repeated visits} if, for every trajectory of system \eqref{general_switching_system}, there exists a sequence of times $t_k\to +\infty$ such that the trajectory is in $\Omega_*$ at such times $t_k$. We 
fix a neighborhood $\Omega$ of  $J_e=\{(v_e(c),w_e(c))\mid c\in[-1,1]\}$ and we 
aim at showing that $\Omega$ is a region of repeated visits provided that $\varepsilon$ is small enough. We start by picking $\Omega_0 = \R^2$, which is trivially a region of repeated visits, 
and we iteratively trim it to obtain smaller and smaller regions of repeated visits, ultimately yielding a region contained in $\Omega$. 

We argue in the following way. Let $\Omega_k$ be a region of repeated visits and consider $R_k\subset \Omega_k$. We claim that if, for each $\tau_0>0$, each trajectory with initial condition in $R_k$ at time $t=\tau_0$  eventually leaves $R_k$ towards $\Omega_k \setminus R_k$, then $\Omega_k\setminus R_k$ is a region of repeated visits. 
Indeed, consider any trajectory of \eqref{general_switching_system} and notice that, by definition, it either intersects $\Omega_k\setminus R_k$ or $R_k$ along an unbounded sequence of times. By the assumption on $R_k$, if the trajectory intersects $R_k$ along an unbounded sequence of times, then the same is also true by replacing $R_k$ by $\Omega_k\setminus R_k$, and the claim is proved. 

To iteratively construct the sets $\Omega_k$ and $R_k$, let us introduce some useful objects in $\R^2$. We use the notation introduced in Section~\ref{key_objects} in what follows. Let $\delta>0$ be a small parameter to be fixed later. Denote by $T_\delta$ the shift $T_\delta:\R^2\to \R^2$ defined by $T_\delta(v,w)=(v+\delta,w)$. Its inverse is $T_{-\delta}$. Let the points $P_1,P_2,P_3,P_4\in \R^2$ be defined as follows:
\begin{equation*}
\begin{array}{ll}
T_{\delta}(C_1)\cap T_{\text{-}\delta}(\Lambda)=\{P_1\},
&T_{\text{-}\delta}(v_e(-1),w_e(-1))=P_2,\\  
T_{ \text{-}\delta}(C_{-1})\cap T_{\delta}(\Lambda)=\{P_3\},
&T_{\delta}(v_e(1),w_e(1))=P_4.
\end{array}
\end{equation*}

Denote by $\Omega_*$ the bounded region delimited by: the segment between $P_1$ and $P_2$, the portion of $T_{-\delta}(C_{-1})$  between $P_2$ and $P_3$, the segment between $P_3$ and $P_4$, the portion of $T_\delta(C_{1})$  between $P_4$ and $P_1$. Provided that $\delta>0$ is small enough, $\Omega_*\subset \Omega$. 

Let $\Gamma_1$ be the portion of  $T_{\text{-}\delta}(C_{-1})$ at the left of $P_3$ and $\Gamma_2$ be the portion of  $T_{\delta}(C_{1})$ at the left of $P_4$, that is, letting $P_3=(p_{31},p_{32})$,
\begin{align*}
    \Gamma_1&=\{(v,r(-1) (v + \delta)  - (v+ \delta)^3/3)\mid v\in(-\infty,    p_{31})\},\\
  \Gamma_2&=\{(v, r(1) (v-\delta)  - (v-\delta)^3/3)\mid v\in(-\infty,v_e(1)+\delta)\}. 
\end{align*} 

Both curves can be seen as graphs of  functions of the variable $v$, and as such, have slope
bounded above by a negative constant, provided that $\delta$ is small enough.
For every $c\in [-1,1]$, the vector field $G_c$ points to the left of $\Gamma_2$, since both $g_1(\cdot,c)$ and $g_2$ are negative on $\Gamma_2$.

We claim that for $\varepsilon>0$ small enough (depending on $\delta$) $G_c$, $c\in [-1,1]$, points to the right of $\Gamma_1$. Indeed, the slope of $G_c(v,w)$ at $(v,w)=(v,r(-1) (v + \delta)  - (v+ \delta)^3/3)$ is equal to 
\[\varepsilon\frac{ g_2(v,w)}{g_1(v,w,c)}=\varepsilon\frac{v - \gamma w + \beta}{r(c) v - \frac{v^3}{3}-w}
=\varepsilon
\frac
{
   \frac{\gamma}{3}v^3
 +\gamma \delta v^2 
+(1 - \gamma r(-1)   +\gamma  \delta^2)v
+\gamma  \delta (\frac{\delta^2}{3} -r(-1)) + \beta}
{
 \delta v^2
+(r(c)  -r(-1)+ \delta^2)v 
+ \delta (\frac{\delta^2}{3}- r(-1)) }. \]
Since the slope of $\Gamma_1$ goes as $-v^2$ as $v\to-\infty$, we easily deduce that, up to taking $\varepsilon$ small enough the slope of $G_c(v,w)$ is larger than that of $\Gamma_1$ everywhere on $\Gamma_1$ for every $c\in [-1,1]$. Hence, $G_c$, $c\in [-1,1]$, points to the right of $\Gamma_1$.

The iterative construction of $\Omega_k$ and $R_k$ is illustrated in Figure~\ref{fig:region_many_visits}. As already mentioned, $\Omega_0 = \R^2$. 
\begin{itemize}
\item Let $R_0=\cup_{i=1}^3 R_{0i}$, where
\begin{align*}
    R_{01}&=\{(v,w)\mid v\le     0,\;w\le  w_m(-1)-\delta\},
\qquad R_{02}
=\{(v,w)\mid v\le  p_{31},\; 
w\le p_{32}\},
\end{align*}
and $R_{03}$ is the subregion of the half-plane $\{w\ge p_{32}\}$ that is to the left of $\Gamma_1$. 
Inside $R_0$, the derivative $\frac{dv}{dt}=g_1(v,w,c)$ is uniformly bounded below by a positive constant. Hence, a trajectory of \eqref{general_switching_system} starting in $R_0$ leaves it in finite time. 
Set $\Omega_1 = \Omega_0 \setminus R_0$.
\item Let $R_1$ be the intersection of $\Omega_1$ with the half-plane that is below $T_\delta(\Lambda)$, i.e., 
\[R_1=\left\{(v,w)\in \Omega_1 \mid w\le  \frac{v-\delta+\beta}{\gamma}\right\}.\]
We notice that $g_2$ is uniformly bounded below by a positive constant in $R_1$. 
As a consequence, a trajectory $(v,w)$ of \eqref{general_switching_system} cannot stay in $R_1$ forever: indeed, on the one hand, if at some time $t$ the point $(v(t),w(t))$ is in $R_1$ and above $C_{-1}$, then $\dot v\le 0$ as long as $(v,w)$ stays in $R_1$, which means that $w$ is upper bounded by $\frac{v(t)-\delta+\beta}{\gamma}$, and, on the other hand, if $(v(t),w(t))$ is in $R_1$ and below $C_{-1}$, then $(v,w)$ reaches either $T_\delta(\Lambda)$ or $C_{-1}$ in finite time. 
In conclusion, every trajectory of \eqref{general_switching_system} leaves $R_1$ in finite time through $T_\delta(\Lambda)\cap \Omega_1$,  towards $\Omega_1\setminus R_1$. 
Set $\Omega_2 = \Omega_1 \setminus R_1$. 
\item 
Let $R_2$ be the set of points that are to the right of $\Gamma_2$ and in the half-plane above $T_\delta(\Lambda)$. Assuming that $\delta$ is small enough, $T_\delta(\Lambda)$ does not intersect $C_{-1}$ in the first quadrant, so that $g_1(\cdot,c)$ is uniformly bounded above by a negative value  on $R_2$ for $c\in [-1,1]$. 
Reasoning as in the previous step, each trajectory of \eqref{general_switching_system} leaves $R_2$ in finite time and can do so 
only through $\Gamma_2$. Set $\Omega_3 = \Omega_2 \setminus R_2$.
\item 
Let $R_3=\Omega_3\setminus \Omega_*$, so that the boundary of  $R_3$ is made of the portion of $\Gamma_1$ at the left of $P_2$, the segment between $P_1$ and $P_2$, and the portion of $\Gamma_2$ at the left of $P_1$. 
Since $g_2$ is uniformly bounded above by a negative value in $R_3$ and $G_c$ points towards the interior of $R_3$ on the portion of $\Gamma_1$ at the left of $P_2$ and the portion of $\Gamma_2$ at the left of $P_1$, then each trajectory of \eqref{general_switching_system} exits $R_3$ in finite time towards $\Omega_*=\Omega_3\setminus R_3$.
\end{itemize}

We conclude that $\Omega_*$ is a region of repeated visits. This completes the proof of Lemma~\ref{lem_region_many_visits}. 

\subsubsection{Proof of Lemma \ref{lemma_geometric_condition}}\label{proof-lemma_geometric_condition}

The idea is to construct a region that is positively invariant for the frozen system \eqref{general_switching_system} corresponding to the vector field $G_c(v,w)\!=\!(g_1(v,w,c), \varepsilon g_2(v,w))^\top$, defined in \eqref{defi_F1}-\eqref{defi_F2}. To that aim, we first describe a positively invariant region for the singular limit corresponding to $\varepsilon=0$, and then we adjust it so that the argument holds for small $\varepsilon > 0$. To show that a closed region is positively invariant, it is enough to show that at every point of its boundary, 
each vector field $G_c$, $c\in[-1,1]$, either points towards the interior of the region or is tangent to it (for a precise characterization even in the case of less regular sets than those considered here, see \cite{Nagumo1942} or \cite{Blanchini1999}). Let $\delta = \frac{1}{2}(w_e(-1) - w_m(1))>0$ and let $p$, $q\in \R$ be the leftmost solutions of the equations 
\begin{equation*}
g_1(p, w_e(1)+\delta,-1) = 0, \quad g_1(q, w_e(-1)-\delta,1)=0,
\end{equation*}
i.e., the only solutions to such equations that satisfy $p<v_m(-1)$ and $q<v_m(1)$ (cf.~Figure~\ref{fig_geometric_condition3}). Let us show that $p < v_e(-1) < q$.
Write $g_1(v,w,c) = h(v,c)-w$. By the definition of $\mathcal{E}_0$ we know that $v_e(c)<v_m(c)$. Observe that $c\mapsto h(v,c)$ is monotonically increasing, while $v\mapsto h(v,c)$ is monotonically decreasing on $(-\infty,v_m(c)]$ 
for $c\in[-1,1]$.
Then 
\[
h(q,-1)<h(q,1)=
w_e(-1)-\delta < w_e(-1) =h(v_e(-1),-1)
\]
implies that $q>v_e(-1)$. Similarly, 
\[
h(p,-1)=w_e(1)+\delta>w_e(1)>w_e(-1)=h(v_e(-1),-1)
\]
implies that
$p<v_e(-1)$. 
Consider the rectangle $R= [p,q] \times [w_e(-1) - \delta, w_e(1) + \delta]$ and denote its sides by $S_1 = [p,q]\times \{w_e(-1)-\delta\}$, $S_2 = \{p\}\times [w_e(-1)-\delta , w_e(1) + \delta]$, $S_3 = [p,q]\times \{w_e(1)+\delta\}$, $S_4 = \{q\}\times [w_e(-1)-\delta , w_e(1) + \delta]$, as represented in Figure~\ref{fig_geometric_condition3}.
Then this region $R$ is positively invariant in the limit case $\varepsilon=0$. This can be deduced from two facts. First, all trajectories of the system are horizontal in this case (as $\dot w=0$), meaning that trajectories cannot leave $R$ from its horizontal segments. Second, $S_2$ is to the left and $S_4$ to the right of each cubic $C_c$, 
$c\in [-1,1]$, meaning that the vector field of \eqref{general_switching_system} points towards right on $S_2$ and towards left on $S_4$.

\begin{figure}
    \centering
\begin{tikzpicture} 
  \node[anchor=south west, inner sep=0cm] (image) at (0,0)
    {\includegraphics[width=0.6\linewidth]{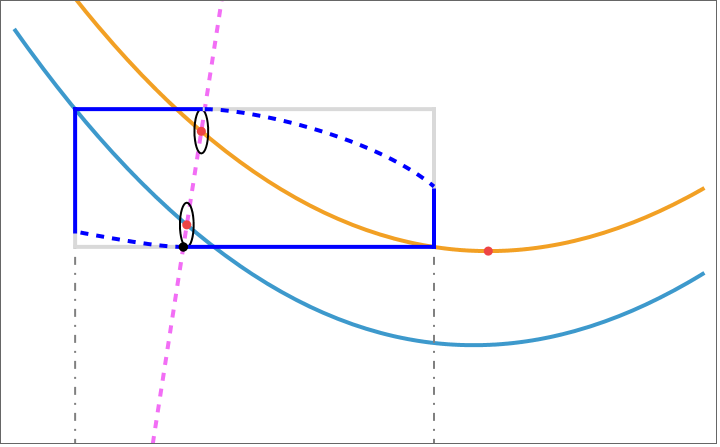}};
  \begin{scope}[x={(image.south east)}, y={(image.north west)}]
    \node at (0.13,0.08) {$p$};
    \node at (0.63,0.08) {$q$};
    \node at (0.2,0.2) {$\Lambda$};
    \node at (0.22,0.4) {$P$};
    \node at (0.8,0.52) {$C_1$};
    \node at (0.8,0.19) {$C_{\text{-}1}$};
    \node[black!60] at (0.43,0.39) {$S_1$};
    \node[black!60] at (0.07,0.58) {$S_2$};
    \node[black!60] at (0.16,0.81) {$S_3$};
    \node[black!60] at (0.63,0.57) {$S_4$};
    \node at (0.45,0.9) {$\left(v_e(1),w_e(1)\right)$};
    \node at (0.4,0.18) {$\left(v_e(\text{-}1),w_e(\text{-}1)\right)$};
    \node at (0.75,0.39) {$\left(v_m(1),w_m(1)\right)$};
    \draw[-{Latex[length=3mm,width=2mm]}, thick]
      (0.37,0.25) -- (0.28,0.46);
    \draw[-{Latex[length=3mm,width=2mm]}, thick]
      (0.38,0.83) -- (0.29,0.72);
    
    \node at (0.5,-0.03) {\Large $v$ axis};
    \node[rotate=90] at (-0.03, 0.5) {\Large $w$ axis};
  \end{scope}
\end{tikzpicture}
    \caption{Description of the geometric elements used in the construction of the invariant region $\Omega$ considered in the proof of Lemma~\ref{lemma_geometric_condition}
    }
    \label{fig_geometric_condition3}
\end{figure}
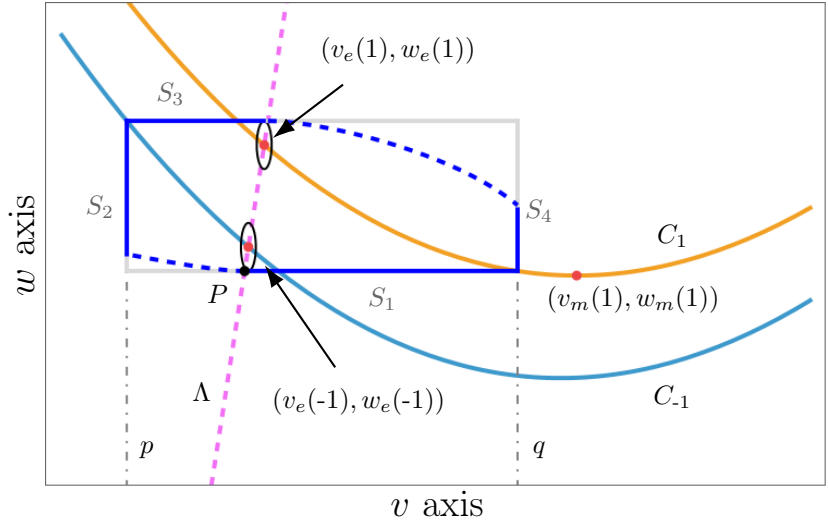

The previous argument fails when $\varepsilon>0$, since the trajectories of the system are no longer horizontal. However, we only need to make minor modifications on the region $R$ to construct a positively invariant neighborhood $\Omega$ of $J_e$ for $\varepsilon >0$ small enough. 

For $\varepsilon>0$, the vertical segments $S_2$ and $S_4$ still cannot be crossed. Indeed, by construction, $g_1(v,w,c) >0$ (respectively, $g_1(v,w, c) <0$) for all $(v,w)$ in the interior of $S_2$ (respectively, $S_4$) and all $c\in [-1,1]$.
Moreover, at the upper-left and the lower-right corners of $R$, the vector field $G_c$ points either to the interior of $R$ or is tangent to the corresponding segment $S_2$ or $S_4$, depending on the value of $c$. 

For points $(v,w)$ in the interior of the lower horizontal segment $S_1$ and to the right of the $w$-nullcline $\Lambda$ (defined in \eqref{eq-Lambda}), we have $g_2(v,w) >0$, so each vector field $G_c$, $c\in [-1,1]$, points towards the interior of $R$. We will now replace the portion of $S_1$ that is to the left of $\Lambda$ (if it is nonempty) by a suitable curve. 

To that aim, consider the integral curve of the vector field $-G_{-1}$ (corresponding to $G_{-1}$ backwards in time) starting from the point $P = (\gamma(w_e(-1) - \delta) - \beta, w_e(-1)-\delta)$, as represented in blue in Figure \ref{fig_geometric_condition3}.
If $\varepsilon\in (0,\varepsilon_0]$ with $\varepsilon_0>0$ small enough, such a curve crosses $S_2$ before crossing the cubic $C_{-1}$. Let $G_{-1,\varepsilon_0}$ be the vector field $G_{-1}$ obtained taking $\varepsilon=\varepsilon_0$. 
Denote by $\Gamma$ the portion of integral line of $G_{-1,\varepsilon_0}$ between $S_2$ and $P$. 
In the construction of $R$, replace the lower left corner of $R$ by $\Gamma$. 
Since $g_1(v,w,c)$ increases monotonically with respect to $c$ where $v<0$, while $g_2$ is independent of $c$ and increases monotonically with respect to $\varepsilon$ on the left of $\Lambda$, it follows that $G_c(v,w)$ points towards the interior of the newly defined region for every $(v,w)$ in $\Gamma$, every $c\in [-1,1]$ and every $\varepsilon\in (0,\varepsilon_0)$, except for the point $P$ at which each $G_c$ is tangent to $S_1$. 
 
A similar analysis can be done to replace the 
upper right corner of $R$, picking the point $(\gamma(w_e(1) + \delta) - \beta, w_e(1)+\delta)$ and considering the integral curve of the vector field $G_{1,\varepsilon_0}$ backwards in time. 

Thus we find a region $\Omega$ as shown in Figure~\ref{fig_geometric_condition3}, which is positively invariant for each vector field $G_c$, $c\in[-1,1]$, $\varepsilon\in (0,\varepsilon_0)$. We conclude that, up to possibly reducing $\varepsilon_0$, every solution of system~\eqref{general_switching_system} eventually reaches $\Omega$ by Lemma~\ref{lem_region_many_visits} and then stays in $\Omega$ for all larger times.

\section{Tonic spiking under piecewise constant input}\label{sec:four}

\subsection{Sufficient condition}

A first way to overcome the obstruction posed by Lemma \ref{lemma_general_properties}-\ref{lemma_preliminary_item2} is to renounce continuous inputs. The next result guarantees tonic spiking using a piecewise-constant input of the form $f(t)=\sign(\cos(\eta t))$ with $\eta$ small enough. Its proof is reported in Section \ref{proof-thm_activation_piecewise_constant}.

\begin{theorem}[Tonic spiking through piecewise constant inputs]\label{thm_activation_piecewise_constant}
Let $\beta, \gamma >0$, $(A,B)\in \mathcal{E}_{\rm unique}$ with $A+B<\sqrt{2}$, and $v_e(-1)< v_m(-1)$. 
Suppose that the geometric condition
\begin{equation}\label{eq:geom_condition}
w_e(-1)<w_m(1)
\end{equation}
is satisfied. Then there exists $\varepsilon_0>0$ such that, for every $\varepsilon\in(0,\varepsilon_0)$, there exists $\eta_0 = \eta_0(\varepsilon) >0$ such that for every $\eta\in(0,\eta_0)$, every solution $(v,w)$ of 
\begin{equation}\label{PAS_FHN_square_initial}
\left\{
\begin{array}{rl}
\dot{v} &= (1-\frac{A^2}{2}-\frac{B^2}{2}-AB {\rm sign}(\cos(\eta t))) v- \frac{v^3}{3} -w,\\
\dot{w} &= \varepsilon(v-\gamma w +\beta),
\end{array}\right.
\end{equation}
undergoes tonic spiking.
\end{theorem}
\begin{remark}
    Observe that the assumptions $(A, B)\in \mathcal{E}_{\rm unique}$, $A+B<\sqrt{2}$, and $v_e(-1)< v_m(-1)$ are satisfied in particular when $(A, B)\in \mathcal{E}_0$. These assumptions are actually strictly weaker since they do not impose that  
    $v_e(c)< v_m(c)$ 
    for each $c\in (-1,1]$.
\end{remark}

\begin{remark}\label{remark_different_values_AB}
By adjusting the values of $A$ and $B$,
system \eqref{PAS_FHN_square_initial} can be seen as an approximation of \eqref{FHN_intro} because of the Fourier series expansion $\sign(\cos(t)) = \frac{4}{\pi} \cos(t) - \frac{4}{3 \pi} \cos(3t) + \cdots$. 
Indeed, we can approximate
\begin{equation*}
1-\frac{A^2}{2}-\frac{B^2}{2}-AB\cos(\eta t) \approx 1-\frac{A^2}{2}-\frac{B^2}{2}- \frac{\pi}{4}AB\sign(\cos(\eta t) )= 1-\frac{\tilde A^2}{2}-\frac{\tilde B^2}{2}-\tilde A\tilde B\sign(\cos(\eta t))
\end{equation*}
where $\tilde A= R \cos\theta$, $\tilde B = R\sin\theta$ with $R= \sqrt{A^2+B^2}$ and 
\begin{equation*}
\theta = \frac{1}{2} \arcsin\left(\frac{\pi}{8}\frac{A B}{A^2 + B^2}\right).
\end{equation*}
\end{remark}

An essential difference between the piecewise constant case \eqref{PAS_FHN_square_initial} and system \eqref{FHN_intro} is that tonic spiking in \eqref{PAS_FHN_square_initial} arises for every  $\eta$ small enough thanks to the discontinuity of the applied input. In \eqref{FHN_intro}, the input is smooth, and the system falls in a quasi-static regime for $\eta$ small in which no action potentials are emitted (Lemma~\ref{lemma_general_properties}-\ref{lemma_preliminary_item2}).

\subsection{Proof of Theorem~\ref{thm_activation_piecewise_constant}}\label{proof-thm_activation_piecewise_constant}

Before the proof, we state the following lemma, which refines Lemma~\ref{lem_region_many_visits} for inputs of the form $f(t)={\rm sign}(\cos(\eta t))$, where the input determines the active vector field, causing the system to follow $G_1$ when $\cos(\eta t)\ge 0$ and $G_{-1}$ when $\cos(\eta t)<0$.

\begin{lemma}[Initial approach to the equilibrium]\label{lemma_time_to_approach}
Let $\beta, \gamma >0$, $(A,B)\in \mathcal{E}_{\rm unique}$ with $A+B<\sqrt{2}$, and $v_e(-1)< v_m(-1)$.  
There exists $\varepsilon_0$ such that for every $0<\varepsilon< \varepsilon_0$ and every $\delta>0$ there exists $\eta_0 >0$ such that if $\eta\in (0,\eta_0)$ every trajectory of system \eqref{PAS_FHN_square_initial} eventually reaches the ball centered in $(v_e(-1),w_e(-1))$ of radius $\delta$. More precisely, there exists a time $t^*>0$ at which $t\mapsto {\rm sign}(\cos(\eta t))$ switches from $-1$ to $1$ and $|(v(t^*), w(t^*)) - (v_e(-1),w_e(-1))| < \delta$.
\end{lemma}
\begin{proof}
First, because of Fact \ref{fact1}, we know that every trajectory of  \eqref{PAS_FHN_square_initial} eventually reaches a compact neighborhood $R$ of the origin and stays there thereafter. From Lemma~\ref{properties_equilibrium}-\ref{lemma_item_global} and since $v_m(-1)^2=r(-1)$, we know that for all $\varepsilon$ small enough, $(v_e(-1),w_e(-1))$ is globally exponentially stable for $G_{-1}$. Then there exists $T>0$ such that
every integral curve of $G_{-1}$ with initial condition $(v_0,w_0)$ 
in $R$
reaches the ball centered at $(v_e(-1), w_e(-1))$ of radius $\delta$ within time $T$ and stays there thereafter. 
The proof is then concluded by taking $\eta_0>0$ such that $\pi/\eta_0>T$. 
\end{proof}

We now proceed to the proof of Theorem~\ref{thm_activation_piecewise_constant}. Let $0<\delta < w_m(1) - w_e(-1)$ be fixed. 
 Thanks to Lemma~\ref{lemma_time_to_approach}, we know that there exists $\varepsilon_0>0$ such that for all $\varepsilon\in (0,\varepsilon_0)$ there exists $\eta_0(\varepsilon)>0$ such that for $\eta \in (0, \eta_0(\varepsilon))$ 
 any trajectory  $(v,w)$ of system 
 \eqref{PAS_FHN_square_initial}
 satisfies 
 $|(v(t^*), w(t^*)) - (v_e(-1),w_e(-1))| < \delta$ at some  $t^*>0$ at which $t\mapsto {\rm sign}(\cos(\eta t))$ switches from $-1$ to $1$. 
 
The proof of the theorem is complete if we show that, up to further reducing $\varepsilon_0$ and $\eta_0(\cdot)$, $(v,w)$ must enter the half-plane $\{v>0\}$ after time $t^*$.
Let us then construct a region $H$ containing the ball centered at $(v_e(-1),w_e(-1))$ of radius $\delta$ such that if $(v(t^*), w(t^*))\in H$ then $(v(\cdot),w(\cdot))$ exits $H$ in finite time after $t^*$ by crossing the line $\{v=0\}$. Let $P_1$ and $P_2$ be the intersection points of $\Lambda$ with the horizontal lines $\{w=w_e(-1)-\delta\}$ and $\{w=w_e(-1)+\delta\}$, respectively, that is, 
\[P_1=(\gamma(w_e(-1)-\delta) - \beta, w_e(-1)-\delta),\qquad P_2=(\gamma(w_e(-1)+\delta) - \beta, w_e(-1)+\delta).\] 

Let $H$ be the bounded region whose boundary is made of the following six curves  (illustrated in Figure~\ref{fig:figure4_construction_H}):
\begin{itemize}
\item $S_1$ is the horizontal line segment between $P_1$ and $(0,w_e(-1)-\delta)$;
\item  $S_2$ is the vertical segment between $(0, w_e(-1) - \delta)$ and $(0,w_m(1))$;
\item  $S_3$ is the line segment between $(0, w_m(1))$ and $P_2$; 
\item  $S_4$ is the horizontal line segment between $P_2$
and
$(p,w_e(-1) + \delta)$, where 
$p\in \R$ is the leftmost solution of the equation 
$g_1(p, w_e(1)+\delta,-1) = 0$, i.e., $(p, w_e(1)+\delta)$ is the leftmost intersection point of the cubic $C_{-1}$ with the horizontal line $\{w=w_e(1)+\delta\}$;
\item  $S_5$ is the vertical line segment between $(p, w_e(-1) + \delta)$ and $(p,q)$, where $q$ is such that $(p,q)$ is the first point of intersection of the vertical line $\{v=p\}$ with the integral curve of the vector field $-G_{-1}$ that starts from $P_1$, where $\varepsilon_0$ is chosen so small that, for every $\varepsilon\in (0,\varepsilon_0)$, when integrating $-G_{-1}$ we cross $\{v = p\}$ before $\{w = w_e(-1) + \delta\}$;
\item  $S_6$ is the arc of the integral curve of the vector field 
$G_{-1}$ that connects $(p,q)$ and $P_1$.
\end{itemize}

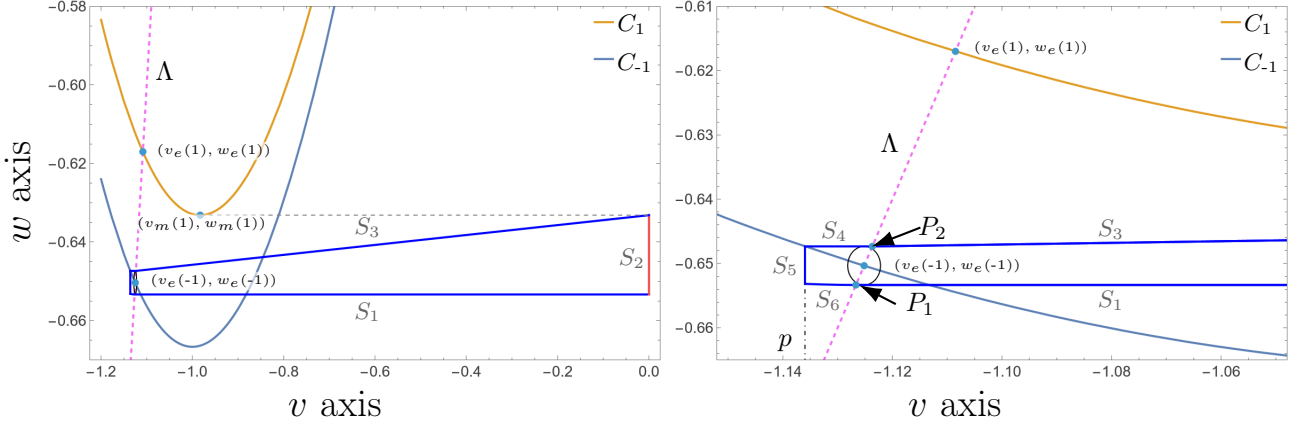
\begin{figure}[ht]
    \centering
    \begin{tikzpicture}
    \node[anchor=south west, inner sep=0cm] (image1) at (0,0)
    {\includegraphics[width=0.47\linewidth]{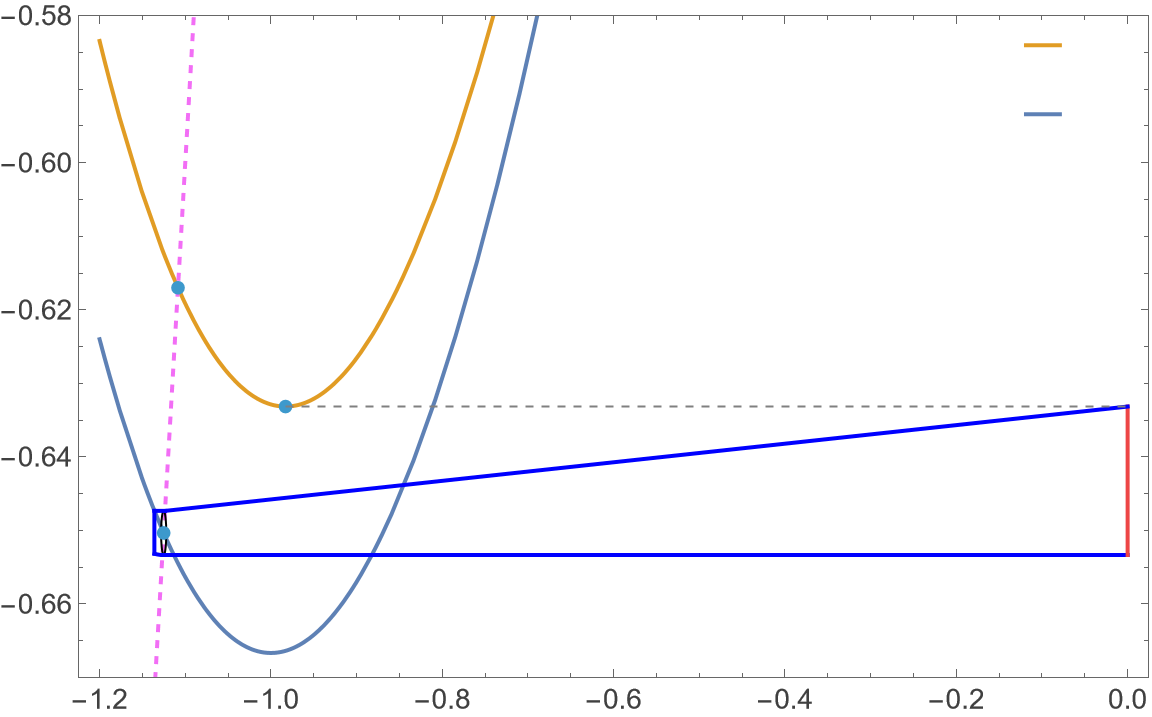}};
    \begin{scope}[
    shift={(image1.south west)},
    x={(image1.south east)}, 
    y={(image1.north west)}]
      \node[black!60] at (0.52,0.18) {\small $S_1$};
      \node[black!60] at (0.95,0.32) {\small $S_2$};
      \node[black!60] at (0.52,0.40) {\small $S_3$};
      \node at (0.19,0.81) {$\Lambda$};
      \node at (0.95,0.93) {\small $C_1$};
      \node at (0.955,0.83) {\small $C_{\text{-}1}$};
      \node[
        fill=white, 
        fill opacity=0.6, 
        text opacity=1, 
        draw=none, 
        thick,
        rounded corners,
        inner sep=1pt
      ]  at (0.27,0.60) {\tiny $\left(v_e(1),w_e(1)\right)$};
      \node[
        fill=white, 
        fill opacity=0.6, 
        text opacity=1, 
        draw=none, 
        thick,
        rounded corners,
        inner sep=1pt
      ]  at (0.25,0.41) {\tiny $\left(v_m(1),w_m(1)\right)$};
      \node[
        fill=white, 
        fill opacity=0.6, 
        text opacity=1, 
        draw=none, 
        thick,
        rounded corners,
        inner sep=1pt
      ]  at (0.27,0.26) {\tiny $\left(v_e(\text{-}1),w_e(\text{-}1)\right)$};
    \end{scope}
    
    \node[anchor=south west, inner sep=0cm] (image2) at (0.48\linewidth,0)
    {\includegraphics[width=0.47\linewidth]{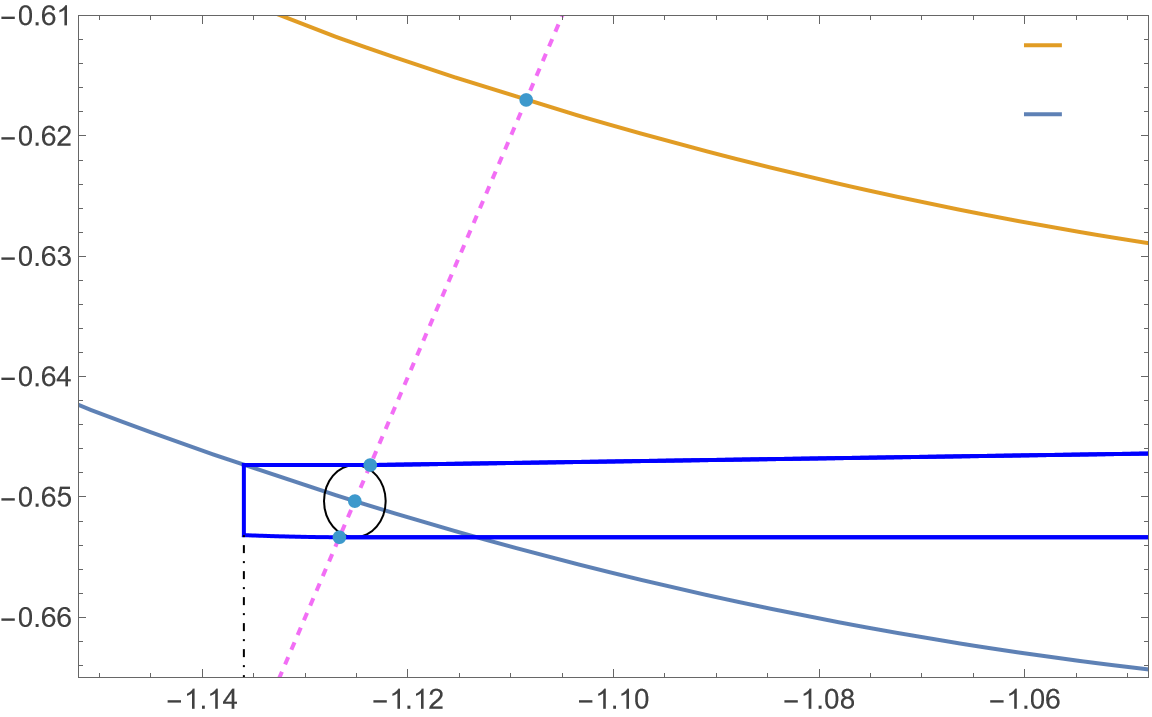}};
    \begin{scope}[
    shift={(image2.south west)},
    x={(image2.south east)}, 
    y={(image2.north west)}]
      \node at (0.18,0.1) {$p$};
      \node at (0.4,0.2) {$P_1$};
      \node at (0.42,0.4) {$P_2$};
      \node[black!60] at (0.71,0.21) {\small $S_1$};
      \node[black!60] at (0.71,0.395) {\small $S_3$};
      \node[black!60] at (0.26,0.39) {\small $S_4$};
      \node[black!60] at (0.18,0.3) {\small $S_5$};
      \node[black!60] at (0.25,0.21) {\small $S_6$};
      \node at (0.35,0.62) {$\Lambda$};
      \node at (0.95,0.93) {\small $C_1$};
      \node at (0.955,0.83) {\small $C_{\text{-}1}$};
      \node[
        fill=white, 
        fill opacity=0.5, 
        text opacity=1, 
        draw=none, 
        thick,
        rounded corners,
        inner sep=0pt
      ] at (0.58,0.87) {\tiny $\left(v_e(1),w_e(1)\right)$};
      \node at (0.46,0.3) {\tiny $\left(v_e(\text{-}1),w_e(\text{-}1)\right)$};
      \draw[-{Latex[length=3mm,width=2mm]}, thick]
      (0.36,0.2) -- (0.3,0.25);
      \draw[-{Latex[length=3mm,width=2mm]}, thick]
      (0.4,0.4) -- (0.32,0.35);

    \end{scope}
    \node at (3.6cm,-0.3cm) {\Large $v$ axis};
    \node at (11.1cm,-0.3cm) {\Large $v$ axis};
    \node[rotate=90] at (-0.4cm, 2.5cm) {\Large $w$ axis};
    \end{tikzpicture}
    \caption{Region $H$ used in the proof of Theorem~\ref{thm_activation_piecewise_constant}. The second picture is a zoomed-in picture near the point $(v_e(-1), w_e(-1))$}
    \label{fig:figure4_construction_H}
\end{figure}

Next, we show that the region $H$ satisfies our desired property. 
Recall that, by definition of $t^*$,  ${\rm sign}(\cos(\eta t)) = 1$ for $t\in (t^*,t^* + \pi/\eta)$. 
Let $t^*+T_0$ be the 
exit time from $H$ of the integral line of $G_1$ starting from 
$(v(t^*), w(t^*))$, that is,
$T_0=\inf\{t>0 \mid e^{t G_1} (v(t^*), w(t^*))\notin H\}$. 
We are left to prove that $T_0< \pi/\eta$ for $\eta$ small enough. 

Notice that the vector field $G_1$ points towards the interior of $H$ along $S_1$ (since $g_2>0$ there),  $S_4$ (since $g_2<0$ there), $S_5$ (since $g_1(\cdot,1)>0$ there), and $S_6$ (since the slope of $G_1$ is less negative than that of $G_{-1}$ there). 
The same is true also for $S_3$, up to further reducing $\varepsilon_0$. 
Hence, every integral curve of $G_1$ starting from the region $H$ can exit it only by crossing $S_2$. Since the region $H$ does not intersect the cubic $C_1$, we know that the horizontal speed of $G_{1}$ is bounded below in $H$ by a positive constant $\xi=\xi(\varepsilon)$, and therefore, we can estimate the time required to leave  $H$ by
$$T_0 \leq \frac{|p|}{\xi}=:T^*.$$
We deduce that $T^*<\pi/\eta$ for all $\eta<  \pi/T^*=:\eta_0(\varepsilon)$.

\section{Tonic spiking under temporal interference stimulation}\label{section_escaping_condition}

In this section, we explore the case when the input signal is the smooth one resulting from temporal interference ($f(t)=\cos(\eta t)$) and when both
$\eta$ and $\varepsilon$ go to zero and are of the same order. We therefore tightly link the TIS frequency $\eta$ with the time constant $\varepsilon$ of ionic channels. The strong dependence between $\eta$ and $\varepsilon$ needed to elicit a neuronal response has already been observed and described \cite{Mirzakhalili2020,plovieNonlinearitiesTimescalesNeural2025}, and it depends on the potassium-channel time constant. 

More precisely, we set $\eta = \kappa \varepsilon$ for some positive constant $\kappa$, and we study the behavior of \eqref{FHN_intro} for $\varepsilon$ small. For the constructions, we require $(A,B)\in \mathcal{E}_0$. Considering the change of variables $t = \varepsilon^{-1} s$, we get the system 
\begin{equation}\label{limit_system_singular0}
\left\{
\begin{array}{rl}
    \varepsilon \frac{d}{ds} v &= v\left(1-\frac{A^2}{2}-\frac{B^2}{2}- A B \cos(\kappa s)\right) - \frac{v^3}{3} - w,\\
    \frac{d}{ds} w &= v - \gamma w + \beta,
\end{array}
\right.
\end{equation}
which leads,  as $\varepsilon \to 0$, to the singular limit system
\begin{equation}\label{limit_system_singular}
\left\{
\begin{array}{rl}
    0 &= v\left(1-\frac{A^2}{2}-\frac{B^2}{2}- A B \cos(\kappa s)\right) - \frac{v^3}{3} - w,\\
    \frac{d}{ds} w &= v - \gamma w + \beta,
\end{array}
\right.
\end{equation}
that we have begun studying in Section \ref{sec-3.2}.

In order to recover a non-autonomous singular limit system, we can consider the extended system
\begin{equation}\label{limit_system_singular0-extended}
\left\{
\begin{array}{rl}
    \varepsilon \frac{d}{ds} v &= v\left(1-\frac{A^2}{2}-\frac{B^2}{2}- A B \cos(\theta)\right) - \frac{v^3}{3} - w,\\
    \frac{d}{ds} w &= v - \gamma w + \beta,\\
    \frac{d}{ds} \theta &=\kappa, 
\end{array}
\right.
\end{equation}
with singular limit 
\begin{equation}\label{limit_system_singular-extended}
\left\{
\begin{array}{rl}
    0 &= v\left(1-\frac{A^2}{2}-\frac{B^2}{2}- A B \cos(\theta)\right) - \frac{v^3}{3} - w,\\
    \frac{d}{ds} w &= v - \gamma w + \beta,\\
    \frac{d}{ds} \theta &=\kappa.
\end{array}
\right.
\end{equation}

The set of points $(v,w,\theta)$ at which $v\left(1-\frac{A^2}{2}-\frac{B^2}{2}- A B \cos(\theta)\right) - \frac{v^3}{3} - w=0$ is called \emph{critical manifold} in the geometric singular perturbation literature. Near stable normally hyperbolic points of the critical manifold,
the trajectories of \eqref{limit_system_singular0-extended} stay close to those of \eqref{limit_system_singular-extended} \cite{Fenichel1979,Jones95}. Points of the critical manifold at which hyperbolicity is lost are called \emph{fold points} and when trajectories of \eqref{limit_system_singular-extended} approach them, the solutions of \eqref{limit_system_singular0-extended} can exhibit several behaviors, such as fast transitions to other hyperbolic components of the critical manifold \cite{Wec98} or \emph{canard trajectories}, which follow repelling slow manifolds \cite{SW01}.

In terms of the two-dimensional dynamics for $(v,w)$, \eqref{limit_system_singular} is made of an algebraic equation imposing that $(v(s),w(s))$ evolve on the cubic $C_{\cos(\kappa s)}$ and a differential equation on $w$. 
Corresponding to stable normally hyperbolic points of the critical manifold, we can identify a notion of solution to \eqref{limit_system_singular} in some regions of the plane.
Such solutions turn out to be concatenations of integral lines of two vector fields $\Xup$ and $\Xdown$, corresponding to the intervals of time where $s\mapsto \cos(\kappa s)$ is increasing or decreasing, respectively.

Trajectories of \eqref{limit_system_singular0} may approach the boundary of the region where the dynamics of \eqref{limit_system_singular} are well-defined, corresponding to fold points of the extended system. By studying how this can happen, we identify heuristic conditions for the presence of tonic spiking, taking into account the global-in-time nature of the tonic spiking phenomenon. 

To help the reader track which system each claim refers to, we use from now on $(v_\varepsilon,w_\varepsilon)$ to denote the trajectories of \eqref{limit_system_singular0}.
We also prefer to focus on the two-dimensional phase space of \eqref{limit_system_singular0} instead of the three-dimensional one of \eqref{limit_system_singular0-extended} and to provide, when possible, elementary arguments for the description of the behavior of the trajectories.

\subsection{Identification of the vector fields \texorpdfstring{$\Xup$ and $\Xdown$}{X up and X down}}

Consider a time $s_0$ and a point $(v_0,w_0)$ in $C_{\cos(\kappa s_0)}\setminus \{\pm(v_m(\cos(\kappa s_0)),w_m(\cos(\kappa s_0)))\}$. By the inverse function theorem, 
there exists a unique smooth function $s\mapsto (v(s),w(s))$ in a  neighborhood of $s_0$ that satisfies system \eqref{limit_system_singular}.
In more detail, since such a smooth solution satisfies
\[w(s)= v(s)\left(1-\frac{A^2}2-\frac{B^2}2- A B \cos(\kappa s)\right) - \frac{v(s)^3}3,\]
and using the differential equation for $w$, we have
\[v(s)-\gamma w(s)+\beta=\dot v(s)\left(1-\frac{A^2}2-\frac{B^2}2- A B \cos(\kappa s)\right)
+\kappa A B v(s)\sin(\kappa s)
- \dot v(s) v(s)^2,\]
from which we obtain
\[\dot v(s)=\frac{v(s)-\gamma w(s)+\beta-\kappa A B v(s)\sin(\kappa s)}{\left(1-\frac{A^2}2-\frac{B^2}2- A B \cos(\kappa s)\right)-v(s)^2}=\frac{v(s)-\gamma w(s)+\beta-\kappa A B v(s)\sin(\kappa s)}{r(\cos(\kappa s))-v(s)^2}.\]
Notice that the denominator does not vanish as long as $(v(s),w(s))$ stays away from the local maximum and minimum of $C_{\cos(\kappa s)}$. 

In the terminology of geometric singular perturbation theory, a triple $(v,w,\theta)$ such that $(v,w)\in C_{\cos(\theta)}\setminus\{\pm (v_m(\cos(\theta)),w_m(\cos(\theta)))\}$ is a hyperbolic point of the critical manifold for \eqref{limit_system_singular-extended}.

The above discussion justifies the following definition. 

\begin{definition}
A \emph{regular solution of \eqref{limit_system_singular}} is a $C^1$ curve
$I\ni s\mapsto (v(s),w(s))$, where $I$ is a nonempty interval of $\R$,   which stays at positive distance from $J_m\cup(-J_m)$ and satisfies \eqref{limit_system_singular} at every $s\in I$. 
\end{definition}

Consider the function $\mathfrak{c}:(\cup_{c\in[-1,1]}C_c)\setminus \{(0,0)\}\to [-1,1]$  associating with a point $(\bar v,\bar w)$ the unique $c$ such that $(\bar v,\bar w)\in C_c$, that is,
\[\mathfrak{c}(\bar v,\bar w):=\frac{-\bar w+\bar v\left(1-\frac{A^2}2-\frac{B^2}2\right)-\frac{\bar v^3}{3}}{\bar v AB }.\]
Notice also that $\sin(\kappa s)$ is equal to $\sqrt{1-\cos(\kappa s)^2}$ (respectively, $-\sqrt{1-\cos(\kappa s)^2}$) when $\tau\mapsto \cos(\kappa \tau)$ is nonincreasing (respectively, nondecreasing) at $\tau=s$. 
Consider then the vector fields
\[\Xdown(v,w)=\begin{pmatrix}
\frac{v-\gamma w+\beta-\kappa A B v\sqrt{1-\mathfrak{c}(v,w)^2}}{
r(\mathfrak{c}(v,w))
-v^2}\\
v-\gamma w+\beta
\end{pmatrix},\qquad \Xup(v,w)=\begin{pmatrix}
\frac{v-\gamma w+\beta+\kappa A B v\sqrt{1-\mathfrak{c}(v,w)^2}}{
r(\mathfrak{c}(v,w))
-v^2}\\
v-\gamma w+\beta
\end{pmatrix},
\]
that are well defined on $\Omega=(\cup_{c\in[-1,1]}C_c)\setminus \left(\{(0,0)\}\cup J_m \cup -J_m\right)$. Based on the discussion above, we have the following result. 

\begin{lemma}
Let $I$ be an interval of $\R$ and 
$I\ni s\mapsto (v(s),w(s))$ be a regular solution of  \eqref{limit_system_singular}. 
Then $I\ni s\mapsto (v(s),w(s))$ is 
the concatenation of integral curves of $\Xup$ (on subintervals of $I$ where $\sin(\kappa s)\le 0$) and $\Xdown$ (on subintervals of $I$ where $\sin(\kappa s)\ge 0$). 
\end{lemma}
\begin{proof}
The only point that still
needs to be clarified is what happens when $(v(s_0),w(s_0))=(0,0)$ for some $s_0\in I$. At that point $\mathfrak{c}$ stops being well defined, but the implicit function argument is still valid, and $\dot w>0$, so that $s_0$ is isolated among times at which $(v,w)$ passes through the origin. The conclusion then follows. 
\end{proof}

The qualitative phase portraits of $\Xdown$ and $\Xup$ are plotted in Figure~\ref{fig:Xupdown}.
Notice that $\dot w$ is positive below $\Lambda$ and negative above it. This is reflected by the fact that the integral curves of both $\Xdown$ and $\Xup$ point upwards below $\Lambda$ and downwards above it. 
On $J_e$, both $\Xdown$ and $\Xup$ are horizontal: $\Xdown$ points to the right, while $\Xup$ points to the left, as it can be easily deduced from the definition of the two vector fields. 

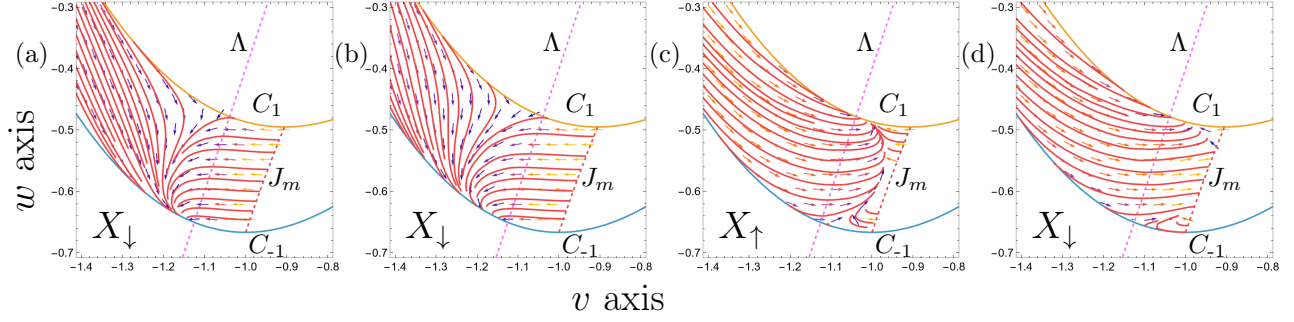
\begin{figure}[ht]
  \centering
  \begin{tikzpicture}
    \node[anchor=south west, inner sep=0cm] (image1) at (0,0)
    {\includegraphics[width=0.22\linewidth]{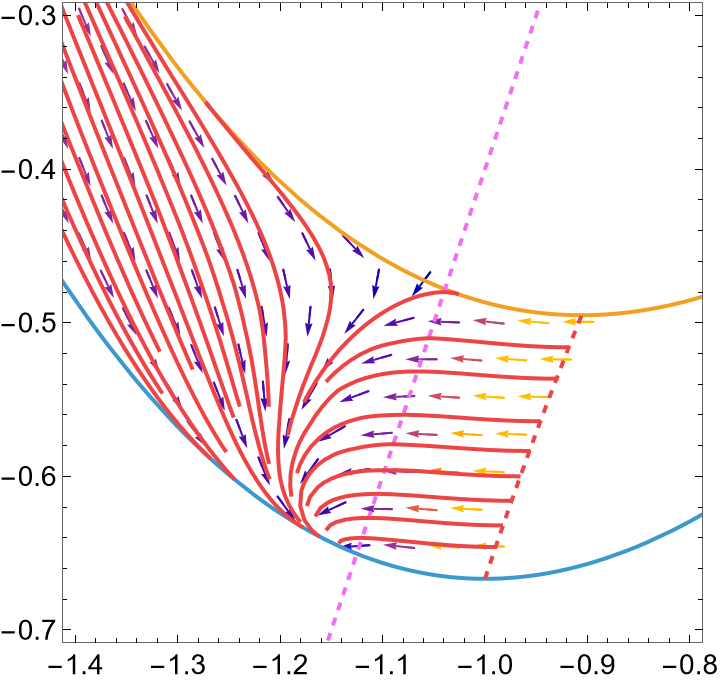}};
    \begin{scope}[
    shift={(image1.south west)},
    x={(image1.south east)}, 
    y={(image1.north west)}]
      \node at (0.65,0.85) {$\Lambda$};
      \node at (0.75,0.62) {$C_1$};
      \node at (0.75,0.105) {$C_{\text{-}1}$};
      \node at (0.22,0.16) {\Large $\Xdown$};
      \node at (0.81,0.35) {$J_m$};
    \end{scope}

    \node[anchor=south west, inner sep=0cm] (image2) at (0.24\linewidth,0)
    {\includegraphics[width=0.22\linewidth]{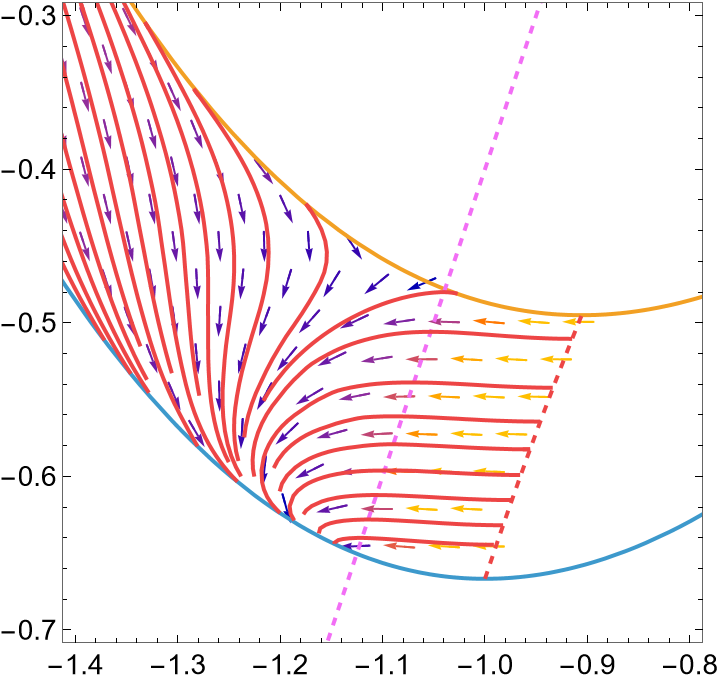}};
    \begin{scope}[
    shift={(image2.south west)},
    x={(image2.south east)}, 
    y={(image2.north west)}]
      \node at (0.65,0.85) {$\Lambda$};
      \node at (0.75,0.62) {$C_1$};
      \node at (0.75,0.105) {$C_{\text{-}1}$};
      \node at (0.22,0.16) {\Large $\Xdown$};
      \node at (0.81,0.35) {$J_m$};
    \end{scope}

    \node[anchor=south west, inner sep=0cm] (image3) at (0.48\linewidth,0)
    {\includegraphics[width=0.22\linewidth]{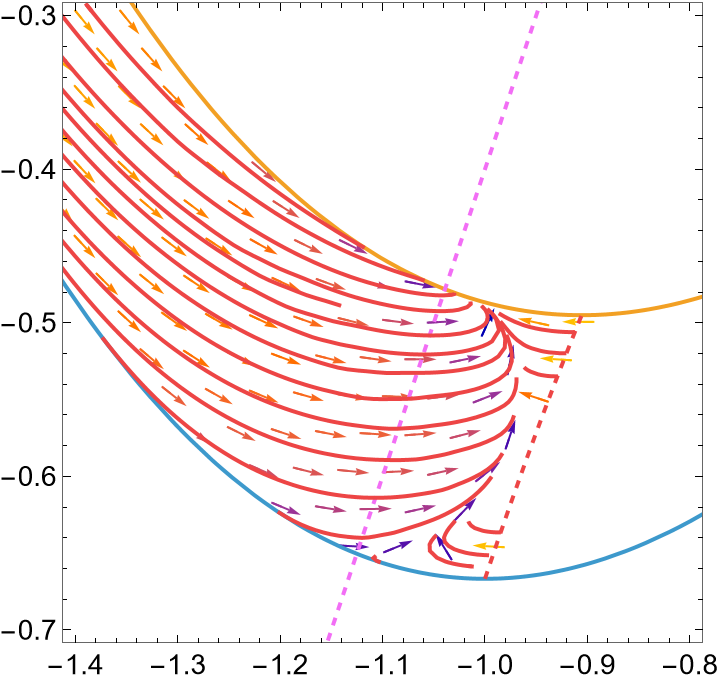}};
    \begin{scope}[
    shift={(image3.south west)},
    x={(image3.south east)}, 
    y={(image3.north west)}]
      \node at (0.65,0.85) {$\Lambda$};
      \node at (0.75,0.62) {$C_1$};
      \node at (0.75,0.105) {$C_{\text{-}1}$};
      \node at (0.22,0.16) {\Large $\Xup$};
      \node at (0.81,0.35) {$J_m$};
    \end{scope}

    \node[anchor=south west, inner sep=0cm] (image4) at (0.72\linewidth,0)
    {\includegraphics[width=0.22\linewidth]{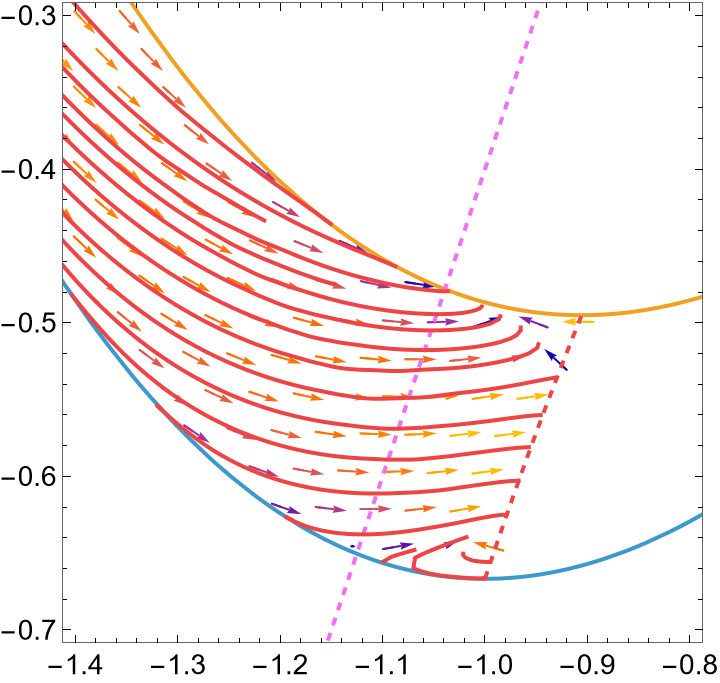}};
    \begin{scope}[
    shift={(image4.south west)},
    x={(image4.south east)}, 
    y={(image4.north west)}]
      \node at (0.65,0.85) {$\Lambda$};
      \node at (0.75,0.62) {$C_1$};
      \node at (0.75,0.105) {$C_{\text{-}1}$};
      \node at (0.22,0.16) {\Large $\Xup$};
      \node at (0.81,0.35) {$J_m$};
    \end{scope}
    \node at (7.5cm,-0.3cm) {\Large $v$ axis};
    \node[rotate=90] at (-0.4cm, 1.5cm) {\Large $w$ axis};
    \node at (-0.015\linewidth, 2.9cm) {(a)};
    \node at (0.23\linewidth, 2.9cm) {(b)};
    \node at (0.47\linewidth, 2.9cm) {(c)};
    \node at (0.71\linewidth, 2.9cm) {(d)};
  \end{tikzpicture}
    
  \caption{Vector plot of $\Xdown$ and $\Xup$ together with some trajectories of the system. (a)-(d) use $\gamma = 0.5$, $\beta = 0.8$, $A= B =0.3$. (a) correspond to $\Xdown$ with $\kappa = 1.3$, (b) $\Xdown$  with $\kappa = 2$, (c) $\Xup$  with $\kappa = 1.3$, (d) $\Xup$  with $\kappa = 2$.
     \label{fig:Xupdown}}
\end{figure}

A useful feature of the vector fields $\Xup$ and $\Xdown$  is the following monotonicity property.
\begin{lemma}[Comparison of trajectories]\label{lemma_comparison}
   Let $\beta,\gamma, \kappa >0$, $(A, B)\in \mathcal{E}_0$. Let $(v_1,w_1), (v_2,w_2)$ be two regular solutions of system \eqref{limit_system_singular} on an interval $(a,b)$.
 Suppose that 
    $v_1(s_0)<v_2(s_0) <v_m(\cos(\kappa s_0))$ for some $s_0\in(a,b)$. Then 
    $$v_1(t) < v_2(t),\quad \forall t\in(a,b).$$ 
\end{lemma}
\begin{proof}
 The proof is an immediate consequence of the fact that the flow of $\Xup$ or $\Xdown$ maps a cubic $C_{c_1}$ into another cubic $C_{c_2}$. Additionally, the trajectories cannot cross because the solution to the corresponding ODE is unique. Finally, because the cubics are 1D manifolds, 
 the evolution preserves the order relation on them.
\end{proof}

\subsection{Relating trajectories\texorpdfstring{ of \eqref{limit_system_singular0} and \eqref{limit_system_singular}}{}}

\subsubsection{Tracking regular trajectories \texorpdfstring{ of \eqref{limit_system_singular} by trajectories of  \eqref{limit_system_singular0}}{}}\label{sec:tracking}

A key fact that makes \eqref{limit_system_singular} an informative limit system for \eqref{limit_system_singular0} is that the trajectories of the latter approximate regular trajectories of the former as $\varepsilon$ goes to zero, under suitable stability conditions.

Intuitively, 
for a trajectory $(v_\varepsilon(\cdot),w_\varepsilon(\cdot))$ of  \eqref{limit_system_singular0},
$|\dot v_\varepsilon(s)|$ is of order $1/\varepsilon$ if $(v_\varepsilon(s),w_\varepsilon(s))$ is not close to $C_{\cos(\kappa s)}$. 
Some branches of $C_{\cos(\kappa s)}$ are attractive for these dynamics, namely those where 
$|v|>|v_m(\cos(\kappa s))|$, while the  portion of $C_{\cos(\kappa s)}$ 
where $|v|<|v_m(\cos(\kappa s))|$
is repulsive. 
As a consequence, 
regular trajectories of \eqref{limit_system_singular} included in the region covered by the attractive branches of the cubics
can be tracked with arbitrary precision by trajectories of  \eqref{limit_system_singular0}.  
 More precisely, classical singular perturbation results (see, e.g., 
 \cite{Jones95,Tihonov1952,Vasileva})
imply the following. 
  
\begin{lemma}\label{lem:tracking}
Let $[s_0,s_1]\ni s\mapsto (v(s),w(s))$
be a regular solution of \eqref{limit_system_singular}
satisfying $v(s)<v_m(\cos(\kappa s))$ for every $s\in [s_0,s_1]$.
Let $(v_0,w_0)$ be such that 
$w_0=w(s_0)$ and $v_0<v_m(\cos(\kappa s_0))$. 
For every $\varepsilon>0$ let  $[s_0,s_1]\ni s\mapsto (v_\varepsilon(s),w_\varepsilon(s))$ be the trajectory of  \eqref{limit_system_singular0} with initial condition 
$(v_0,w_0)$.
Then 
$v_\varepsilon(s)<0$ for every $s\in [s_0,s_1)$ for $\varepsilon$ small enough
and, given any $\tilde s_0\in (s_0,s_1)$, 
$\sup_{s\in [\tilde s_0,s_1]}|(v_\varepsilon(s),w_\varepsilon(s))-(v(s),w(s))|\to 0$ as $\varepsilon\to 0$, locally uniformly with respect to $s_0$ (determining $w_0$ and the choice of the regular trajectory) and $v_0$.
\end{lemma}

\subsubsection{Behavior near \texorpdfstring{$J_m$}{Jm}}

We consider here what happens when a regular solution of \eqref{limit_system_singular}  approaches $J_m$. 

Consider $(v_0,w_0)=(v_m(\cos(\kappa s_0)),w_m(\cos(\kappa s_0)))$. The quantity $v_0-\gamma w_0+\beta$ would  give the derivative of $w$ at $s_0$ if system \eqref{limit_system_singular} were satisfied with initial condition $(v(s_0),w(s_0))=(v_0,w_0)$. 
Another quantity that should be taken into account is the derivative of $s\mapsto w_m(\cos(\kappa s))$ at $s=s_0$. 
If the latter derivative is smaller than the former, i.e., if either $s\mapsto \cos(\kappa s)$ is decreasing at $s_0$ or if (cf.~\eqref{eq-vm-wm})
\begin{equation}\label{nonescaping_condition}
v_m(\cos(\kappa s_0)) A B \kappa \sin(\kappa s_0) < v_m(\cos(\kappa s_0)) - \gamma w_m(\cos(\kappa s_0)) + \beta,
\end{equation}
then we can associate with the same initial condition $(v(s_0),w(s_0))=(v_0,w_0)$ two distinct solutions of \eqref{limit_system_singular}, corresponding to the two branches of $C_{\cos(\kappa s_0)}$ that originate from $(v_0,w_0)$. If, instead, 
\begin{equation}\label{escaping_condition}
    v_m(\cos(\kappa s_0)) A B \kappa \sin(\kappa s_0) > v_m(\cos(\kappa s_0)) - \gamma w_m(\cos(\kappa s_0)) + \beta,
\end{equation}
then no continuous solution of \eqref{limit_system_singular} with $(v(s_0),w(s_0))=(v_0,w_0)$ exists. For $\varepsilon$ small, we expect that the solution $(v_\varepsilon(\cdot),w_\varepsilon(\cdot))$ of \eqref{limit_system_singular0} with initial condition $(v_\varepsilon(s_0),w_\varepsilon(s_0))=(v_0,w_0)$ joins almost horizontally the vertical axis, leading to an action potential.  
The following lemma formalizes this intuition.

\begin{lemma}[Local escaping]\label{lem:escaping}
Assume that $\cos(\kappa s)$ is increasing at $s=s_0$ and that \eqref{escaping_condition} holds. 
Then there exists a neighborhood $Q$ of $(v_m(\cos(\kappa s_0)),w_m(\cos(\kappa s_0)))$ such that for $\varepsilon$ small enough and 
for every $(v_0,w_0)\in Q$ the solution to \eqref{limit_system_singular0} with initial condition $(v_\varepsilon(s_0),w_\varepsilon(s_0))=(v_0,w_0)$ undergoes an action potential after $s_0$.     
\end{lemma}
\begin{proof}
We are going to take as neighborhood of $(v_m(\cos(\kappa s_0)),w_m(\cos(\kappa s_0)))$ a square of the type  
\[Q_\delta:=(v_m(\cos(\kappa s_0))-\delta,v_m(\cos(\kappa s_0))+\delta)\times (w_m(\cos(\kappa s_0))-\delta,w_m(\cos(\kappa s_0))+\delta),\]
for some $\delta>0$ small enough. 

\begin{figure}[ht]
    \centering
    \begin{tikzpicture}
    \node[anchor=south west, inner sep=0cm] (image1) at (0,0)
    {\includegraphics[width=0.55\linewidth]{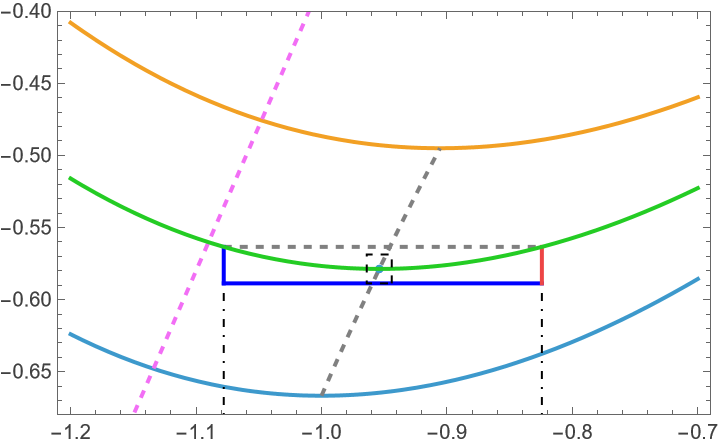}};
    \begin{scope}[
    shift={(image1.south west)},
    x={(image1.south east)}, 
    y={(image1.north west)}]
      \node at (0.52,0.19) {\small $J_m$};
      \node[
        fill=white, 
        fill opacity=0.6, 
        text opacity=1, 
        draw=none, 
        thick,
        rounded corners,
        inner sep=1pt
      ] at (0.34,0.2) {\small $v_1$};
      \node[
        fill=white, 
        fill opacity=0.6, 
        text opacity=1, 
        draw=none, 
        thick,
        rounded corners,
        inner sep=1pt
      ] at (0.78,0.16) {\small $v_2$};
    
      \draw[-{Latex[length=3mm,width=2mm]}, thick]
      (0.6,0.5) -- (0.55,0.42);
      
      \node at (0.62,0.55) {\small $Q_\delta$};
      \node at (0.44,0.84) {$\Lambda$};
      \node at (0.22,0.420) {\tiny $\left(v_1,\bar{w}\right)$};
      \node at (0.83,0.42) {\tiny $\left(v_2,\bar{w}\right)$};
      \node at (0.13,0.83) {$C_1$};
      \node at (0.19,0.6) {$C_{\cos(\kappa s_0)}$};
      \node at (0.13,0.28) {$C_{\text{-}1}$};      
    \end{scope}
    \node at (4.5cm,-0.3cm) {\Large $v$ axis};
    \node[rotate=90] at (-0.4cm, 2.7cm) {\Large $w$ axis};
    \end{tikzpicture}
    \caption{Regions $Q$ and $Q_\delta$ in the proof of Lemma \ref{lem:escaping}}
    \label{fig:figure6_5_regionQ}
\end{figure}
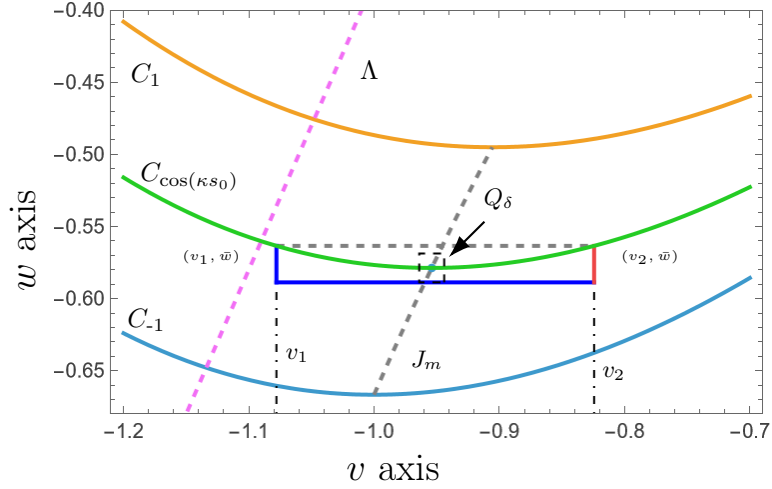

Consider $\bar w$ in the interval $(w_m(\cos(\kappa s_0)),w_m(1))$ and set $\bar s:=\min\{s>s_0\mid w_m(\cos(\kappa s))=\bar w\}$. Let $v_1<v_2<0$ be such that $(v_1,\bar w)$ and $(v_2,\bar w)$ are the two intersection points of the cubic 
$C_{\cos(\kappa s_0)}$ with the horizontal line $\{w=\bar w\}$ that lie in the half-plane $\{v<0\}$. Let, moreover, $\delta$ be such that $w_m(\cos(\kappa s_0))+\delta<\bar w$ (See Figure \ref{fig:figure6_5_regionQ}).

Notice that for every $\varepsilon>0$ and every $s\in [s_0,\bar s]$, the vector $(\frac{d}{ds}v,\frac{d}{ds}w)$ obtained from \eqref{limit_system_singular0} points to the right when computed on the vertical segment $\{v_1\}\times (w_m(\cos(\kappa s_0))-\delta,\bar w)$ and to the top when computed on  the horizontal segment $(v_1,v_2)\times \{w_m(\cos(\kappa s_0))-\delta\}$.

Thanks to \eqref{escaping_condition}, up to taking $\bar w$ close enough to $w_m(\cos(\kappa s_0))$ and $\delta$ small enough, we can assume that there exists $\nu>1$ such that 
\[v_m(\cos(\kappa s))AB \kappa \sin(\kappa s)\ge \nu (v-\gamma w+\beta)  \]
for every $(v,w)\in [v_1,v_2]\times [w_m(\cos(\kappa s_0))-\delta,\bar w]$ and every $s\in [s_0,\bar s]$.

We deduce
that if a trajectory $(v_\varepsilon(\cdot),w_\varepsilon(\cdot))$ of \eqref{limit_system_singular0} stays in 
$[v_1,v_2]\times [w_m(\cos(\kappa s_0))-\delta,\bar w]$ for all times $s$ in the interval $[s_0,\bar s]$ then 
\begin{align}
    \bar w-w_m(\cos(\kappa s_0)) 
    &=\int_{s_0}^{\bar s}\frac{d}{ds} w_m(\cos(\kappa s))ds
    =\int_{s_0}^{\bar s} 
    v_m(\cos(\kappa s))AB \kappa \sin(\kappa s)ds\nonumber\\
   &
    \ge 
    \nu \int_{s_0}^{\bar s}(v_\varepsilon(s)-\gamma w_\varepsilon(s)+\beta)ds =\nu \int_{s_0}^{\bar s}\frac{d}{ds}w_\varepsilon(s)ds
    =\nu (w_\varepsilon(\bar s)-w_\varepsilon(s_0)).\label{eq:fornu}
\end{align}
Let $\delta$ be small enough so that 
\[ \frac{\bar w-w_m(\cos(\kappa s_0))}{\nu}\le \bar w-w_m(\cos(\kappa s_0))-2\delta.\]
As a consequence,  we can conclude from \eqref{eq:fornu} that the solution to \eqref{limit_system_singular0} with initial condition $(v_\varepsilon(s_0),w_\varepsilon(s_0))\in Q_\delta$ either stays in $[v_1,v_2]\times [w_m(\cos(\kappa s_0))-\delta,\bar w]$ for all $s\in [s_0,\bar s]$, and in that case satisfies $w_\varepsilon(\bar s)\le \bar w-\delta$, or exits $[v_1,v_2]\times [w_m(\cos(\kappa s_0))-\delta,\bar w]$ before time $\bar s$, and in that case it does so through the segment $\{v_2\}\times [w_m(\cos(\kappa s_0))-\delta,\bar w]$.

In both cases, we can repeat the argument for the  construction of the region $H$  in the proof of Theorem~\ref{thm_activation_piecewise_constant}, deducing that for $\varepsilon$ small enough (independent of the specific choice of initial condition in $Q_\delta$), the trajectory 
$(v_\varepsilon,w_\varepsilon)$ undergoes  an action potential after $s_0$. 
\end{proof}

\begin{remark}  \label{rmk:Xup-left}  When either of the conditions~\eqref{nonescaping_condition} and~\eqref{escaping_condition} is satisfied, it is said that the triple $(v_m(\cos(\theta)),w_m(\cos(\theta)),\theta)$ with $\theta=\kappa s_0$ is a \emph{regular fold point} for \eqref{limit_system_singular-extended}. In this case, it is known that solutions reaching the curve of fold points necessarily jump \cite{Wec98}. In 
terms of the phase portrait of $\Xup$, we can notice that the denominator in the first component of $\Xup$ goes to zero and is negative as $(v,w)$ approaches $J_m$ from the left (since $v<v_m(c)$ implies that $v^2>v_m(c)^2=r(c)$). 
The numerator, meanwhile, is positive (respectively, negative) at points of $J_m$ where \eqref{nonescaping_condition} (respectively, \eqref{escaping_condition}) holds. 
Hence, regular trajectories on the left of $J_m$ are almost horizontal and point to the left (respectively, right) near points where \eqref{nonescaping_condition} (respectively, \eqref{escaping_condition}) holds.
\end{remark}

\subsection{No tonic spiking under a non-escaping condition: Proof of Theorem \ref{theorem_results_singular_system}}\label{sec-nontonic}~\\
Lemma~\ref{lemma_non_spiking} below is a reformulation of Theorem~\ref{theorem_results_singular_system} that uses the notation introduced in this section. Its proof is now straightforward, based on the analysis of the regular trajectories of the limit system discussed above and the comparison result from Lemma~\ref{lemma_comparison}.

\begin{figure}[ht]
  \centering
    \centering
    \begin{tikzpicture}
      \node[anchor=north east, inner sep=0cm] (image1) at (0.3\linewidth,10cm)
        {\includegraphics[width=0.3\linewidth]{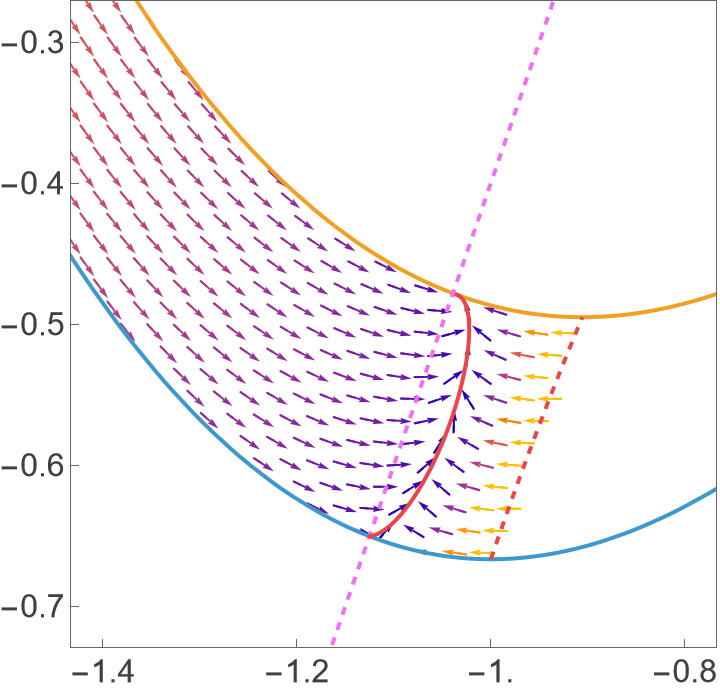}};
      \begin{scope}[
        shift={(image1.south west)},
        x={(image1.south east)}, 
        y={(image1.north west)}]
        \node at (0.65,0.85) {$\Lambda$};
        \node at (0.75,0.62) {$C_1$};
        \node at (0.75,0.15) {$C_{\text{-}1}$};
        \node at (0.22,0.16) {\Large $\Xup$};
        \node at (0.8,0.35) {$J_m$};
        \node at (0.9,0.9) {(a)};
      \end{scope}
      \node[anchor=north east, inner sep=0cm] (image2) at (0.62\linewidth,10cm)
        {\includegraphics[width=0.3\linewidth]{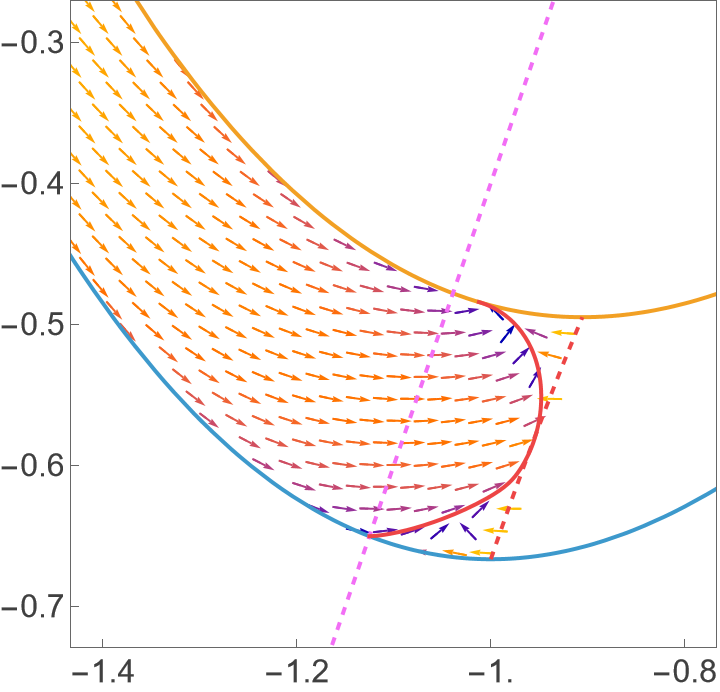}};
      \begin{scope}[
        shift={(image2.south west)},
        x={(image2.south east)}, 
        y={(image2.north west)}]
        \node at (0.65,0.85) {$\Lambda$};
        \node at (0.75,0.62) {$C_1$};
        \node at (0.75,0.15) {$C_{\text{-}1}$};
        \node at (0.22,0.16) {\Large $\Xup$};
        \node at (0.8,0.35) {$J_m$};
        \node at (0.9,0.9) {(b)};
      \end{scope}
      \node[anchor=north east, inner sep=0cm] (image3) at (0.93\linewidth,10cm)
        {\includegraphics[width=0.31\linewidth]{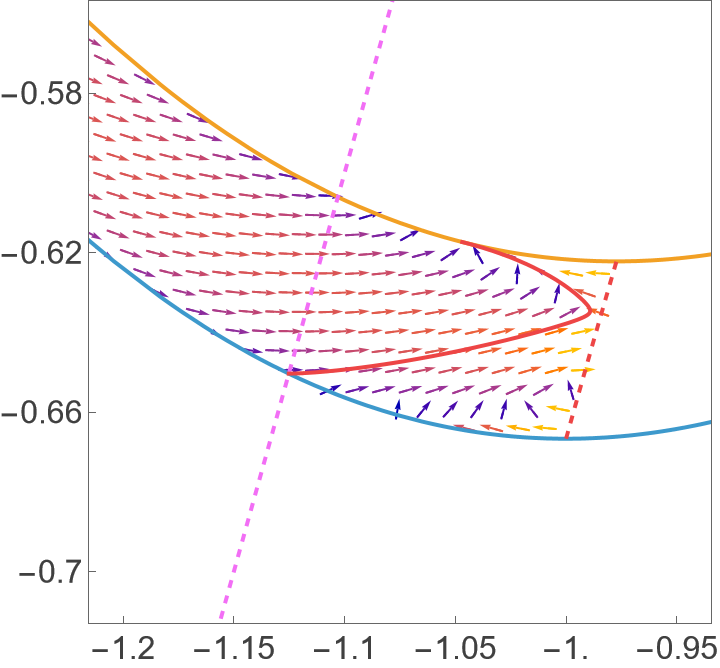}};
      \begin{scope}[
        shift={(image3.south west)},
        x={(image3.south east)}, 
        y={(image3.north west)}]
        \node at (0.45,0.85) {$\Lambda$};
        \node at (0.65,0.70) {$C_1$};
        \node at (0.65,0.3) {$C_{\text{-}1}$};
        \node at (0.22,0.16){\Large $\Xup$};
        \node at (0.9,0.45) {$J_m$};
        \node at (0.9,0.9) {(c)};
      \end{scope}
      \node[anchor=north east, inner sep=0cm] (image4) at (0.3\linewidth,5cm)
        {\includegraphics[width=0.3\linewidth]{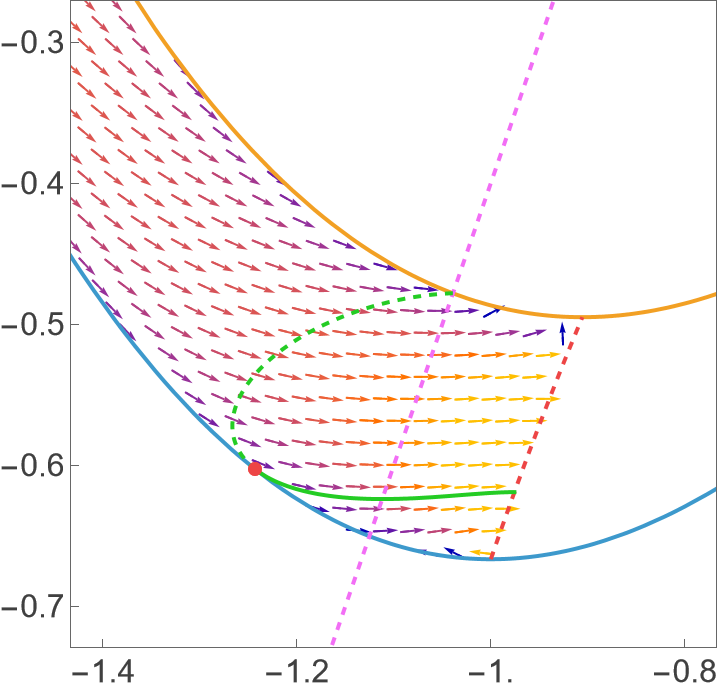}};
      \begin{scope}[
        shift={(image4.south west)},
        x={(image4.south east)}, 
        y={(image4.north west)}]
        \node at (0.65,0.85) {$\Lambda$};
        \node at (0.75,0.62) {$C_1$};
        \node at (0.75,0.15) {$C_{\text{-}1}$};
        \node at (0.22,0.16) {\Large $\Xup$};
        \node at (0.3,0.3) {$P$};
        \node at (0.8,0.35) {$J_m$};
        \node at (0.9,0.9) {(d)};
      \end{scope}
      \node[anchor=north east, inner sep=0cm] (image5) at (0.62 \linewidth,5cm)
        {\includegraphics[width=0.305\linewidth]{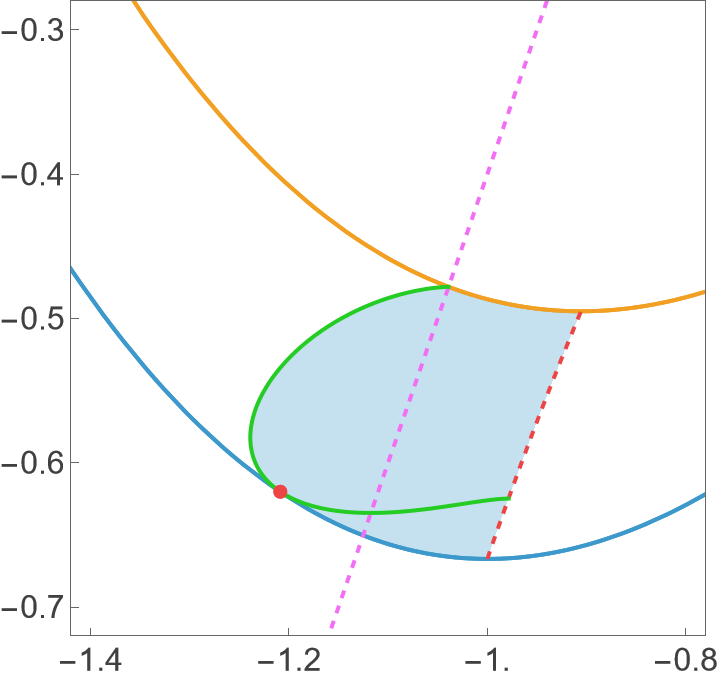}};
      \begin{scope}[
        shift={(image5.south west)},
        x={(image5.south east)}, 
        y={(image5.north west)}]
        \node at (0.65,0.85) {$\Lambda$};
        \node at (0.75,0.62) {$C_1$};
        \node at (0.75,0.14) {$C_{\text{-}1}$};
        \node at (0.8,0.35) {$J_m$};
        \node at (0.35,0.25) {$P$};
        \node at (0.6,0.4) {\Large $\mathcal{E}$};
        \node at (0.9,0.9) {(e)};
      \end{scope}

    \node[rotate=90] at (-0.5cm, 2.5cm) {\Large $w$ axis};
    \node[rotate=90] at (-0.5cm, 7.5cm) {\Large $w$ axis};
    \node at (2.5cm, 0cm) {\Large $v$ axis};
    \node at (7.5cm, 0cm) {\Large $v$ axis};
    \node at (12.2cm, 5cm) {\Large $v$ axis};
    \end{tikzpicture}

  \caption{(a) - (d) Vector plot of the field $\Xup$. All cases use $\gamma = 0.5$, $\beta = 0.8$. 
  (a) $A= B =0.3$, $\kappa=0.5$, 
  (b) $A= B =0.3$, $\kappa = 1.58$, 
  (c) $A=B=0.15$, $\kappa = 6.4$, 
  (d),(e) $A=B=0.3$, $\kappa = 3$. (a) - (c) satisfy the conditions of Theorem \ref{theorem_results_singular_system}, whereas (d) and (e) satisfy Assumption (E).}
  \label{fig:nonescaping}    
\end{figure}
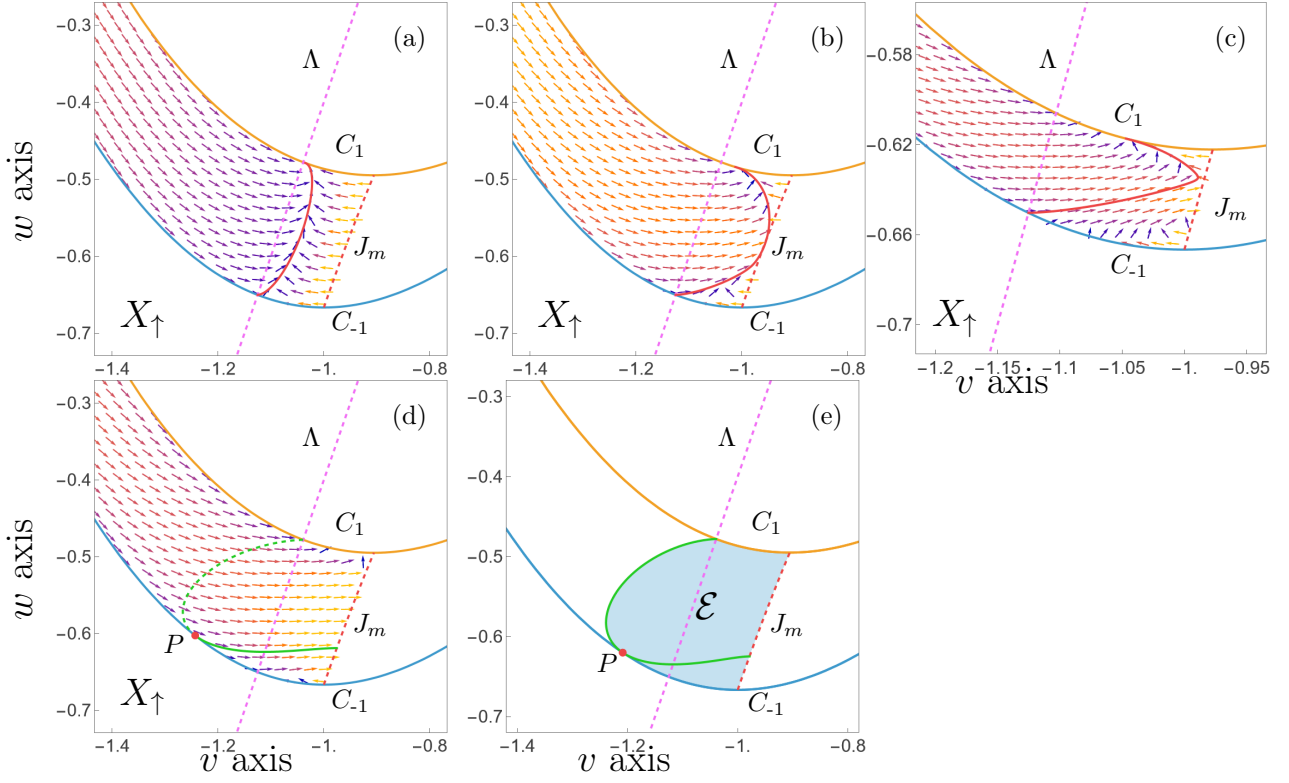

\begin{lemma}\label{lemma_non_spiking}
    Let $(v,w)$ be a maximal integral curve of the vector field $\Xup$ with initial condition  $(v_e(-1),w_e(-1))$. Suppose that $(v,w)$ reaches $C_1$ before reaching $J_m$. Then there exists $\varepsilon_0>0$ such that for all $\varepsilon\in (0, \varepsilon_0)$ no solution  of \eqref{FHN_intro} undergoes tonic spiking.
\end{lemma}
\begin{proof}
By Lemma~\ref{lem_region_many_visits}, every trajectory of \eqref{FHN_intro}, and hence of \eqref{limit_system_singular0}, ends up in a neighborhood of $J_e$. Lemma~\ref{lemma_comparison} and the hypothesis on the integral curve of $\Xup$ with initial condition $(v_e(-1),w_e(-1))$ guarantee that any regular trajectory of \eqref{limit_system_singular} starting from $J_e$  can be extended forward in time indefinitely without intersecting $J_m$.
Lemma~\ref{lem:tracking} then allows us to conclude.
\end{proof}

\begin{corollary}\label{cor:noescaping}
For $\kappa$  small enough  there exists $\varepsilon_0>0$ such that for all $\varepsilon\in (0, \varepsilon_0)$ no solution of \eqref{FHN_intro} undergoes tonic spiking.
\end{corollary}
\begin{proof} 
For $\kappa$ small enough, condition \eqref{nonescaping_condition} is satisfied along the entire curve $J_m$. Let $(v,w)$ be the maximal integral curve of the vector field $\Xup$ with initial condition $(v_e(-1),w_e(-1))$. Since $\Xup$ points left near $J_m$ (Remark~\ref{rmk:Xup-left}), the trajectory $(v,w)$ stays at a positive distance from $J_m$. We conclude using Lemma~\ref{lemma_non_spiking}.
\end{proof}

Examples where the assumptions of Lemma \ref{lemma_non_spiking} are satisfied are illustrated in Figure~\ref{fig:nonescaping} a), b), c). 

Another consequence of the assumptions of Lemma~\ref{lemma_non_spiking} is the existence of periodic trajectories contained in the region bounded by $C_1$, $C_{-1}$, $\Gamma_1= \{ \exp(t\Xup)(v_e(-1),w_e(-1))\mid t\in[0,\pi/\kappa]\}$, $\Gamma_2= \{ \exp(t\Xdown)(v_e(1),w_e(1))\mid t\in[0,\pi/\kappa]\}$. This is a direct consequence of the intermediate value theorem applied to  $\exp(\frac{\pi}{\kappa}\Xup) \exp(\frac{\pi}{\kappa}\Xdown): C_1 \to C_1$.

More generally, it may happen that trajectories of $\Xup$ from $J_e$ never reach $J_m$, even if condition \eqref{escaping_condition} is satisfied for some $s_0$ (cf.~the right of Figure~\ref{fig:nonescaping}c). 
According to Lemma~\ref{lemma_non_spiking}, also in this case, we observe non-tonic spiking for $\varepsilon$ sufficiently small.

\subsection{Heuristics for tonic spiking under an escaping condition}\label{sec:heuristics}

In this section, we discuss a positive spiking result: we identify a condition under which heuristic reasoning yields that most trajectories of \eqref{limit_system_singular0} undergo tonic spiking. Numerical evidence supporting the heuristics is provided in Section~\ref{sec:numerical_experiments}.

Let us introduce Assumption (E), where the letter E stands for \emph{escaping}. 

\begin{esc_cond}
Denoting by $P$ the endpoint on $C_{-1}$ of the integral curve of $X_{\downarrow}$ with initial condition  $(v_e(1),w_e(1))$, we  assume that the integral curve of $X_{\uparrow}$ with initial condition $P$ reaches a point of $J_m$ where \eqref{escaping_condition} is satisfied. 
\end{esc_cond}
The situation described in Assumption~(E) is illustrated in Figure~\ref{fig:nonescaping}d. 

Suppose Assumption~(E) holds. Denote by ${\cal E}$ the bounded region delimited by $C_1$, $C_{-1}$, $J_m$, and the arc $\Gamma$, corresponding to the integral curve of $\Xdown$ connecting $(v_e(1),w_e(1))$ and $P$ (see Figure~\ref{fig:nonescaping}e).
Then 
\begin{itemize}
\item integral lines of $\Xup$ with initial condition in $C_{-1}$ at the right of $P$ end up in $J_m$ before touching $C_1$, as a consequence of Lemma~\ref{lemma_comparison};

\item again by Lemma~\ref{lemma_comparison}, integral lines of $\Xdown$ with initial condition in $\mathcal{E}$ end up in $C_{-1}$ at the right of $P$;

\item since $\Xup$ is inward pointing along $\Gamma$, integral lines of $\Xup$ with initial condition in $\Gamma$ cannot leave $\mathcal{E}$ crossing $\Gamma$;

\item we deduce that a regular solution of \eqref{limit_system_singular} starting in ${\cal E}$  finally ends up in $J_m$ (with at most two switches between $\Xup$ and $\Xdown$); 
\item the endpoints on $J_m$ of regular trajectories of \eqref{limit_system_singular} starting in ${\cal E}$ 
cannot  satisfy
\eqref{nonescaping_condition}, since otherwise the last piece of the regular trajectory, which is necessarily an integral curve of $\Xup$, could not have reached $J_m$ (Remark~\ref{rmk:Xup-left}). 
Since the set of $c\in [-1,1]$ such that $v_m(c) A B \kappa \sqrt{1-c^2} =v_m(c) - \gamma w_m(c) + \beta$ has measure zero, we deduce that for almost every initial condition in $\mathcal{E}$, the corresponding regular trajectory of \eqref{limit_system_singular} ends up in $J_m$ at a point where  \eqref{escaping_condition} holds.
\end{itemize}

Given an initial condition in $\mathcal{E}$ such that the corresponding regular trajectory of \eqref{limit_system_singular} ends up in $J_m$ at a point where  \eqref{escaping_condition} holds, we can apply Lemmas~\ref{lem:tracking} and \ref{lem:escaping} to deduce that for $\varepsilon>0$ small enough, the trajectory of \eqref{limit_system_singular0} with the same initial condition undergoes an action potential. 
Combining this observation with Lemma~\ref{lem_region_many_visits},
we can expect that, for $\varepsilon>0$ small enough,  most trajectories of \eqref{limit_system_singular0} undergo tonic spiking. 

We call the argument \emph{heuristic} since we cannot guarantee analytically that the value of $\varepsilon$ coming from Lemma~\ref{lem:escaping} is uniform when the initial condition in $\mathcal{E}$ approaches the exceptional regular trajectories of \eqref{limit_system_singular} reaching $J_m$ at points where \eqref{escaping_condition} is not satisfied (i.e., at non-regular fold points). But since the extension of those regular trajectories beyond $J_m$ is unstable, we can conjecture that initial conditions for trajectories of \eqref{limit_system_singular0} that are not tonically spiking are of measure zero. 
The main difficulty in rendering such a reasoning rigorous is that the properties of the solutions of 
\eqref{limit_system_singular0-extended} 
that we can infer (for $\varepsilon$ small) from 
those of 
\eqref{limit_system_singular-extended} using singular perturbation arguments typically are valid on compact time intervals, while tonic spiking is a global-in-time property.

In the next section, we present numerical evidence for our conjecture, illustrating, in particular, the robustness of the tonic spiking pattern with respect to different choices of parameters for which Assumption~(E) is satisfied. 

\section{Numerical experiments}\label{sec:numerical_experiments}
To evaluate the criterion derived in the preceding analysis, we conduct two experiments to assess tonic activation of system \eqref{FHN_intro} as stimulus parameters are varied, while satisfying Assumption (E). The numerical integrations and spike detection were performed in Julia 1.13.1, while the figures were generated in Wolfram Mathematica, Version 14.3.

In our experiments, we distinguish between two long-time dynamical behaviors. Figure~\ref{fig:representative_trajectories} shows representative trajectories in non-spiking and spiking regimes. We selected these examples from the parameter sets used in the subsequent experiments and simulated them over a sufficiently long interval to distinguish transient activation from repeated spiking.

\begin{figure}
    \centering
    \begin{tikzpicture}
      \node[anchor=north east, inner sep=0cm] (image1) at (0,0)
        {\includegraphics[width=0.8\linewidth]{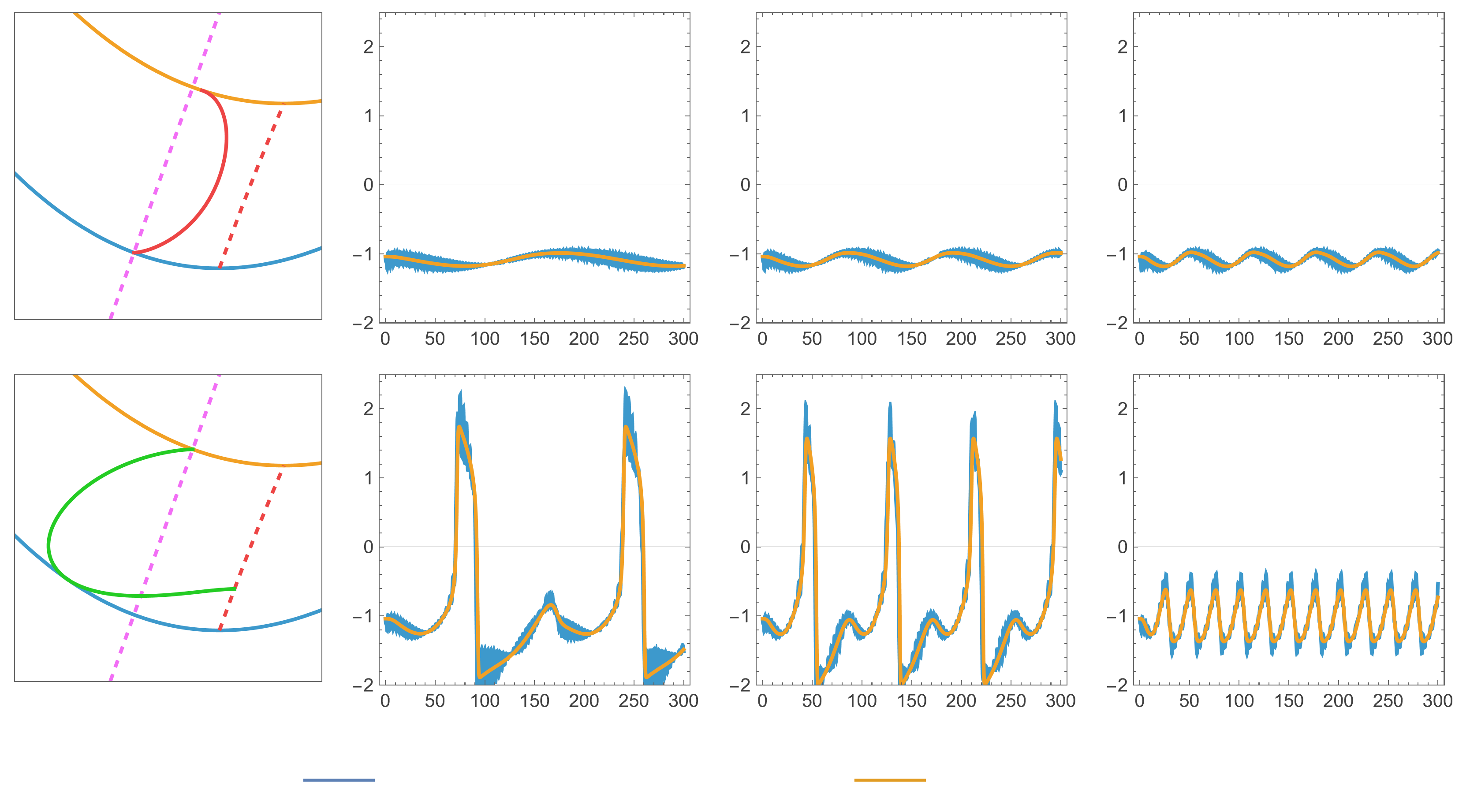}};
      \begin{scope}[
        shift={(image1.south west)},
        x={(image1.south east)}, 
        y={(image1.north west)}]
        \node at (0.06,0.65) {$\Lambda$};
        \node at (0.17,0.91) {$C_1$};
        \node at (0.17,0.63) {$C_{\text{-}1}$};
        \node at (0.19,0.75) {$J_m$};

        \node at (0.06,0.205) {$\Lambda$};
        \node at (0.17,0.465) {$C_1$};
        \node at (0.17,0.185) {$C_{\text{-}1}$};
        \node at (0.19,0.305) {$J_m$};

        \node at (0.36,0.1) {time};
        \node at (0.62,0.1) {time};
        \node at (0.87,0.1) {time};

        \node at (0.385,0.04) {Demodulated TIS};
        \node at (0.705,0.04) {Averaged};
        
      \end{scope}
    \end{tikzpicture}
    \caption{Representative trajectories illustrating the non-escaping and escaping regimes. In all panels, $\gamma=0.5$, $\beta=0.8$, and $A=B=0.3$. The top and bottom rows correspond to $\kappa=1$ and $\kappa=2.5$, respectively. The first column shows the singular-limit trajectories used to test the geometric criteria: condition~\eqref{nonescaping_condition} holds for $\kappa=1$, whereas condition~\eqref{escaping_condition} holds at some point of $J_m$ for $\kappa=2.5$, consistent with Assumption~(E). Columns 2--4 correspond to $\varepsilon=0.03$, $0.06$, and $0.1$, with $\eta=\kappa\varepsilon$. The light-blue curves show the demodulated voltage $V(t)-A\sin(\omega_1t)-B\sin(\omega_2t)$ of the non-averaged system~\eqref{full_FHN}, with $\omega_1=2$ and $\omega_2=\omega_1+\eta$, while the orange curves show the voltage component of the averaged system~\eqref{general_switching_system} with $f(t)=\cos(\eta t)$. No action potentials are observed in the top row. In the bottom row, repeated spiking is observed for $\varepsilon=0.03$ and $0.06$, but not for $\varepsilon=0.1$, consistent with the small-$\varepsilon$ character of the theoretical predictions. The demodulated and averaged trajectories remain close in all cases.}
    \label{fig:representative_trajectories}
\end{figure}

\subsection{Experiment 1} 

In our first experiment, we study tonic spiking of system \eqref{FHN_intro} with $\eta = \kappa\varepsilon$. We fix $\beta = 0.8$ and $\gamma = 0.5$, and perform eight parameter sweeps, corresponding to $A=B\in\{0.15,0.20,\ldots,0.50\}$. In each sweep, $\kappa$ ranges over $[0.2, 12]$ with step $\Delta \kappa = 0.2$, and $\varepsilon$ ranges over $[0.005,0.265]$ with step $\Delta \varepsilon = 0.005$.

All trajectories start from the initial condition $(0,w_e(1))$, from which the solution performs an initial action potential over the whole parameter range considered, and are simulated over the time interval $[0,T]$ with $T=\max(2000, 6\pi/\eta)$. Following Definition~\ref{def-spike}, we count the action potentials as the connected components of $\{t \in [0, T] : v(t) \geq 0\}$, that is, as the excursions of $v$ into the region $v \geq 0$. Since the initial condition lies on $\{v = 0\}$, the initial action potential is always included in this count. Tonic spiking, as introduced in Definition~\ref{def-spike}, cannot be assessed over a finite time interval. Accordingly, for the purposes of the numerical classification, we label a trajectory as tonic spiking if at least two action potentials are detected during $[0,T]$, i.e., if at least one action potential follows the initial one. Since, according to the escaping mechanism of Section~\ref{section_escaping_condition}, action potentials are triggered at specific phases of the envelope, the relevant time scale is its period $2\pi/\eta$; the simulation time is therefore chosen so that every trajectory is observed over at least three envelope periods. The bound $6\pi/\eta$ exceeds $2000$ only for $\eta < 0.0094$, a small region in the lower-left corner of each panel.

To test whether the resulting picture persists under the original two-frequency forcing \eqref{interferential_input}, we repeat the experiment for the full FitzHugh--Nagumo system \eqref{full_FHN} with the same parameter values and initial conditions, and with carrier frequencies $\omega_1\in \{2, 3, 4, 10, 20\}$ and $\omega_2 = \omega_1 + \eta$. In this case, action potentials are detected on the demodulated voltage $\tilde{v} = V-A\sin(\omega_1 t) - B\sin(\omega_2 t)$. To suppress residual oscillations at the carrier frequency, we first smooth $\tilde{v}$ with a Gaussian kernel of standard deviation $1$ time unit ($5$ samples at the sampling interval $\Delta t = 0.2$ used in the simulations). Action potentials are then counted, as for the averaged system, as the excursions of the smoothed signal into the region $\tilde{v}\geq 0$, where an excursion is considered to end once $\tilde{v}$ has returned below $-0.5$. The classification rule is the same as above.

The results are shown in Figure~\ref{fig:simulation_different_AB}, with one panel for each value of $A=B$ and with $\kappa$ and $\varepsilon$ on the horizontal and vertical axes, respectively. The grayscale heatmap gives the classification for the averaged system \eqref{FHN_intro}: dark cells indicate tonic spiking and light cells indicate non-tonic spiking. Each colored curve is the transition boundary between tonic and non-tonic spiking for the full system at one carrier frequency $\omega_1$. The dashed red vertical line marks the theoretical $\kappa$ threshold predicted by the escaping condition \eqref{escaping_condition}; numerically, it is approximated by the first value of $\kappa$, on a grid with step $0.1$, for which the discretized escaping condition is satisfied.

For $A=B=0.15$, we observe no tonic spiking. This is consistent with Theorem~\ref{thm_geometric_condition}, which states that tonic spiking cannot occur for small $\varepsilon$ when $(A, B)\in \mathcal{E}_0$ satisfies the geometric condition~\eqref{eq:no-geom_condition}. For larger values of $A=B$, the agreement between the numerical classification and the theoretical threshold improves as $\varepsilon \to 0$, consistently with the singular limit analysis of Section~\ref{section_escaping_condition}: the heatmap for the averaged system reproduces the parameter-space organization predicted by the singular analysis as $\varepsilon$ decreases. Moreover, the threshold value of $\kappa$ given by the escaping condition decreases as $A=B$ increases, and the same trend is observed in the simulations. Finally, the interfaces computed for the full system approach the classification of the averaged system as the carrier frequency $\omega_1$ increases. The remaining discrepancies are concentrated near the activation boundary, where the finite-time classification is most sensitive to the carrier frequency and to the spike-detection rule.

\subsection{Experiment 2} 

Our second experiment explores how tonic spiking of the averaged system \eqref{FHN_intro} with $\eta = \kappa \varepsilon$ depends on the initial condition, and compares the results with the singular-limit analysis of Section~\ref{section_escaping_condition}. We fix $A = B = 0.3$ and consider two parameter sets: $(\gamma, \beta) = (0.5, 0.8)$ for cases (a), (b), and (c), and $(\gamma, \beta) = (0.6, 0.7)$ for cases (d), (e), and (f). For each parameter set, we take $\kappa \in \{1, 2, 2.5\}$ and $\varepsilon \in \{0.02, 0.04, 0.06, 0.1\}$, which gives $24$ parameter combinations. For each combination, the initial conditions $(v_0, w_0)$ range over an evenly spaced grid of the square $[-2,2]\times[-2,2]$ with $21^2 = 441$ points.

Each trajectory is simulated over the time interval $[0,T]$ with $T = 1000$, and we count its action potentials as in Experiment~1, that is, as the excursions of $v$ into the region $v \geq 0$ (no smoothing is needed for the averaged system). Since $\eta \geq 0.02$ for all parameter combinations, this simulation time covers at least three envelope periods $2\pi/\eta$, as in Experiment~1. Since, depending on the initial condition, a trajectory may perform a single action potential before settling, a single action potential is not regarded as evidence of tonic spiking.

The results are shown in Figure~\ref{fig:simulations_different_kappa_ep}. Each row corresponds to one case: rows (a)--(c) correspond to the first parameter set with $\kappa = 1, 2, 2.5$, respectively, and rows (d)--(f) to the second parameter set with the same values of $\kappa$. The first column shows the singular-limit test based on the results of Sections~\ref{sec-nontonic} and~\ref{sec:heuristics}: in cases (a) and (d), the criterion of Section~\ref{sec-nontonic} predicts non-tonic spiking, whereas in cases (b), (c), (e), and (f), the heuristic of Section~\ref{sec:heuristics} predicts tonic spiking. Columns 2--5 correspond to $\varepsilon = 0.02, 0.04, 0.06, 0.1$, respectively, and show, as a function of the initial condition $(v_0, w_0)$, the number of action potentials observed.

In rows (a) and (d), at most one action potential is observed over the whole simulation for every initial condition, which is consistent with the absence of tonic spiking predicted by Theorem~\ref{theorem_results_singular_system}. In rows (b), (c), (e), and (f), the condition for tonic spiking of Section~\ref{sec:heuristics} is satisfied, so tonic spiking is expected for almost every initial condition when $\varepsilon$ is small. This is observed in rows (b) and (c) for the smaller values of $\varepsilon$, while for larger values of $\varepsilon$ the trajectories transition to non-tonic spiking, consistently with the small-$\varepsilon$ character of the prediction. In rows (e) and (f), tonic spiking is observed for all values of $\varepsilon$ tested.

\begin{figure}
\hspace{1cm}
\begin{tikzpicture}
  \node[anchor=south west] (fig8legend) at (0,0) {\includegraphics[width=0.9\linewidth]{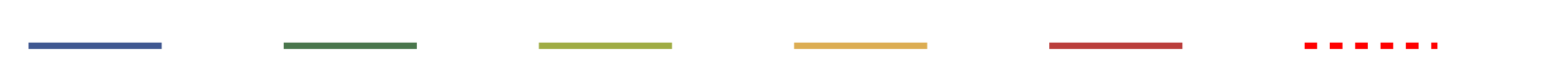}}; \begin{scope}[ shift={(fig8legend.south west)}, x={(fig8legend.south east)}, y={(fig8legend.north west)} ] 
    \node at (0.07 ,0.10) {$\omega_1=2$}; 
    \node at (0.23, 0.10) {$\omega_1=3$}; 
    \node at (0.39 ,0.10) {$\omega_1=4$}; 
    \node at (0.55,0.10) {$\omega_1=10$}; 
    \node at (0.71 ,0.10) {$\omega_1=20$}; 
    \node at (0.88,0.10) {$\kappa$ threshold}; 
  \end{scope}
\end{tikzpicture}

  \begin{tikzpicture}[x=1pt,y=1pt]

\node[anchor=south west] (fig8a) at (0,110)
    {\includegraphics[width=0.22\linewidth]{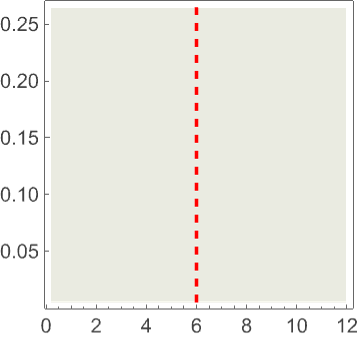}};
      \begin{scope}[
        shift={(fig8a.south west)},
        x={(fig8a.south east)}, 
        y={(fig8a.north west)}]
        \node[
        fill=white, 
        fill opacity=0.0, 
        text opacity=1, 
        draw=none, 
        thick,
        rounded corners,
        inner sep=1pt
      ]  at (0.74,0.72) {\tiny $A\!=\!B\!=\!0.15$};
      \end{scope}
\node[anchor=south west] (fig8b) at (0.23\linewidth,110)
    {\includegraphics[width=0.22\linewidth]{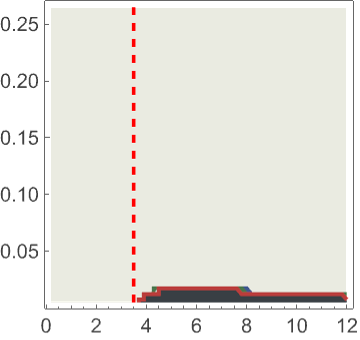}};
      \begin{scope}[
        shift={(fig8b.south west)},
        x={(fig8b.south east)}, 
        y={(fig8b.north west)}]
        \node[
        fill=white, 
        fill opacity=0.0, 
        text opacity=1, 
        draw=none, 
        thick,
        rounded corners,
        inner sep=1pt
      ]  at (0.74,0.72) {\tiny $A\!=\!B\!=\!0.2$};
      \end{scope}
\node[anchor=south west] (fig8c) at (0.46\linewidth,110)
    {\includegraphics[width=0.22\linewidth]{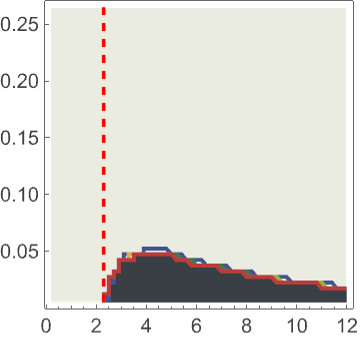}};
      \begin{scope}[
        shift={(fig8c.south west)},
        x={(fig8c.south east)}, 
        y={(fig8c.north west)}]
        \node[
        fill=white, 
        fill opacity=0.0, 
        text opacity=1, 
        draw=none, 
        thick,
        rounded corners,
        inner sep=1pt
      ]  at (0.74,0.72) {\tiny $A\!=\!B\!=\!0.25$};
      \end{scope}
\node[anchor=south west] (fig8d) at (0.69\linewidth,110)
    {\includegraphics[width=0.22\linewidth]{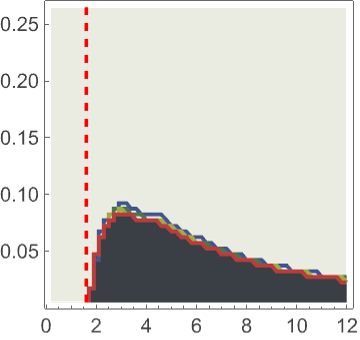}};
      \begin{scope}[
        shift={(fig8d.south west)},
        x={(fig8d.south east)}, 
        y={(fig8d.north west)}]
        \node[
        fill=white, 
        fill opacity=0.0, 
        text opacity=1, 
        draw=none, 
        thick,
        rounded corners,
        inner sep=1pt
      ]  at (0.74,0.72) {\tiny $A\!=\!B\!=\!0.3$};
      \end{scope}

\node[anchor=south west] (fig8e) at (0,5)
    {\includegraphics[width=0.22\linewidth]{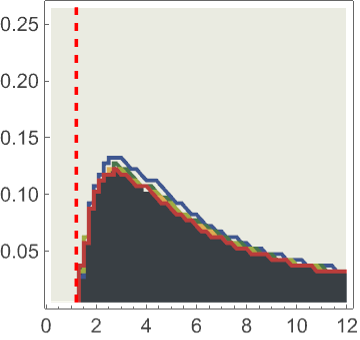}};
\begin{scope}[
        shift={(fig8e.south west)},
        x={(fig8e.south east)}, 
        y={(fig8e.north west)}]
        \node[
        fill=white, 
        fill opacity=0.0, 
        text opacity=1, 
        draw=none, 
        thick,
        rounded corners,
        inner sep=1pt
      ]  at (0.74,0.72) {\tiny $A\!=\!B\!=\!0.35$};
      \end{scope}
\node[anchor=south west] (fig8f) at (0.23\linewidth,5)
    {\includegraphics[width=0.22\linewidth]{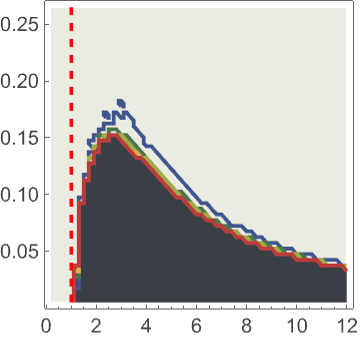}};
\begin{scope}[
        shift={(fig8f.south west)},
        x={(fig8f.south east)}, 
        y={(fig8f.north west)}]
        \node[
        fill=white, 
        fill opacity=0.0, 
        text opacity=1, 
        draw=none, 
        thick,
        rounded corners,
        inner sep=1pt
      ]  at (0.74,0.72) {\tiny $A\!=\!B\!=\!0.4$};
      \end{scope}
\node[anchor=south west] (fig8g)at (0.46\linewidth,5)
    {\includegraphics[width=0.22\linewidth]{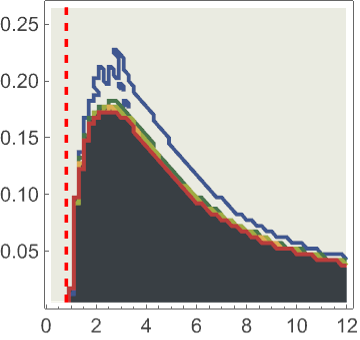}};
\begin{scope}[
        shift={(fig8g.south west)},
        x={(fig8g.south east)}, 
        y={(fig8g.north west)}]
        \node[
        fill=white, 
        fill opacity=0.0, 
        text opacity=1, 
        draw=none, 
        thick,
        rounded corners,
        inner sep=1pt
      ]  at (0.74,0.72) {\tiny $A\!=\!B\!=\!0.45$};
      \end{scope}
\node[anchor=south west] (fig8h)at (0.69\linewidth,5)
    {\includegraphics[width=0.22\linewidth]{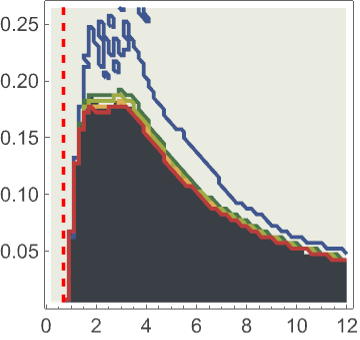}};
\begin{scope}[
        shift={(fig8h.south west)},
        x={(fig8h.south east)}, 
        y={(fig8h.north west)}]
        \node[
        fill=white, 
        fill opacity=0.0, 
        text opacity=1, 
        draw=none, 
        thick,
        rounded corners,
        inner sep=1pt
      ]  at (0.74,0.72) {\tiny $A\!=\!B\!=\!0.5$};
      \end{scope}

\node at (210,0) {\Large $\kappa$ axis};
\node[rotate=90] at (-10,100) {\Large $\varepsilon$ axis};

\end{tikzpicture}

  \caption{Results for Experiment 1: Numerical activation classification for $\beta=0.8$, $\gamma=0.5$, $\eta=\kappa\varepsilon$, and the indicated values of $A=B$. The grayscale heatmap is obtained from the averaged system: dark cells denote parameter pairs for which at least two action potentials are detected over the simulation interval $[0,T]$. The colored curves are empirical transition interfaces for the full two-frequency system at $\omega_1\in\{2,3,4,10,20\}$. The dashed red line is the theoretical threshold associated with the escaping condition \eqref{escaping_condition}. For the experiments, the system is initialized at $(0, w_e(1))$ corresponding to the specific pair $(A,B)$. 
  }
  \label{fig:simulation_different_AB}
\end{figure}

\begin{figure}
\centering
  \setlength{\fboxsep}{0pt}%
  \begin{tikzpicture}[x=1pt,y=1pt]

    \node[anchor=south west] (fig9) at (0,0)
    {\includegraphics[width=0.83\linewidth]{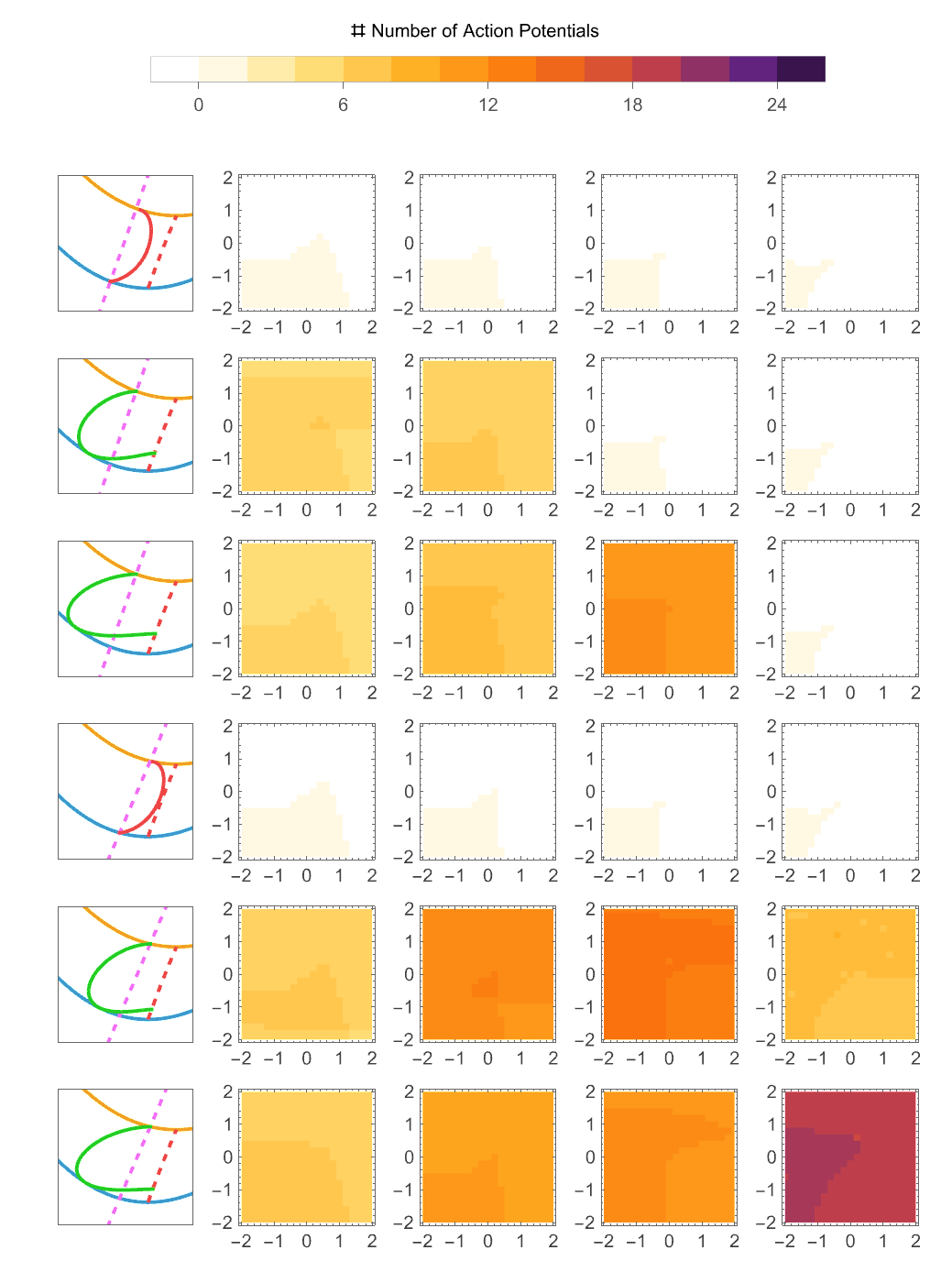}};
      \begin{scope}[
        shift={(fig9.south west)},
        x={(fig9.south east)}, 
        y={(fig9.north west)}]
        \node at (0.18,0.84) {\small $C_1$};
        \node at (0.18,0.76) {\small $C_{-1}$};
        \node at (0.10,0.77) {\small $\Lambda$};
        \node at (0.192,0.795) {\small $J_m$};
        
        \node[
          fill=white, fill opacity=0.0, text opacity=1,  
          draw=none, thick, rounded corners, inner sep=1pt,
          align=left
          ]  at (0.307,0.83) {\small $\kappa\!=\!1$\\
          \small $\varepsilon\!=\!0.02$};
        \node[
          fill=white, fill opacity=0.0, text opacity=1, 
          draw=none, thick, rounded corners, inner sep=1pt,
          align=left
          ]  at (0.493,0.83) {\small $\kappa\!=\!1$\\
          \small $\varepsilon\!=\!0.04$};
        \node[
          fill=white, fill opacity=0.0, text opacity=1, 
          draw=none, thick, rounded corners, inner sep=1pt,
          align=left
          ]  at (0.68,0.83) {\small $\kappa\!=\!1$\\
          \small $\varepsilon\!=\!0.06$};
        \node[
          fill=white, fill opacity=0.0, text opacity=1, 
          draw=none, thick, rounded corners, inner sep=1pt,
          align=left
          ]  at (0.86,0.83) {\small $\kappa\!=\!1$\\
          \small $\varepsilon\!=\!0.1$};
        \node at (0.18,0.7) {\small $C_1$};
        \node at (0.18,0.62) {\small $C_{-1}$};
        \node at (0.10,0.63) {\small $\Lambda$};
        \node at (0.192,0.655) {\small $J_m$};
        \node[
          fill=white, fill opacity=0.0, text opacity=1, 
          draw=none, thick, rounded corners, inner sep=1pt,
          align=left
          ]  at (0.307,0.69) {\small $\kappa\!=\!2$\\
          \small $\varepsilon\!=\!0.02$};
        \node[
          fill=white, fill opacity=0.0, text opacity=1, 
          draw=none, thick, rounded corners, inner sep=1pt,
          align=left
          ]  at (0.493,0.69) {\small $\kappa\!=\!2$\\
          \small $\varepsilon\!=\!0.04$};
        \node[
          fill=white, fill opacity=0.0, text opacity=1, 
          draw=none, thick, rounded corners, inner sep=1pt,
          align=left
          ]  at (0.68,0.69) {\small $\kappa\!=\!2$\\
          \small $\varepsilon\!=\!0.06$};
        \node[
          fill=white, fill opacity=0.0, text opacity=1, 
          draw=none, thick, rounded corners, inner sep=1pt,
          align=left
          ]  at (0.86,0.69) {\small $\kappa\!=\!2$\\
          \small $\varepsilon\!=\!0.1$};
        \node at (0.18,0.56) {\small $C_1$};
        \node at (0.18,0.48) {\small $C_{-1}$};
        \node at (0.10,0.49) {\small $\Lambda$};
        \node at (0.192,0.515) {\small $J_m$};
        \node[
          fill=white, fill opacity=0.0, text opacity=1, 
          draw=none, thick, rounded corners, inner sep=1pt,
          align=left
          ]  at (0.307,0.55) {\small $\kappa\!=\!2.5$\\
          \small $\varepsilon\!=\!0.02$};
        \node[
          fill=white, fill opacity=0.0, text opacity=1, 
          draw=none, thick, rounded corners, inner sep=1pt,
          align=left
          ]  at (0.493,0.55) {\small $\kappa\!=\!2.5$\\
          \small $\varepsilon\!=\!0.04$};
        \node[
          fill=white, fill opacity=0.0, text opacity=1, 
          draw=none, thick, rounded corners, inner sep=1pt,
          align=left
          ]  at (0.68,0.55) {\small $\kappa\!=\!2.5$\\
          \small $\varepsilon\!=\!0.06$};
        \node[
          fill=white, fill opacity=0.0, text opacity=1,  
          draw=none, thick, rounded corners, inner sep=1pt,
          align=left
          ]  at (0.86,0.55) {\small $\kappa\!=\!2.5$\\
          \small $\varepsilon\!=\!0.1$};
        \node at (0.19,0.42) {\small $C_1$};
        \node at (0.18,0.34) {\small $C_{-1}$};
        \node at (0.11,0.345) {\small $\Lambda$};
        \node at (0.192,0.375) {\small $J_m$};
        \node[
          fill=white, fill opacity=0.0, text opacity=1, 
          draw=none, thick, rounded corners, inner sep=1pt,
          align=left
          ]  at (0.307,0.41) {\small $\kappa\!=\!1$\\
          \small $\varepsilon\!=\!0.02$};
        \node[
          fill=white, fill opacity=0.0, text opacity=1,  
          draw=none, thick, rounded corners, inner sep=1pt,
          align=left
          ]  at (0.493,0.41) {\small $\kappa\!=\!1$\\
          \small $\varepsilon\!=\!0.04$};
        \node[
          fill=white, fill opacity=0.0, text opacity=1,  
          draw=none, thick, rounded corners, inner sep=1pt,
          align=left
          ]  at (0.68,0.41) {\small $\kappa\!=\!1$\\
          \small $\varepsilon\!=\!0.06$};
        \node[
          fill=white, fill opacity=0.0, text opacity=1,  
          draw=none, thick, rounded corners, inner sep=1pt,
          align=left
          ]  at (0.86,0.41) {\small $\kappa\!=\!1$\\
          \small $\varepsilon\!=\!0.1$};
        \node at (0.19,0.28) {\small $C_1$};
        \node at (0.18,0.2) {\small $C_{-1}$};
        \node at (0.11,0.205) {\small $\Lambda$};
        \node at (0.192,0.235) {\small $J_m$};
        \node[
          fill=white, fill opacity=0.0, text opacity=1,  
          draw=none, thick, rounded corners, inner sep=1pt,
          align=left
          ]  at (0.307,0.27) {\small $\kappa\!=\!2$\\
          \small $\varepsilon\!=\!0.02$};
        \node[
          fill=white, fill opacity=0.0, text opacity=1,  
          draw=none, thick, rounded corners, inner sep=1pt,
          align=left
          ]  at (0.493,0.27) {\small $\kappa\!=\!2$\\
          \small $\varepsilon\!=\!0.04$};
        \node[
          fill=white, fill opacity=0.0, text opacity=1,  
          draw=none, thick, rounded corners, inner sep=1pt,
          align=left
          ]  at (0.68,0.27) {\small $\kappa\!=\!2$\\
          \small $\varepsilon\!=\!0.06$};
        \node[
          fill=white, fill opacity=0.0, text opacity=1,  
          draw=none, thick, rounded corners, inner sep=1pt,
          align=left
          ]  at (0.86,0.27) {\small $\kappa\!=\!2$\\
          \small $\varepsilon\!=\!0.1$};    
        \node at (0.19,0.14) {\small $C_1$};
        \node at (0.18,0.06) {\small $C_{-1}$};
        \node at (0.11,0.065) {\small $\Lambda$};
        \node at (0.192,0.095) {\small $J_m$};
        \node[
          fill=white, fill opacity=0.0, text opacity=1,  
          draw=none, thick, rounded corners, inner sep=1pt,
          align=left
          ]  at (0.307,0.13) {\small $\kappa\!=\!2.5$\\
          \small $\varepsilon\!=\!0.02$};
        \node[
          fill=white, fill opacity=0.0, text opacity=1,  
          draw=none, thick, rounded corners, inner sep=1pt,
          align=left
          ]  at (0.493,0.13) {\small $\kappa\!=\!2.5$\\
          \small $\varepsilon\!=\!0.04$};
        \node[
          fill=white, fill opacity=0.0, text opacity=1,  
          draw=none, thick, rounded corners, inner sep=1pt,
          align=left
          ]  at (0.68,0.13) {\small $\kappa\!=\!2.5$\\
          \small $\varepsilon\!=\!0.06$};
        \node[
          fill=white, fill opacity=0.0, text opacity=1, 
          draw=none, thick, rounded corners, inner sep=1pt,
          align=left
          ]  at (0.86,0.13) {\small $\kappa\!=\!2.5$\\
          \small $\varepsilon\!=\!0.1$};    

        \draw[black, thick] (0.05,0.06) -- (0.05,0.43);
        \draw[black, thick] (0.05,0.48) -- (0.05,0.85);
        \node at (0.025, 0.83) {(a)}; 
        \node at (0.025, 0.69) {(b)}; 
        \node at (0.025, 0.55) {(c)}; 
        \node at (0.025, 0.41) {(d)}; 
        \node at (0.025, 0.27) {(e)}; 
        \node at (0.025, 0.13) {(f)}; 
        
        \node at (0.5,0.01) {\Large $v$ axis};
        \node[rotate=90] at (0.99,0.5) {\Large $w$ axis};        
      \end{scope}
  \end{tikzpicture}
   \caption{Results for Experiment 2 for the averaged system \eqref{FHN_intro}.  
   In the first column, we verify Assumption (E) graphically, and the condition for tonic spiking is satisfied for (b), (c), (e), and (f). Columns 2-5 contain heatmaps in which the number of spikes is counted starting from a specific initial condition. The parameters used are (a)-(c)  $A=B=0.3$, $\gamma=0.5$, $\beta= 0.8$; (d)-(f) $A=B=0.3$, $\gamma=0.6$, $\beta= 0.7$. Additionally, we consider three different values of 
   $\kappa$ ($\kappa =1$ in    (a),(d);  $\kappa=2$ in (b),(e);  $\kappa= 2.5$ in (c),(f)). In each case, 
   four values of $\varepsilon$ are considered:  $0.02, 0.04, 0.06, 0.1$.
   }\label{fig:simulations_different_kappa_ep}
\end{figure}

\section{Discussion and perspectives}\label{sec:discussion}

In this work, we studied a model of temporal interference stimulation and explored how the source parameters, i.e., amplitudes and beat frequency, must be chosen to induce tonic spiking. Using techniques based on geometric arguments, singular perturbations, and analysis of switching systems, we develop characterizations of tonic spiking, in particular when $\varepsilon$ is small. Our main contribution is a concrete method for testing for tonic spiking based on the algorithmic construction of geometric objects.

The first element of our characterization is the identification of a geometric obstruction. 
Such a condition is shown to be sharp in the sense that for certain piecewise constant signals, the reversed condition implies tonic spiking (Theorem~\ref{thm_activation_piecewise_constant}).

The second part of our construction studied the regime in which the beat frequency is chosen to be proportional to the small parameter $\varepsilon$, which allows us to characterize tonic activation by studying a `hybrid system' obtained as a singular limit under the scaling $\eta = \kappa \varepsilon$. The resulting system allows us to formulate concrete criteria for characterizing tonic spiking (Section \ref{section_escaping_condition}). This analysis suggests that a suitable choice for the system's beat frequency is related to the time constant of the neuron's ionic channels. 

Our positive criteria for tonic spiking (Section \ref{sec:heuristics}) predict tonic spiking for most initial conditions and $\varepsilon$ small, and this is confirmed by the numerical experiments (Section \ref{sec:numerical_experiments}). While the criterion implies a typical spiking behavior, the existence of exceptional trajectories may lead to the presence of special initial conditions from which tonic spiking disappears  (Section \ref{sec:heuristics}). The existence of such solutions is inherent to the periodicity of the mathematical setting (Lemma~\ref{lemma_general_properties}, item \ref{lemma_preliminary_item1}). 

Our criteria for tonic spiking are based on properties of a particular solution of an ODE; thus, continuous dependence gives us a notion of robustness with regard to the parameters. Additionally, observe that the arguments in Section \ref{section_escaping_condition} are based on studying intervals where $f$ is monotone, and not the specific formula of the signal; thus, we have a notion of robustness with respect to the shape of the periodic signal $f$.

A limitation of our analysis of tonic spiking is that it relies on local solution properties, which limits its ability to yield a global answer to the problem (beyond the heuristic). We expect that a complete characterization must account for synchronization to explain the convergence of solutions to the non-autonomous system to periodic trajectories. An important aspect that makes this analysis more difficult is the presence of exceptional $2\pi/\eta$ periodic trajectories of the system (Lemma \ref{lemma_general_properties} item \ref{lemma_preliminary_item1}), which, depending on $\varepsilon>0$, might be comparable to or shorter than the typical length of an action potential, and whose amplitude might be smaller than a predefined activation threshold.

\section*{Acknowledgments}
The first author has been partially supported by the ANID Millennium Science Initiative Program through Millennium Nucleus for Applied Control and Inverse Problems NCN19-161, and the grant ANID Fondecyt postdoctorado 3240512.

\section*{Data Availability}
All code used for the simulations and the construction of the figures is available in the GitHub repository \url{https://github.com/estebanpaduro/fhn_activation_tis}

\appendix
\section{Proof of Lemma \ref{properties_equilibrium}}\label{sec-properties}
Any equilibrium $(v,w)$ of \eqref{eq-1} satisfies 
\begin{equation*}
\left\{\begin{array}{rl}
    0 & = r(c)v - \frac{v^3}{3} - w,\\
    0 &= v-\gamma w + \beta.
\end{array}\right.
\end{equation*}
Solving for $w$ and substituting we get that $v$ solves the cubic equation $f(v)=v^3- 3(r(c) -\frac{1}{\gamma}) v + \frac{3\beta}{\gamma} =0$. 
Since $f(0)>0$, $\lim_{v\to -\infty} f(v) = -\infty$,  and $f'$ changes sign at most once in $(-\infty,0)$, it follows that there exists a unique $v=v_e(c)$ in $(-\infty,0)$ for which $f(v)=0$. 

To prove item \ref{lemma_item_uniqueness}, recall that the depressed cubic $z^3 + pz + q = 0$ has a unique real root if and only if $4p^3+27q^2 >0$. In our case, this condition becomes $\beta^2/\gamma^2 > \frac{4}{9}(r(c)-1/\gamma)^3$. 

To prove item~\ref{lemma_item_iv}, let us linearize \eqref{eq-1} near the equilibrium $(v_e(c),w_e(c))$. To that aim, we translate the equilibrium to the origin by considering the state variables
\begin{equation*}
    \tilde v:=v-v_e(c),\quad \tilde w:=w-w_e(c).
\end{equation*}
 Then system \eqref{eq-1} reads
\begin{subequations}\label{eq-centered}
\begin{align}
    \dot{\tilde v} &= (r(c) - v_e(c)^2) \tilde{v} - v_e(c) \tilde{v}^2 - \frac{1}{3} \tilde{v}^3 - \tilde{w},\\
    \dot{\tilde w} &=\varepsilon(\tilde{v} - \gamma \tilde{w}).
\end{align}
\end{subequations}
The linearization of this system at the origin is 
\begin{equation*}
    \frac{d}{dt} \binom{V}{W} = \left(\begin{array}{cc}
         r(c)-v_e(c)^2& -1 \\
         \varepsilon & - \varepsilon \gamma 
    \end{array}\right)\binom{V}{W} =: J \binom{V}{W}.
\end{equation*}
The local exponential stability of $(v_e(c),w_e(c))$ is equivalent to the conditions $\text{tr}(J)<0$ and $\det(J)>0$, which can be combined into $r(c)-v_e(c)^2 < \min\{ \varepsilon \gamma, 1/\gamma\}$, thus establishing item~\ref{lemma_item_iv}. 

To prove item \ref{lemma_item_global}, we first need a result regarding the size of the basin of attraction of the equilibrium $(v_e(c),w_e(c))$. To that aim, consider the Lyapunov function
$$L(v,w) = \frac{1}{2}\tilde{v}^2+\frac{1}{2\varepsilon} \tilde{w}^2.$$
Taking its derivative along the trajectories of the recentered system \eqref{eq-centered} leads to 
\begin{align*}
\frac{d}{dt}L(\tilde v,\tilde w)&=\tilde{v}\left( (r(c)-v_e(c)^2)\tilde{v} - v_e(c)\tilde{v}^2 - \frac{1}{3}\tilde{v}^3-\tilde{w} \right) + \tilde{w}\left(  \tilde{v} - \gamma \tilde{w}\right)\\
&= (r(c) - v_e(c)^2)\tilde{v}^2 - v_e(c) \tilde{v}^3 - \frac{1}{3}\tilde{v}^4 - \gamma \tilde{w}^2\\
&\leq \left(r(c)-v_e(c)^2 -v_e(c)\tilde{v}\right)\tilde{v}^2 - \gamma \tilde{w}^2.
\end{align*}
Therefore, if $M$ is such that
\begin{equation}\label{eq:dddd}
r(c)-v_e(c)^2 -v_e(c)\tilde{v}<0
\end{equation}
on the set where $L(\tilde v,\tilde w) \leq M$, then $L$ is a Lyapunov function in such a sublevel set of $L$. By assumption $r(c)<v_e(c)^2$. Recalling that $v_e(c)<0$, \eqref{eq:dddd} can be rewritten as 
\begin{align*}
\tilde{v}<\frac{-(r(c)-v_e(c)^2)}{-v_e(c)}
 = \frac{v_e(c)^2-r(c)}{|v_e(c)|}.
\end{align*}
Setting $M:=(v_e(c)^2-r(c))^2/(2|v_e(c)|^2)$, we deduce that 
the sublevel set $\frac{1}{2}\tilde{v}^2+\frac{1}{2\varepsilon} \tilde{w}^2 \leq M$
is contained in the basin of attraction of the origin. In particular, the basin of attraction of the equilibrium $(v_e(c),w_e(c))$ contains the horizontal segment
\[\Omega := \left[v_e(c) - \sqrt{2M} , v_e(c) + \sqrt{2M}\right]\times \{w_e(c)\}, \]
which does not depend on $\varepsilon>0$.

Second, according to Fact~\ref{fact1}, we can pick $L>0$ large enough that the region $R = [-L,L]\times [-S,S]$,  with $S:=\frac{L+\beta}{\gamma}+1$, is positively invariant and globally exponentially attractive for system \eqref{eq-1}. It is clear that $R$ contains the equilibrium $(v_e(c),w_e(c))$ (otherwise, global convergence to $R$ would be impeded by an equilibrium outside). Consider the arc $\Gamma := \{(v,w)\in C_c\cap R\mid  v\leq v_e(c)\}$. We claim that any trajectory of  system~\eqref{eq-1} initially in $R$ either reaches $\Gamma$ in finite time or converges to $(v_e(c), w_e(c))$ as $t\to +\infty$. To see this, consider any trajectory  $(v,w)$ of  system~\eqref{eq-1} with $(v(0), w(0))\in R$. By forward invariance of $R$, we can extend the solution for all positive times. Compactness of $R$ implies that as $t\to \infty$, the trajectory must converge to an $\omega$-limit set. Since the system is bidimensional, the Poincar\'e--Bendixson theorem ensures that such an $\omega$-limit set can only be made of stationary points or periodic orbits. In the first case, since the equilibrium is unique by assumption, the limit must be $(v_e(c), w_e(c))$, and we are done. In the second case, the periodic orbit is a closed curve, which bounds an open region that must contain the equilibrium $(v_e(c),w_e(c))$ of the system \cite[Chapter 10]{hirsch_differential_2013}. Since any trajectory in $R$ that encircles the equilibrium must cross $\Gamma$ in finite time, the claim is proved.

Third, for $0<\delta < \sqrt{2M}$, consider the region $H$ above $\Omega$ delimited by $\Gamma$, $T_{-\delta} \Gamma$ (the left translation of $\Gamma$ by $\delta$), $\partial R$, and the segment $\Omega$. Observe that $\dot{w}<0$ in $H$ since it is to the left of $\Lambda$. On $\partial R$, the vector field $G_c$ points inwards $R$. On $\Gamma$,  $G_c$ points to the left, so it points inwards $H$. On $T_{-\delta} \Gamma$, none of the components $g_1$ and $\varepsilon g_2$ of $G_c$ vanish. 
Let $K:=\min_{(x,y)\in T_{-\delta} \Gamma} \frac{g_2(x,y)}{g_1(x,y,c)} <0$, so that $\varepsilon K$ is the most negative slope of $G_c$ on $T_{-\delta}\Gamma$.  
Let also $s<0$ be the least negative slope of $\Gamma$.
Take $\varepsilon$ small so that $s < -\varepsilon K$. We deduce that $G_c$  is inward-pointing towards $H$ on $T_{-\delta} \Gamma$. Since  $H$ contains no stationary points in its interior, any trajectory of \eqref{eq-1} starting inside $H$ must leave $H$ in finite time by crossing $\Omega$ or converge to the equilibrium as $t \to +\infty$. Since, as proved above,  $\Omega$ belongs to the attraction basin of $(v_e(c),w_e(c))$, the trajectory must converge to said equilibrium in all cases.

Thus, we have shown that $R$ is exponentially attractive, that $(v_e(c),w_e(c))$ is locally exponentially stable, and that every trajectory starting in $R$ reaches the attraction basin of $(v_e(c),w_e(c))$ for $\varepsilon>0$ small enough. We can then use a compactness argument on $R$ to conclude that $(v_e(c),w_e(c))$ is globally exponentially stable for $\varepsilon>0$ small enough, which gives us item~\ref{lemma_item_global}. 
\qed

\section{Proof of Lemma~\ref{lemma_general_properties}}\label{appendix_proof_lemma_general_properties}

To prove item \ref{lemma_preliminary_item1}, consider the map $\phi: \R^2 \to \R^2$ that associates with a point $(v_0,w_0)\in \R^2$ the evaluation at time $T$ of the solution $(v,w)$ of \eqref{general_switching_system}. According to Fact \ref{fact1}, $\phi$ is well-defined and continuous, and for every $L>0$ large enough $\phi(R_L)\subset R_L$, where  $R_L = [-L,L]\times[-S,S]$ and $S=(L+\beta)/\gamma+1$. Brouwer's fixed point theorem then implies that the map $\phi$ has a fixed point inside $R_L$. Since the right-hand side of \eqref{general_switching_system} is $T$-periodic, a fixed point of $\phi$ corresponds to a periodic trajectory $(\bar v(\cdot),\bar w(\cdot))$ of the system. 
Such a periodic trajectory is not constant since $f$ is non-constant and $c\mapsto (v_e(c),w_e(c))$ is injective. 

Regarding the last part of the statement of item \ref{lemma_preliminary_item1}, notice that Fact \ref{fact1} implies that the support of $(\bar v(\cdot),\bar w(\cdot))$ is contained in $R_L$. 
In particular, $(\bar v(\cdot),\bar w(\cdot))$ is $K$-Lipschitz continuous, with $K$ independent of $T$. 
Moreover, considering the functions $g_1$ and $g_2$ defined in \eqref{defi_F1}-\eqref{defi_F2}, the support of $(\bar v(\cdot),\bar w(\cdot))$
must intersect at least three of the four regions
\[\{(v,w)\mid \exists\, c\in [-1,1]\mbox{ s.t. } \sigma_1 g_1(v,w,c)\ge 0,\; \sigma_2 g_2(v,w)\ge 0\},\qquad \sigma_1,\sigma_2\in\{-1,1\},\]
since otherwise the trajectory could not be periodic. 
Hence, given any $\delta_1>0$, there exists $\delta_2>0$ such that either ${\rm dist}((\bar v(t),\bar w(t)),J_e)<\delta_1$ for some $t\in [0,T]$ (with $J_e$ defined in \eqref{eq-Je}) or the diameter of the support of $(\bar v(\cdot),\bar w(\cdot))$ is larger than $\delta_2$. 
Taking $\delta_1$ smaller than ${\rm dist}(J_e,\{v=0\})$ and $T$ smaller than $\min(\delta_2,{\rm dist}(J_e,\{v=0\})-\delta_1)/K$,  it follows that the support of $(\bar v(\cdot),\bar w(\cdot))$ is contained in $\{v<0\}$.

For item \ref{lemma_preliminary_item2}, one uses standard converse Lyapunov results to show that for every $c$ there exists a Lyapunov function for the dynamics $(\dot v,\dot w)=G_c(v,w)$ in a neighborhood of $(v_e(c),w_e(c))$ and that such a Lyapunov function can be chosen depending smoothly on $c$. Computing the derivative of this Lyapunov function along solutions of the (slowly) time-varying system, the result follows from basic manipulations (see, for instance, \cite[Lemma 3.2]{cerpaImpactHighFrequencyBasedStability2024}). 

We now move to establish item \ref{lemma_preliminary_item4}. Consider the frozen system \eqref{eq-1} with $c=\bar\varphi$. By assumption, its equilibrium $(v_e(\bar\varphi),w_e(\bar\varphi))$ is locally exponentially stable. The first part of the statement is then a direct consequence of the averaging theorem \cite[Theorem 10.4]{khalil2013} and of the local exponential stability of $(v_e(\bar \varphi),w_e(\bar \varphi))$. For the second part, we invoke \cite[Theorem 2]{teel1999semi}, which states that global asymptotic stability of the averaged system implies semiglobal practical stability for the original system. Since there exists a compact set in which our system evolves (Fact \ref{fact1}), solutions are eventually trapped in a neighborhood of the equilibrium, whose size can be made arbitrarily small by picking $\eta$ large enough.


\begin{thebibliography}{10}
\expandafter\ifx\csname url\endcsname\relax
  \def\url#1{\texttt{#1}}\fi
\expandafter\ifx\csname doi\endcsname\relax
  \def\doi#1{\burlalt{doi:#1}{http://dx.doi.org/#1}}\fi
\expandafter\ifx\csname urlprefix\endcsname\relax\def\urlprefix{URL }\fi
\expandafter\ifx\csname href\endcsname\relax
  \def\href#1#2{#2}\fi
\expandafter\ifx\csname burlalt\endcsname\relax
  \def\burlalt#1#2{\href{#2}{#1}}\fi

\bibitem{Blanchini1999}
F.~Blanchini.
\newblock Set invariance in control.
\newblock {\em Automatica J. IFAC}, 35(11):1747--1767, 1999.
\newblock \doi{10.1016/S0005-1098(99)00113-2}.

\bibitem{Cao2020b}
J.~Cao, B.~Doiron, C.~Goswami, and P.~Grover.
\newblock {The mechanics of temporal interference stimulation}.
\newblock {\em bioRxiv}, 2020.
\newblock \doi{10.1101/2020.04.23.051870}.

\bibitem{cerpaImpactHighFrequencyBasedStability2024}
E.~Cerpa, N.~Corrales, M.~Courdurier, L.~E. Medina, and E.~Paduro.
\newblock The {{Impact}} of {{High-Frequency-Based Stability}} on the {{Onset}}
  of {{Action Potentials}} in {{Neuron Models}}.
\newblock {\em SIAM Journal on Applied Mathematics}, 84(5):1910--1936, 2024,
  \burlalt{2402.05886}{http://arxiv.org/abs/2402.05886}.
\newblock \doi{10.1137/24M1645632}.

\bibitem{cerpa2023}
E.~Cerpa, M.~Courdurier, E.~Hern\'{a}ndez, L.~E. Medina, and E.~Paduro.
\newblock A partially averaged system to model neuron responses to
  interferential current stimulation.
\newblock {\em J. Math. Biol.}, 86(1):8, 2023.
\newblock \doi{10.1007/s00285-022-01839-8}.

\bibitem{cerpaApproximationStabilityResults2025}
E.~Cerpa, M.~Courdurier, E.~Hernández, L.~E. Medina, and E.~Paduro.
\newblock Approximation and stability results for the parabolic
  {FitzHugh}-{Nagumo} system with combined rapidly oscillating sources.
\newblock {\em Discrete and Continuous Dynamical Systems}, 50:69--102, 2026,
  \burlalt{2305.00123}{http://arxiv.org/abs/2305.00123}.
\newblock \doi{10.3934/dcds.2026001}.

\bibitem{Fenichel1979}
N.~Fenichel.
\newblock Geometric singular perturbation theory for ordinary differential
  equations.
\newblock {\em J. Differential Equations}, 31(1):53--98, 1979.
\newblock \doi{10.1016/0022-0396(79)90152-9}.

\bibitem{FitzHugh1961}
R.~FitzHugh.
\newblock {Impulses and Physiological States in Theoretical Models of Nerve
  Membrane}.
\newblock {\em Biophysical Journal}, 1(6):445--466, 1961.
\newblock \doi{10.1016/S0006-3495(61)86902-6}.

\bibitem{Grossman2017}
N.~Grossman, D.~Bono, N.~Dedic, S.~B. Kodandaramaiah, A.~Rudenko, H.-J. Suk,
  A.~M. Cassara, E.~Neufeld, N.~Kuster, L.-H. Tsai, A.~Pascual-Leone, and E.~S.
  Boyden.
\newblock {Noninvasive Deep Brain Stimulation via Temporally Interfering
  Electric Fields}.
\newblock {\em Cell}, 169(6):1029--1041.e16, 2017.
\newblock \doi{10.1016/j.cell.2017.05.024}.

\bibitem{hayashi_global_1999}
M.~Hayashi.
\newblock Global {{Asymptotic Stability}} of {{FitzHugh-Nagumo System}}.
\newblock In {\em Proceedings of the {{Ninth International Colloquium}} on
  {{Differential Equations}}}, pages 191--196. De Gruyter, 1999.
\newblock \doi{10.1515/9783112318973-029}.

\bibitem{hirsch_differential_2013}
M.~W. Hirsch, S.~Smale, and R.~L. Devaney.
\newblock {\em Differential equations, dynamical systems, and an introduction
  to chaos}.
\newblock Academic Press, Waltham, MA, 3rd edition, 2013.

\bibitem{Jones95}
C.~K. R.~T. Jones.
\newblock Geometric singular perturbation theory.
\newblock In {\em Dynamical systems ({M}ontecatini {T}erme, 1994)}, volume 1609
  of {\em Lecture Notes in Math.}, pages 44--118. Springer, Berlin, 1995.
\newblock \doi{10.1007/BFb0095239}.

\bibitem{Karimi2019}
F.~Karimi, A.~Attarpour, R.~Amirfattahi, and A.~Z. Nezhad.
\newblock {Computational analysis of non-invasive deep brain stimulation based
  on interfering electric fields}.
\newblock {\em Physics in Medicine \& Biology}, 64(23):235010, 2019.
\newblock \doi{10.1088/1361-6560/ab5229}.

\bibitem{karimiNeuromodulationEffectTemporal2024}
N.~Karimi, R.~Amirfattahi, and A.~Zeidaabadi~Nezhad.
\newblock Neuromodulation effect of temporal interference stimulation based on
  network computational model.
\newblock {\em Frontiers in Human Neuroscience}, 18:1436205, 2024.
\newblock \doi{10.3389/fnhum.2024.1436205}.

\bibitem{kaumann_uniqueness_1983}
E.~Kaumann and U.~Staude.
\newblock Uniqueness and nonexistence of limit cycles for the {F}itz{H}ugh
  equation.
\newblock In {\em Equadiff 82 ({W}\"{u}rzburg, 1982)}, volume 1017 of {\em
  Lecture Notes in Math.}, pages 313--321. Springer, Berlin, 1983.
\newblock \doi{10.1007/BFb0103262}.

\bibitem{khalil2013}
H.~K. Khalil.
\newblock {\em Nonlinear Systems}.
\newblock Prentice Hall, Upper Saddle River, NJ, third edition, 2002.

\bibitem{kostovaFITZHUGHNAGUMOREVISITED2004}
T.~Kostova, R.~Ravindran, and M.~Schonbek.
\newblock Fitz{H}ugh-{N}agumo revisited: types of bifurcations, periodical
  forcing and stability regions by a {L}yapunov functional.
\newblock {\em Internat. J. Bifur. Chaos Appl. Sci. Engrg.}, 14(3):913--925,
  2004.
\newblock \doi{10.1142/S0218127404009685}.

\bibitem{Mirzakhalili2020}
E.~Mirzakhalili, B.~Barra, M.~Capogrosso, and S.~F. Lempka.
\newblock {Biophysics of Temporal Interference Stimulation}.
\newblock {\em Cell Systems}, 11(6):557--572.e5, 2020.
\newblock \doi{10.1016/j.cels.2020.10.004}.

\bibitem{missey2025non}
F.~Missey, E.~Acerbo, A.~S. Dickey, J.~Trajlinek, O.~Studni{\v{c}}ka,
  C.~Lubrano, M.~de~Ara{\'u}jo~e Silva, E.~Brady, V.~V{\v{s}}iansk{\'y},
  J.~Szabo, et~al.
\newblock Non-invasive temporal interference stimulation of the hippocampus
  suppresses epileptic biomarkers in patients with epilepsy: Biophysical
  differences between kilohertz and amplitude modulated stimulation.
\newblock {\em Brain Stimulation}, 19:102981, 2026.
\newblock \doi{10.1016/j.brs.2025.11.008}.

\bibitem{Missey2021}
F.~Missey, E.~Rusina, E.~Acerbo, B.~Botzanowski, A.~Tr{\'{e}}buchon,
  F.~Bartolomei, V.~Jirsa, R.~Carron, and A.~Williamson.
\newblock {Orientation of Temporal Interference for Non-invasive Deep Brain
  Stimulation in Epilepsy}.
\newblock {\em Frontiers in Neuroscience}, 15:633988, 2021.
\newblock \doi{10.3389/fnins.2021.633988}.

\bibitem{nagumoActivePulseTransmission1962}
J.~Nagumo, S.~Arimoto, and S.~Yoshizawa.
\newblock An {Active} {Pulse} {Transmission} {Line} {Simulating} {Nerve}
  {Axon}.
\newblock {\em Proceedings of the IRE}, 50(10):2061--2070, 1962.
\newblock \doi{10.1109/JRPROC.1962.288235}.

\bibitem{Nagumo1942}
M.~Nagumo.
\newblock \"{U}ber die {L}age der {I}ntegralkurven gew\"{o}hnlicher
  {D}ifferentialgleichungen.
\newblock {\em Proc. Phys.-Math. Soc. Japan (3)}, 24:551--559, 1942.

\bibitem{opancarSameBiophysicalMechanism2025}
A.~Opan\v{c}ar, P.~Ondr\'a\v{c}kov\'a, D.~S. Rose, J.~Trajlinek,
  V.~{{\fontencoding{T1}\selectfont\DJ}}erek, and E.~D. G{\l}owacki.
\newblock The same biophysical mechanism is involved in both temporal
  interference and direct {kHz} stimulation of peripheral nerves.
\newblock {\em Nature Communications}, 16(1):9006, 2025.
\newblock \doi{10.1038/s41467-025-64059-w}.

\bibitem{plovieNonlinearitiesTimescalesNeural2025}
T.~Plovie, R.~Schoeters, T.~Tarnaud, W.~Joseph, and E.~Tanghe.
\newblock Nonlinearities and timescales in neural models of temporal
  interference stimulation.
\newblock {\em Bioelectromagnetics}, 46(1):e22522, 2025.
\newblock \doi{10.1002/bem.22522}.

\bibitem{sugie_nonexistence_1991}
J.~Sugie.
\newblock Nonexistence of periodic solutions for the {FitzHugh} nerve system.
\newblock {\em Quarterly of Applied Mathematics}, 49(3):543--554, 1991.
\newblock \doi{10.1090/qam/1121685}.

\bibitem{SW01}
P.~Szmolyan and M.~Wechselberger.
\newblock Canards in {$\mathbb{R}^3$}.
\newblock {\em J. Differential Equations}, 177(2):419--453, 2001.
\newblock \doi{10.1006/jdeq.2001.4001}.

\bibitem{teel1999semi}
A.~R. Teel, J.~Peuteman, and D.~Aeyels.
\newblock Semi-global practical asymptotic stability and averaging.
\newblock {\em Systems \& Control Letters}, 37(5):329--334, 1999.
\newblock \doi{10.1016/S0167-6911(99)00039-0}.

\bibitem{Tihonov1952}
A.~N. Tihonov.
\newblock Systems of differential equations containing small parameters in the
  derivatives.
\newblock {\em Mat. Sbornik N.S.}, 31(73):575--586, 1952.

\bibitem{Vasileva}
A.~B. Vasil'eva.
\newblock Asymptotic behaviour of solutions of certain problems for ordinary
  non-linear differential equations with a small parameter multiplying the
  highest derivatives.
\newblock {\em Uspehi Mat. Nauk}, 18(3(111)):15--86, 1963.
\newblock \doi{10.1070/RM1963v018n03ABEH001137}.

\bibitem{violante2023non}
I.~R. Violante, K.~Alania, A.~M. Cassar{\`a}, E.~Neufeld, E.~Acerbo, R.~Carron,
  A.~Williamson, D.~L. Kurtin, E.~Rhodes, A.~Hampshire, et~al.
\newblock Non-invasive temporal interference electrical stimulation of the
  human hippocampus.
\newblock {\em Nature Neuroscience}, 26(11):1994--2004, 2023.
\newblock \doi{10.1038/s41593-023-01456-8}.

\bibitem{wangResponsesModelCortical2023}
B.~Wang, A.~S. Aberra, W.~M. Grill, and A.~V. Peterchev.
\newblock Responses of {Model} {Cortical} {Neurons} to {Temporal}
  {Interference} {Stimulation} and {Related} {Transcranial} {Alternating}
  {Current} {Stimulation} {Modalities}.
\newblock {\em Journal of Neural Engineering}, 19(6):066047, 2022.
\newblock \doi{10.1088/1741-2552/acab30}.

\bibitem{Ward2009}
A.~R. Ward.
\newblock {Electrical stimulation using kilohertz-frequency alternating
  current}.
\newblock {\em Physical Therapy}, 89(2):181--190, 2009.
\newblock \doi{10.2522/ptj.20080060}.

\bibitem{Wec98}
M.~Wechselberger.
\newblock {\em Singularly perturbed folds and canards in $\mathbb{R}^3$}.
\newblock PhD thesis, Vienna University of Technology, Austria, 1998.

\bibitem{xu_precision_2025}
S.~Xu, H.~Cui, X.~Xiao, F.~Manshaii, G.~Hong, and J.~Chen.
\newblock Precision at {Deep} {Brain}: {Noninvasive} {Temporal} {Interference}
  {Stimulation}.
\newblock {\em ACS Nano}, 19(46):39589--39614, 2025.
\newblock \doi{10.1021/acsnano.5c15238}.

\bibitem{yang_transcranial_2025}
C.~Yang, Y.~Xu, X.~Feng, B.~Wang, Y.~Du, K.~Wang, J.~Lü, L.~Huang, Z.~Qian,
  Z.~Wang, N.~Chen, J.~Zhou, C.~Zhang, and Y.~Liu.
\newblock Transcranial {Temporal} {Interference} {Stimulation} of the {Right}
  {Globus} {Pallidus} in {Parkinson}'s {Disease}.
\newblock {\em Movement Disorders}, 40(6):1061--1069, 2025.
\newblock \doi{10.1002/mds.29967}.

\end{thebibliography}
\end{document}